\documentclass{article}
\usepackage{amsmath}
\usepackage{amsthm}
\usepackage{amsfonts}
\usepackage{amssymb}
\usepackage[all]{xy}
\usepackage{amscd}

\font\smallsc=cmcsc10
\font\smallsl=cmsl10

\newtheorem{Lem}{Lemma}[section]
\newtheorem{Prop}[Lem]{Proposition}
\newtheorem{Thm}[Lem]{Theorem}
\theoremstyle{definition}
\newtheorem{Def}[Lem]{Definition}
\newtheorem{Exa}[Lem]{Example}
\newtheorem{Rem}[Lem]{Remark}

\newcommand{\E}{\mathcal E}
\newcommand{\W}{\mathcal W}
\newcommand{\A}{\mathcal A}

\newcommand{\F}{\mathcal F}
\newcommand{\Z}{\mathcal Z}
\newcommand{\I}{\mathcal I}
\newcommand{\M}{\mathcal M}
\newcommand{\N}{\mathcal N}
\newcommand{\T}{\mathcal T}
\newcommand{\K}{\mathcal K}
\renewcommand{\S}{\mathcal S}
\newcommand{\R}{\mathcal R}
\renewcommand{\L}{\mathcal L}
\renewcommand{\O}{\mathcal O}
\newcommand{\X}{\mathcal X}
\newcommand{\C}{\mathcal C}
\newcommand{\Y}{\mathcal Y}
\newcommand{\col}{\colon}

\newcommand{\ra}{\rightarrow}
\newcommand{\ol}{\overline}
\newcommand{\ul}{\underline}
\newcommand{\ox}{\otimes}
\newcommand{\IP}{\mathbb P}
\newcommand{\IA}{\mathbb A}
\newcommand{\wh}{\widehat}
\newcommand{\J}{\ol J}
\newcommand{\w}{\bigomega}
\newcommand{\lra}{\longrightarrow}
\renewcommand{\:}{\colon}

\newcommand{\wt}{\widetilde}
\renewcommand{\P}{\mathcal P}

\def\cocoa{{\hbox{\rm C\kern-.13em o\kern-.07em C\kern-.13em o\kern-.15em A}}}

\begin{document}

\title{Degree-2 Abel maps for nodal curves}
\author{J.~Coelho, E.~Esteves and M.~Pacini}
\maketitle

\begin{abstract}
\noindent
We present numerical conditions for the existence of natural 
degree-2 Abel maps 
for any given nodal curve. \cocoa\ scripst were written and have so far verified 
the validity of the conditions for numerous curves.
\end{abstract}

\section{Introduction}

The theory of Abel maps for curves goes back to work by Abel in the 
early nineteenth century. It was Riemann though, in his seminal 1857 paper 
\cite{Ri}, 
that introduced the maps themselves; see \cite{KPic} for the history. 
In modern terms, given a (connected projective) curve $C$ defined over an 
algebraicaly closed field $K$, if $C$ is smooth then 
its degree-$d$ Abel map is a map $A^d_\L\:S^d(C)\to J_C$ from the $d$-th 
symmetric product of the curve to its Jacobian $J_C$, given by taking a sum 
of $d$ points of $C$, say $Q_1+\cdots+Q_d$, to $\L\ox\O_C(-Q_1-\dots-Q_d)$. 
Here, $\L$ is a line bundle of degree $d$, the particular choice of which 
being irrelevant.

The importance of the Abel maps appears already in the work by Abel: They 
encode a lot of geometric information about the curve. More precisely, 
the fibers of $A^d_\L$ are the linear equivalence classes of
degree-$d$ effective 
divisors of $C$. Thus, all possible embeddings of $C$ in projective spaces 
are known once we know its Abel maps.

The 1857 paper by Riemann introduced as well the notion of ``moduli,'' 
essentially by stating, in modern terms, that the moduli space $M_g$ of 
genus-$g$ smooth curves has dimension $3g-3$, if $g\geq 2$. The space itself 
was constructed almost a century afterwards, followed by the 
construction of its compactification $\ol M_g$, the moduli space of 
(Deligne--Mumford) stable curves. Those are the curves, possibly
reducible, whose only singularities are ordinary nodes and 
whose group of automorphisms is finite. Abel maps are so naturally defined 
that they vary well in families of smooth curves. It is thus natural to ask 
what happens when smooth curves degenerate to stable ones.

Not only natural, but potentially very useful. Understanding degenerations 
of Abel maps is tantamount to understanding degenerations of linear series. 
It was through the study of these degenerations that the celebrated 
Brill--Noether and Gieseker--Petri Theorems 
were proved; see \cite{GH} and 
\cite{Gi}. The approach to these theorems has been systematized by 
the theory of limit linear series on curves of compact type, developed 
by Eisenbud and Harris; see \cite{HM}. Other results have been obtained 
through the theory; see \cite{EH}.

The theory of Eisenbud and Harris points out to a partial compactification 
of the variety of linear series over the moduli space of curves of 
compact type. It is natural to ask whether one can extend this 
compactification over the whole $\ol M_g$, as Eisenbud and Harris themselves 
asked in \cite{EH2}. The supposedly intermediate step, that of constructing 
a compactification of the relative Picard scheme over $\ol M_g$, has been 
carried out by Caporaso \cite{C}. But the final step has proved to be
very difficult.

Sketches on how to deal with limit linear series for stable curves not of 
compact type are sparse in the literature. Recently, Medeiros and the 
second author \cite{EM} were able to describe limit canonical series on 
curves with two components, using the theory presented in \cite{Es}, the 
ingredients of which having already appeared in \cite{R}. But their 
description is rather complicated, and relies on the strong assumption 
that the components intersect each other at points in general position
on each component. The same assumption is present in \cite{EP}, where 
the case of curves with more components is partially considered.

In view of the difficulties, one might ask whether we stand a better 
chance of compactifying the variety of linear series over $\ol M_g$ by 
looking at degenerations of Abel maps. Indeed, Abel maps of all degrees 
have been constructed and studied for integral curves in \cite{AK}. Degree-1 
Abel maps for stable curves were constructed in \cite{CE} and studied in 
\cite{CCE}. Higher-degree Abel maps for curves of compact type appeared soon 
afterwards in \cite{CP}. 

At this point one might say that the theory of 
degenerations of Abel maps is on even terms with that of limit linear series. 
A comparison between the two theories has in fact appeared in \cite{EO}, 
based on the more refined notion of limit linear series introduced by 
Osserman \cite{O}, and thus limited to two-component curves. 

The present article advances the theory of degenerations of Abel maps, 
by presenting purely numerical conditions for the existence of degree-2 
Abel maps for any given nodal curve. Whether the curve is stable or not is 
immaterial. The specific moduli, even the genera of its irreducible 
components is immaterial. Where the points of intersection of components 
lie on each component is immaterial, in stark contrast with \cite{EM} and 
\cite{EP}. In fact, for nodal curves with the same dual graph, either all 
of them or none of them have 
degree-2 Abel maps. An algorithm has been implemented by means of 
\cocoa\ scripts, available at
$$
\text{\rm http://w3.impa.br/$\sim$esteves/CoCoAScripts/Abelmaps}
$$
to check the validity of these conditions for any 
given curve and has, so far, always returned a positive answer.
Our approach is simply to extend the construction done in \cite{Co} 
for two-component two-node 
curves. 
Here is what we do. Assume all the singularities of $C$ are ordinary nodes. 
Let $C_1,\dots,C_p$ denote the components of $C$. 
Instead of dealing with the curve $C$ itself, we 
consider a (one-parameter) smoothing $\C/B$ of $C$, 
that is, a flat, projective map $\C\to B$ to $B:=\text{Spec}(K[[t]])$ whose 
generic fiber $X$ is smooth and whose 
closed fiber is isomorphic to $C$, whence identified with $C$ by any chosen 
isomorphism. We assume $\C$ is regular. 
Furthermore, instead of considering the symmetric product, we consider 
the ordinary Cartesian product $\C^2:=\C\times_B\C$. The Jacobian is replaced 
by the so-called compactified Jacobian $\ol{\mathcal J}$, parameterizing 
torsion-free, rank-1 sheaves on the fibers of $\C/B$ 
that are $C_1$-quasistable with respect to a fixed polarization 
$\ul e=(e_1,\dots,e_p)$ on 
$C$; see Section \ref{sect2.2}. Its generic fiber over $B$ is isomorphic to 
the Jacobian of $X$. Finally, let $\P$ be a line bundle on $\C$ of 
relative degree $e_1+\cdots+e_p+2$ over $B$.

Consider the rational map $\alpha\:\C^2\dashrightarrow\ol{\mathcal J}$ 
defined over 
the generic point of $B$ by taking a pair of points $(Q_1,Q_2)$ on 
$X$ to $\P|_X\ox\O_X(-Q_1-Q_2)$. Using ``twisters,'' 
it is not difficult to show 
that $\alpha$ extends over the smooth locus of $\C^2/B$, whose complement 
is the codimension-2 locus whose components are of the form $N\times
C$ or $C\times N$, with $N$ a node of $C$; see Section \ref{2.3}. The map does not 
necessarily extend to the whole $\C^2$; already the simple case dealt
with in \cite{Co} shows this. Our goal is to produce a minimal 
resolution for $\alpha$.

By a minimal resolution for $\alpha$ we mean a map 
$\phi\:\wh\C^2\to\C^2$ which is an isomorphism away from the pairs $(R,S)$ 
of nodes of $C$, where either $R$ and $S$ are both reducible, or 
$R=S$. Furthermore, for such a pair $(R,S)$ we want that either $\phi$ 
be an isomorphism at 
it or $\phi^{-1}(R,S)\cong\IP^1_K$, the former only if $R\neq S$ and 
$\alpha$ extends over $(R,S)$. 
We want as well that $\phi$ be symmetric, that is, that the switch involution 
$\iota$ of $\C^2$ lift to an involution $\wh\iota$ of $\wh\C^2$, as we 
want to take the quotient of $\wh\C^2$ by the action of $\wh\iota$ to end up 
with a partial resolution of the relative symmetric product $S^2(\C/B)$. 
Finally, we want a map $\wh\alpha\:\wh\C^2\to\ol{\mathcal J}$ such that 
$\wh\alpha=\wh\alpha\wh\iota$ and $\wh\alpha=\alpha\phi$. Our 
Theorem \ref{mainthm} describes such a map $\phi$ when certain 
numerical conditions are verified on the dual graph $\Gamma_C$ of $C$.  

When these conditions are verified, 
our $\phi$ is produced by a sequence of blowups. The first 
is the blowup along the diagonal $\Delta\subset\C^2$, resulting in a flag 
scheme; see Section \ref{flag}. Clearly, $\iota$ lifts to an involution of 
this blowup; the quotient by the action of the lift is the relative 
Hilbert scheme $\text{Hilb}^2_{\C/B}$. We are thus producing a 
partial resolution of $\text{Hilb}^2_{\C/B}$ as well. The remaining blowups 
are along the (strict transforms of) products $Y\times Z$ of 
proper subcurves of 
$C$, chosen according to the numerical conditions verified on $\Gamma_C$. To 
make sure that $\iota$ lifts to an involution on $\wh\C^2$, if a blowup 
in the resolution sequence is along $Y\times Z$ for $Y\neq Z$, we require the next 
blowup to be along $Z\times Y$.

In Sections \ref{3} and \ref{4} we describe the effect of successive blowups 
of (modifications of) $\C^2$ along (the strict transforms of) products of 
subcurves of $C$. 
Being $\C$ regular, the product $\C^2$ is regular away from pairs 
$(R,S)$ of nodes of $C$. Blowing up $\C^2$ along $Y\times Z$ for 
proper subcurves $Y$ and $Z$ of $C$ desingularizes $\C^2$ at the pairs 
$(R,S)\in (Y\cap Y')\times (Z\cap Z')$, where $Y':=\ol{C-Y}$ and 
$Z':=\ol{C-Z}$, by ``adding'' $\IP^1_K$ over each such pair. The added 
$\IP^1_K$ belongs to the strict transforms of $Y\times Z$ and 
$Y'\times Z'$ but not to those of $Y\times Z'$ and $Y'\times Z$. The 
blowups along $Y\times Z$ and $Y'\times Z'$ are the same, but
different from those along $Y'\times Z$ and $Y\times Z'$.

We give a brief explanation of the nature of the numerical conditions we 
check on $\Gamma_C$; 
more details can be found in Section \ref{6}. Let $\Gamma^0_C$ be the 
essential dual graph of $C$, obtained from $\Gamma_C$ by removing the 
edges with equal ends. (Thus we are ignoring the irreducible nodes of $C$.) 
Let $V$ be the set of vertices and $E$ the set of edges of $\Gamma^0_C$. 
Then $V=\{C_1,\dots,C_p\}$. We 
may view $\ul e$ as a map $\ul e\:V\to\mathbb Q$ taking $C_i$ to $e_i$ for each 
$i=1,\dots,p$. We may also view the multidegree of $\P|_C$ as a map 
$\ul q\:V\to\mathbb Z$, 
taking $C_i$ to $\deg(\P|_{C_i})$ for each $i=1,\dots,p$. 
Set $v:=C_1$. We call $(\Gamma^0_C,\ul e,\ul q,v)$ degree-2 Abel data. 

A resolution of the degree-2 Abel data $(\Gamma^0_C,\ul e,\ul q,v)$ is a 
map $\ul r\:E^2\to V^2$ sending the diagonal of $E^2$ to that of $V^2$ 
and such that $\ul r(R,S)_1$ is an end of $R$ and $\ul r(R,S)_2$ is an 
end of $S$. We define when a resolution $\ul r$ is quasistable; see 
Subsections \ref{admfun} and \ref{admfun2}. 
The definition is based only on the (combinatorial) degree-2 Abel data. 

The singular locus $\Sigma$ 
of the Abel data is the subset of pairs $(R,S)\in E^2$ such that for any two 
quasistable resolutions $\ul r_1$ and $\ul r_2$, either 
$\ul r_1(R,S)=\ul r_2(R,S)$ or $\ul r_1(R,S)_j\neq\ul r_2(R,S)_j$ for 
$j=1,2$; see Definition \ref{admfun3}. 
Notice that $\Sigma$ contains the diagonal $\Delta_E$ of $E^2$. 
We say that $\Sigma$ is solvable if quasistable 
resolutions exist.

A blowup sequence for the Abel data is a pair $(I_1,I_2)$ of sequences 
$I_1=(I_{1,1},\dots,I_{1,u})$ and $I_2=(I_{2,1},\dots,I_{2,u})$ of
equal lengths of proper nonempty subsets of $V$. It is called 
symmetric if the subsets are symmetric to each other, that is, 
whenever $I_{1,j}\neq I_{2,j}$ we have 
$(I_{1,j-1},I_{2,j-1})=(I_{2,j},I_{1,j})$ or
$(I_{1,j+1},I_{2,j+1})=(I_{2,j},I_{1,j})$. The center $\Xi$ of the blowup
sequence is the set of pairs $(R,S)\in E^2$ such that
either $R=S$ or  there is $j$ such that 
one and only one end of $R$ lies in $I_{1,j}$ 
and one and only one end of $S$ lies in $I_{2,j}$. We call the minimum
such $j$ the order of $(R,S)$ in the blowup sequence.

We say that a blowup sequence $(I_1,I_2)$ resolves the Abel
data if $\Sigma$ is solvable, $\Xi\supseteq\Sigma$, and 
for each $(R,S)\in\Sigma-\Delta_E$, 
any quasistable resolution $\ul r$ satisfies 
$\ul r(R,S)\in I_{1,j}\times (V-I_{2,j})$ or 
$\ul r(R,S)\in (V-I_{1,j})\times I_{2,j}$, where $j$ is the order of $(R,S)$.
We say that $(I_1,I_2)$ resolves the Abel data minimally if 
$\Xi=\Sigma$. 

Given a blowup sequence $(I_1,I_2)$, 
consider the map $\phi\:\wh\C^2\to\C^2$ 
obtained by blowing up $\C^2$ along the diagonal and then successively 
along the strict transforms of $Y_1\times Z_1,\dots,Y_u\times Z_u$, 
where $Y_j:=\cup_{l\in I_{1,j}}C_l$ and $Z_j:=\cup_{l\in I_{2,j}}C_l$ for 
$j=1,\dots,u$. If $(I_1,I_2)$ is symmetric, so
is $\phi$. Our Theorem \ref{mainthm} says that, if $(I_1,I_2)$
resolves the Abel data minimally, then $\phi$ 
is a minimal resolution for $\alpha$. 

The biggest unanswered 
question is whether the singular locus of a degree-2 Abel data is always 
solvable. Numerous examples checked by the implemented algorithm suggest 
that, not only is the answer positive, but also that there is a
symmetric blowup sequence resolving the Abel data minimally, and thus
producing a symmetric minimal resolution for $\alpha$. 
Also, the third
author \cite{Pac} showed that, when $\mathcal P|_C$ has multidegree
$(2,0,\dots,0)$ and $\ul e=(0,0,\dots,0)$, the singular locus is
solvable. Furthermore, in this case there is a blowup sequence
$(I_1,I_2)$ resolving the Abel data (possibly nonminimally),
where $I_{1,j}=I_{2,j}$ for every $j$, and the subcurves 
$Y_j:=\cup_{l\in I_{1,j}}C_l$ are 2-tails or 3-tails.

A related question would be: Can we understand 
the minimal resolution found in a functorial way, that is, does it 
represent a natural functor? It might be easier to answer the first 
question after the second is answered.

Restricting $\phi\:\wh\C^2\to\C^2$ and 
$\wh\alpha$ over the special 
point of $B$ we obtain a ``resolution'' $\wh C^2\to C^2$ and an ``Abel map'' 
$\wh C^2\to\ol J$. The ``resolution'' does not depend on the smoothing 
$\C/B$ taken but the ``Abel map'' does. More precisely, the ``Abel map'' 
depends (only) on the resulting enriched structure, as defined by Main\`o 
\cite{M}.

In the process leading to the proof of Theorem \ref{mainthm}, we are 
led in Sections \ref{3} and \ref{4} 
to understand blowups of the triple product $\C^3:=\C\times_B\C\times_B\C$. 
It turns out in Section \ref{6} 
to be immaterial what sequence of blowups of $\C^3$ 
we perform. However, the information 
we gather will certainly be useful for anyone trying to construct 
degree-3 Abel maps. We have this in mind in Section \ref{5} as well, where 
we study the relationship between torsion-free, rank-1 sheaves on $C$ and on 
the curve obtained from $C$ by replacing each node by a chain of smooth 
rational curves of variable length (including length 0, which in fact 
corresponds to not replacing the node). For our immediate purposes, the 
length is at most 2, but anyone dealing with higher-degree Abel maps will 
likely profit from our study of the general case.

We worked with the compactification by quasistable sheaves produced in 
\cite{E01}, which is a fine moduli space. Another compactification, that 
considered in \cite{C} or \cite{Sesh}, could have been used. Since there is 
a map from the former to the latter, we produce degree-2 Abel maps 
to the latter as well. 
However, by the same reason, a priori it could be easier to 
construct these maps directly.

In short, here is a summary of the article. In Section \ref{2}, we review the 
theory of compactified Jacobians, how one can use ``twisters'' to extend 
Abel maps up to codimension 2, and the construction of the flag scheme. 
In Section \ref{3}, we describe the blowups of $\C^2$ and $\C^3$ in the 
local analytic setting. In Section \ref{4}, we use this local analysis to 
describe globally blowups of $\C^2$ and $\C^3$. In Section \ref{5}, 
we study the relationship between torsion-free, rank-1 sheaves on $C$ and on 
the curve obtained from $C$ by replacing each node by a chain of smooth 
rational curves of variable length. In Section \ref{6}, we prove our main 
theorem, Theorem \ref{mainthm}, already discussed above. Finally, in 
Section \ref{7}, we present examples of curves for which symmetric minimal 
resolutions for degree-2 Abel maps exist, presenting the resolution in each 
case.

\section{Compactified Jacobians}\label{2}

\subsection{Curves and their smoothings}\label{preli}

A \emph{curve} is a reduced, projective scheme of pure dimension 1 
over an algebraically closed field $K$. A curve may have several 
irreducible components, which will be simply called \emph{components}. 
We will always assume our curves to be \emph{nodal}, 
meaning that the singularities are \emph{nodes}, that is, 
analytically like the origin on the 
union of the coordinate axes of the plane $\mathbb A^2_K$. A node is 
called \emph{irreducible} if its removal does not disconnect the curve in any 
Zariski neighborhood of it; otherwise, it is called \emph{reducible}. 

Let $C$ be a curve defined over an algebraically closed field $K$. 
Its \emph{dual graph} $\Gamma_C$ is the graph whose 
set of vertices $V_C$ consists of the components of $C$ and whose 
set of edges $E_C$ consists of 
its singularities, the ends of an edge being the components on 
which the corresponding singularity lies. The \emph{essential dual graph} 
$\Gamma^0_C$ is the subgraph of $\Gamma_C$ where the edges with equal ends, 
corresponding to irreducible nodes of $C$, are removed.

A \emph{subcurve} of 
$C$ is a (nonempty, reduced) union of components of 
$C$. It is a curve by itself. 
If $Y$ is a proper subcurve of $C$, we let $Y':=\overline{C-Y}$, and call it the 
\emph{complementary subcurve} of $Y$. We set $k_Y:=\#(Y\cap Y')$.  

A \emph{chain of rational curves} is a 
curve whose components are smooth and
rational and can be ordered, $E_1,\dots,E_n$, in such a way that 
$\#E_i\cap E_{i+1}=1$ for $i=1,\dots,n-1$ and  
$E_i\cap E_j=\emptyset$ if $|i-j|>1$.  If $n$ is the number of components, 
we say that the chain has \emph{length $n$}. Two chains of 
the same length are isomorphic. The
components $E_1$ and $E_n$ are called the \emph{extreme curves} of the 
chain. A connected subcurve of a chain is also a chain, and is called a 
\emph{subchain}.

Let $\N$ be a collection of nodes of $C$, 
and $\eta\:\N\to\mathbb N$ a function. Denote by $\wt C_\N$ the partial 
normalization of $C$ at $\N$. For each $P\in\N$, let $E_P$ be a chain of 
rational curves of length $\eta(P)$. Let $C_\eta$ denote the curve 
obtained as the union of $\widetilde C_\N$ and the $E_P$ for $P\in\N$ in the 
following way: Each chain $E_P$ intersects no other chain, but intersects 
$\wt C_\N$ transversally at two points, the branches over $P$ on $\wt C_\N$ 
on one hand, and nonsingular points on each of the two extreme curves of $E_P$  
on the other hand. There is a natural map 
$\mu_\eta\:C_\eta\to C$ collapsing each 
chain $E_P$ to a point, whose restriction to $\wt C_\N$ is the partial 
normalization map. The curve $C_\eta$ and the map $\mu_\eta$ are well-defined 
up to $C$-isomorphism.

There are two special cases of the above construction that will be useful 
for us. First, if $\N=\{R\}$ and $\eta$ takes $R$ to $1$, let 
$C_R:=C_\eta$ and $\mu_R:=\mu_\eta$. Second, if $\N=\N(C)$, where 
$\N(C)$ is the collection of reducible nodes of $C$, and 
$\eta\:\N(C)\to\mathbb N$ is the constant function with value $m$, let 
$C(m):=C_\eta$ and $\mu(m):=\mu_\eta$. Set 
$C(0):=C$ and $\mu(0):=\text{id}_C$.

A \emph{family of} (\emph{connected}) \emph{curves} 
is a proper and flat morphism 
$\pi\colon\mathcal C\rightarrow B$ whose geometric fibers are
connected curves. (All schemes are assumed locally Noetherian.)
If $b$ is a geometric point of $B$, we put $C_b:=\pi^{-1}(b)$. 
A \emph{smoothing} of $C$ is a generically smooth family of curves 
$\pi\colon\C\rightarrow B$ over an irreducible scheme $B$, together with a 
point $0\in B$ whose residue field is $K$, and an 
isomorphism $C_0\cong C$. We will always assume 
$B\cong\text{Spec}(K[[t]])$ and $\C$ is regular at the reducible 
nodes of $C_0$. Also, the isomorphism $C_0\cong C$ will be left implicit.

The \emph{degree}, or \emph{total degree}, of a coherent sheaf $\F$ 
of constant generic rank 1 on $C$ is 
$$
\deg(\F):=\chi(\F)-\chi(\O_C).
$$
We will identify line bundles with invertible sheaves. If $\F$ is an 
invertible sheaf, its \emph{multidegree} is the function 
$\underline d_\F\:V_C\to\mathbb Z$ that to each 
component $Y$ of $C$ associates the integer $\deg(\F|_Y)$. Given an ordering 
of the components of $C$, we may also view $\ul d_\F$ as a tuple of integers.

Given a smoothing $\C/B$ of $C$, a \emph{twister} of $\C/B$ is a 
special line bundle of degree 0 on $C$, of the form $\O_{\C}(Z)|_C$, 
where $Z$ is a Cartier divisor of $\C$ supported in $C$, so a formal sum of 
components of $C$. (Notice that each component of $C$ is a Cartier divisor 
of $\C$ because $\C$ is regular at the reducible nodes of $C$.) 
A twister has degree 0 by continuity of the 
degree, since $\O_{\C}(Z)$ is trivial away from $C$. 

The smoothing $\C/B$ being fixed, two divisors $Z_1$ and $Z_2$ 
of $\C$ supported in $C$ give isomorphic twisters, 
$\O_{\C}(Z_1)|_C\cong\O_{\C}(Z_2)|_C$, if and only 
if they give twisters of the same multidegree, if and only if $Z_1-Z_2$ is a 
multiple of $C$; see \cite{CAJM}, Lemma~3.4, p.~7.
The multidegree of a twister $\O_{\C}(Z)|_C$ depends only on $Z$, but the 
twister itself depends on the smoothing $\C/B$. As 
soon as the latter is fixed, we will use the notation
$$
\O_C(Z):=\O_{\C}(Z)|_C.
$$
In her thesis \cite{M}, Main\`o describes which invertible sheaves on $C$ are 
twisters. A description is also given in Section 6 of 
\cite{EM}, p.~288.

{\bf Throughout the article}, 
fix an algebraically closed field $K$ and a 
\emph{connected} 
curve $C$ over $K$. Denote by $C_1,\ldots,C_p$ its 
components, and fix a smoothing $\C/ B$ of $C$ with regular total 
space $\C$. 

\subsection{Jacobians and their compactifications}\label{sect2.2}

Fix an integer $d$. 
Since $C$ is a proper scheme over $K$, by \cite{BLR}, Thm.~8.2.3,
p.~211, there is a scheme, locally of finite type over $K$, parameterizing 
degree-$d$ line bundles on $C$; denote it by $J_C^d$. It decomposes as 
\begin{equation}\label{Jacobdecomp}
J^d_C=\underset{ d_1+\ldots+d_p=d}{\underset{\ul{d}=(d_1,\dots,d_p)}
{\coprod}}J^{\ul{d}}_C,
\end{equation}
where $J^{\ul{d}}_C$ is the connected component of $J_C^d$ parameterizing line 
bundles $\L$ such that 
$\deg(\L|_{C_i})=d_i$  for $i=1,\dots,p$. The $J^{\ul{d}}_C$ are 
quasiprojective varieties. 

The scheme $J_C^d$ is in a natural way an open dense subscheme of $\J_C^d$, 
the scheme over $K$ 
parameterizing torsion-free, rank-1, simple sheaves of degree $d$ 
on $C$; see \cite{E01} for the construction of $\J_C^d$ and its properties. 
(Recall that a coherent sheaf $\I$ on $C$ is 
\emph{torsion-free} if it has no embedded components, \emph{rank-1} if it has 
generic rank 1 at each component of $C$, and \emph{simple} if 
$\text{Hom}(\I,\I)=K$.) The scheme $\J_C^d$ is universally closed over $K$ 
but, in general, not separated and only locally of finite type. Moreover, in 
contrast to $J_C^d$, the scheme $\J_C^d$ is connected, hence not
easily decomposable. Thus, to deal with a 
manageable piece of it, we resort to polarizations. 

For our purposes, a 
\emph{polarization of degree $d$} is any $p$-tuple of rational numbers 
$\ul{e}=(e_1,\ldots,e_p)$ summing up to $d$. 
Fixing a polarization $\ul{e}$ of degree $d$, we say that a 
degree-$d$, torsion-free, rank-1 sheaf $\I$ on $C$ is \emph{semistable} if for 
every subcurve $Y\subseteq C$ we have
\begin{equation}\label{BI1}
\Big|\deg(\I_Y)-e_Y\Big|\leq \frac{k_Y}{2},
\end{equation}
where $\I_Y$ is the restriction of $\I$ to $Y$ modulo torsion, and 
$$
e_Y:=\sum_{C_i\subseteq Y}e_i.
$$
We say that $\I$ is \emph{stable} if all the 
inequalities in \eqref{BI1} are strict for proper subcurves $Y\subsetneqq C$. 
Furthermore, we say that $\I$ is \emph{$C_i$-quasistable}, for a 
component $C_i$ of $C$, if all the inequalities in \eqref{BI1} 
hold and, moreover,
\begin{equation}\label{BI2}
\deg(\I_Y)>e_Y-\frac{k_Y}{2}
\end{equation}
whenever $Y$ is a proper subcurve of $C$ containing $C_i$. 

The $C_i$-quasistable sheaves are simple, what can be easily proved using 
for instance \cite{E01}, Prop.~1, p.~3049. 
Their importance is that they form 
an open subscheme $\J^{d,i}_C$ of $\J^d_C$ that is projective over 
$K$. Furthermore, there is a projective 
scheme $\ol{\mathcal J}^{d,i}_{\C/B}$ over $B$ whose fiber over the special 
point of $B$ is 
(isomorphic to) $\J^{d,i}_C$ 
and whose fiber over a geometric point $b$ of $B$ over its generic point is 
(isomorphic to) $J^d_{C_b}$. 

Conditions \eqref{BI1} and \eqref{BI2} are purely numerical. In fact, 
let $\Gamma_C$ be the dual graph of $C$. A 
\emph{generalized multidegree} $\ul{d}$ 
is the assignment of an integer to each vertex and each 
edge of $\Gamma_C$, with the condition that to an edge only 0 or 1 is assigned. 
The (\emph{total}) \emph{degree} $d$ of $\ul{d}$ is the sum of all these 
assigned integers. A subcurve $Y$ of $C$ corresponds to a subgraph 
$\Gamma_Y$ of $\Gamma_C$, whose vertices are the components of 
$Y$ and whose edges are the singularities of $Y$. The \emph{degree of 
$\ul{d}$ on $Y$}, denoted $d_Y$, is 
the sum of all integers assigned to the vertices and edges of 
$\Gamma_Y$. 

Fixing a polarization $\ul{e}$ of degree $d$, we say that a 
generalized multidegree $\ul{d}$ of total degree $d$ 
is \emph{semistable} if for 
every proper subcurve $Y\subsetneqq C$ we have 
$$
\Big|d_Y-e_Y\Big|\leq \frac{k_Y}{2},
$$
and \emph{stable} if strict inequalities hold. 
Also, we say that $\ul{d}$ is \emph{$C_i$-quasistable}, for a 
component $C_i$ of $C$, if $\ul d$ is semistable and
$$
d_Y>e_Y-\frac{k_Y}{2}
$$
whenever $Y\supseteqq C_i$. 

To a degree-$d$ torsion-free, rank-1 sheaf $\I$ on $C$ there corresponds a 
generalized multidegree 
$\ul d$ of total degree $d$ 
that assigns 0 to each node of $C$ around which $\I$ is invertible, 
and 1 otherwise, and assigns $\deg(\I_{C_i})$ to each component 
$C_i$. Then $d_Y=\deg(\I_Y)$ for each subcurve $Y$ of $C$, and thus 
$\I$ is semistable, or stable, or $C_i$-quasistable 
if and only if $\ul d$ is. 

{\bf Throughout the article}, fix an integer $f$ and fix a degree-$f$ 
polarization 
$\ul e:=(e_1,\dots,e_p)$. Set 
$$
J:=J_C^f\cap\ol J_C^{f,1},\quad\ol J:=\ol J_C^{f,1},\quad
\ol{\mathcal J}:=\ol{\mathcal J}^{f,1}_{\C/B}.
$$ 

\subsection{Abel maps}\label{2.3}

Let $d$ be a positive integer, and $\P$ a line bundle on $C$ of degree 
$d+f$. If $C$ is smooth, its \emph{degree-$d$ Abel map} 
is the map $S^d(C)\to J$ induced from  
$$
\begin{array}{rcl}
\alpha^d\col C^{\times d}&\longrightarrow& J\\
(Q_1,\ldots,Q_d)&\mapsto&\P\ox\O_C(-Q_1-\ldots-Q_d),
\end{array}
$$
where $S^d(C)$ is the quotient of $C^{\times d}$ by the 
group of permutations 
of its factors.

(Often the map found in the literature is 
the one obtained from $\alpha^d$ by composing with the isomorphism 
$J^f_C\to J^{-f}_C$ taking a line bundle to its dual. 
Our choice does not differ from this one in any 
relevant way for the purposes of this article.)

The above definition of $\alpha^d$ does not make sense if $C$ is singular. 
Indeed, 
it would be natural to think of defining the degree-$d$ Abel map 
of $C$ by letting
$$
(Q_1,\ldots,Q_d)\mapsto\P\ox\I_{Q_1|C}\ox\cdots\ox\I_{Q_d|C},
$$
but two problems arise: First, the sheaf on the right might not be 
torsion-free; in fact, it will fail to be torsion-free if and only if two 
among the $Q_i$ are the same node of $C$. This is not a serious problem, as 
we can replace $C^{\times d}$, or rather $S^d(C)$, by the Hilbert scheme 
$\text{Hilb}^d_C$, parameterizing subschemes $F$ of length $d$ of $C$, and 
define the degree-$d$ Abel map by letting 
$$
F\mapsto \P\ox\I_{F|C}.
$$
This works if $C$ is irreducible; see \cite{AK}, (8.2), p.~101. 
But if $C$ is reducible, a second 
problem might nonetheless arise, namely that 
$\P\ox\I_{Q_1|C}\ox\cdots\ox\I_{Q_d|C}$ or $\P\ox\I_{F|C}$ might 
not be $C_1$-quasistable, whence not parameterized by $\ol J$. 

To solve the second problem, we resort to twisters arising from 
the smoothing $\C/B$. Associate to a twister 
$\O_C(Z)$ a multidegree $\underline d^Z$, that of the line bundle. 
Recall that 
the multidegree does not depend on the choice we made of the smoothing 
$\C/B$. Also, the smoothing being fixed, there is a bijective correspondence 
between twisters and their multidegrees.  

Now, let $\dot C$ denote the smooth locus of $C$. 
The product $\dot C^{\times d}$ 
decomposes as
$$
\dot C^{\times d}=\underset{1\leq i_1,\dots,i_d\leq p}
{\underset{\underline i=(i_1,\dots,i_d)\in\mathbb N^d}
{\coprod}}\dot C_{i_1}\times \dot C_{i_2}\times\cdots\times\dot C_{i_d},
$$
where $\dot C_j:=C_j\cap\dot C$ for each $j=1,\dots,p$. To each 
$d$-tuple $\ul i$ as above, we associate a 
multidegree $\ul f:=(f_1,\dots,f_p)$ of degree $f$, 
by letting
$$
f_j:=\deg(\P|_{C_j})-\#\{\ell\,|\,i_\ell=j\}
$$ 
for $j=1,\dots,p$, and to $\ul f$ a unique twister whose multidegree $\ul t$ 
is  such that $\ul f-\ul t$ is 
$C_1$-quasistable. Let $Z_{\ul i}$ be a formal sum of components of 
$C$ such that $\O_C(Z_{\ul i})$ has multidegree $\ul t$. Then there is a 
natural map 
$$
\begin{array}{rcl}
\dot{\alpha}^d_{\ul i}\col 
\dot C_{i_1}\times\cdots\times\dot C_{i_d}&\longrightarrow& J\\
(Q_1,\ldots,Q_d)&\mapsto&\P\ox\O_C(-Q_1-\dots-Q_d)\ox\O_C(-Z_{\ul i}).
\end{array}
$$
Putting together the $\dot{\alpha}^d_{\ul i}$, we obtain a 
map $\dot{\alpha}^d\:\dot C^{\times d}\to J$.

(The existence of the twister with multidegree $\ul t$ is the numerical version of 
\cite{E01}, Thm.~32, (4), p.~3068, applied to the restriction over the 
generic fiber of any extension to $\C$ of any line 
bundle on $C$ having multidegree $\ul f$. As for uniqueness, suppose 
there exist twisters of multidegrees $\ul t_1$ and $\ul t_2$
such that $\ul g:=\ul f-\ul t_1$ and $\ul h:=\ul f-\ul t_2$ are 
$C_1$-quasistable.  By contradiction, suppose $\ul t_1\ne \ul t_2$. Then 
$\ul b:=\ul g-\ul h$ is the multidegree of a nontrivial twister. 
Thus, as in the proof of \cite{C}, Lemma~4.1, p.~623, there exist
an integer $q\geq 1$ and a decomposition $C=Y_0\cup\dots\cup Y_q$ in subcurves 
$Y_j$ containing no components in common such that 
$\ul b=\sum_{j=1}^q\ell_j \ul b_j$, where $\ul b_j$
is the multidegree of $\O_C(Y_j)$ for $j=1,\dots,q$, and the $\ell_j$
are distinct positive integers. Set $Y:=Y_0$. Then 
$b_Y\geq k_Y$ and $b_{Y'}\le -k_Y$. From the semistability of $\ul g$ and 
$\ul h$, it follows that $g_{Y'}-e_{Y'}=-\frac{k_{Y'}}{2}$ and 
$h_Y-e_Y=-\frac{k_Y}{2}$. Thus, if $C_1\subseteq Y$ (respectively, if 
$C_1\not\subseteq Y$), then $\ul h$ (respectively, $\ul g$) 
is not $C_1$-quasistable, reaching a contradiction.) 

The map $\dot{\alpha}^d$ can be extended over the generic point of
$B$.
To do this, first we extend the line bundle $\P$ to one on $\C$. This is 
possible because $\P$ is associated to a Cartier divisor with 
support on $\dot C$, and $B\cong\text{Spec}(K[[t]])$. Abusing notation, 
$\P$ will also denote the extension. For each positive integer $j$, set 
$\C^j:=\C\times_B\cdots\times_B\C$, the fibered product of $j$ 
copies of $\C$ over $B$. Likewise, $\dot\C^j$ denotes the product 
of $j$ copies of $\dot\C$ over $B$, where $\dot\C$ denotes the smooth
locus of $\C/B$. 
Denote by  $\xi\:\C^{d+1}\to\C$ and 
$\rho_i:\C^{d+1}\rightarrow\C^2$ 
the projection onto the last factor and that onto the product over 
$B$ of the $i$-th and last 
factors, for $i=1,\dots,d$. 
Let $\Delta\subset\C^2$ be the diagonal subscheme, and put
$$
\Delta_{i}:=\rho_i^{-1}(\Delta)
$$
for each $i=1,\dots,d$. 

Since $\C$ is regular, the product 
$\dot C_{i_1}\times\dots\times\dot C_{i_d}\times C_{i_{d+1}}$ 
is a Cartier divisor of $\dot\C^d\times_B\C$ for every 
$i_1,\dots,i_{d+1}\in\{1,\dots,p\}$. 
So, for each $\ul i=(i_1,\dots,i_d)\in\{1,\dots,p\}^d$, if 
$Z_{\ul i}=\sum_ja_jC_j$, 
we have a well-defined Cartier divisor:
$$
\dot C_{i_1}\times\dots\times\dot C_{i_d}\times Z_{\ul i}:=
\sum_j a_j\dot C_{i_1}\times\dots\times\dot C_{i_d}\times C_j.
$$

Finally, we obtain a map 
\begin{equation}\label{dotalpha} 
\dot\alpha^d_{\C/B}\:\dot\C^d\lra\ol{\mathcal J}
\end{equation}
defined by the family of invertible sheaves on 
$\dot\C^d\times_B\C/\dot\C^d$ obtained as the tensor product of the 
``natural'' sheaf
$$
\Big(\xi^*\P\ox\I_{\Delta_1|\C^{d+1}}\ox\cdots\ox\I_{\Delta_d|\C^{d+1}}
\Big)\Big|_{\dot\C^d\times_B\C}
$$
with its ``correction:''
$$
\underset{1\leq i_1,\dots,i_d\leq p}
{\underset{\ul i=(i_1,\dots,i_d)\in\mathbb N^d}
{\bigotimes}}\O_{\dot\C^d\times_B\C}
(-\dot C_{i_1}\times\dots\times\dot C_{i_d}\times Z_{\ul i}).
$$
As the correction is by a Cartier divisor supported over $0\in B$, 
it follows that $\dot\alpha^d_{\C/B}$ restricts to the ``classical'' 
degree-$d$ Abel map on a smooth geometric fiber of $\C/B$.

We may ask whether $\dot\alpha^d_{\C/B}$, viewed as a rational map 
on $\C^d$, defined in codimension 2, extends to the whole $\C^d$. 
Surprisingly, the answer is yes for $d=1$, as shown in \cite{CE} and 
\cite{CCE}, at least if $\P=\O_C$ or $\P=\O_C(P)$ for $P\in\dot C_1$. 
However, as it can be expected, the answer is, in general, no for 
$d\geq 2$; see \cite{Co}. Moreover, as shown in \cite{Co} and as we will see, 
a resolution is not just a matter 
of blowing up along all the diagonals of $\C^d$, small and large. More 
blowups are often necessary. We will see ahead the case $d=2$.

{\bf Throughout the article}, fix an invertible sheaf $\P$ on $\C$ 
whose restriction to $C$ has degree $f+2$. 

\subsection{The flag scheme}\label{flag}

Recall that $\C^j$ is the fibered product of $j$ copies of $\C$ over $B$, for 
$j\in\mathbb N$. 
Let $\Delta\subset\C^2$ be the diagonal subscheme, and $\Delta_1$ and 
$\Delta_2$ the two ``large diagonals'' of 
$\C^3$, inverse images of $\Delta$ under the projections 
$\rho_1,\rho_2\:\C^3\to\C^2$, where $\rho_i$ is the projection onto the product 
over $B$ of the $i$-th and last factors of $\C^3$, for $i=1,2$.

Our first goal is to modify the base $\C^2$ to be able to replace 
$\I_{\Delta_1|\C^3}\ox\I_{\Delta_2|\C^3}$ by a relatively torsion-free, 
rank-1 sheaf 
over the base wherever the tensor product 
is not so, that is, over the pairs $(R,R)$, 
where $R$ is a node of $C$. This is done by replacing $\C^2$ by the 
\emph{flag scheme}:
$$
\text{\bf P}_{\C^2}(\I_{\Delta|\C^2}):=
\text{Proj}_{\C^2}(\S(\I_{\Delta|\C^2})),
$$
where $\S(\I_{\Delta|\C^2})$ is the sheaf of symmetric algebras of 
$\I_{\Delta|\C^2}$. In our case, another description is available.

\begin{Prop}\label{fl} 
Let $\phi\:\wt\C^2\to\C^2$ be the blowup of $\C^2$ along 
$\Delta$. Let $\rho_i:=p_i\phi$, where $p_i\:\C^2\to\C$ is the projection 
onto the $i$-th factor for $i=1,2$. Let $R\in C$. For $i=1,2$, let 
$X_i:=\rho^{-1}_i(R)$ and denote by $\mu_i\:X_i\to C$ the restriction of 
$\rho_{3-i}$ to $X_i$. Then the following statements hold:
\begin{enumerate}
\item $\wt\C^2$ is $\C^2$-isomorphic to $\text{\bf P}_{\C^2}(\I_{\Delta|\C^2})$.
\item $\rho_i$ is flat for $i=1,2$.
\item If $R$ is not a node of $C$ then $\mu_i$ is an isomorphism for $i=1,2$.
\item If $R$ is a node of $C$ then $X_i$ is $C$-isomorphic to $C_R$ and 
$\wt\C^2$ is regular along the rational component of $X_i$ contracted by 
$\mu_i$ for $i=1,2$.
\end{enumerate}
\end{Prop}

\begin{proof} As the Rees algebra of an ideal is a quotient of the 
symmetric algebra, $\wt\C^2$ is naturally a subscheme of 
$\text{\bf P}_{\C^2}(\I_{\Delta|\C^2})$. To show 
the blowup is the whole flag scheme is a local matter that need only be checked at 
the points where $\Delta$ fails to be Cartier, that is, at the pairs 
$(R,R)$, where $R$ is a node of $C$. Recall that $\C$ is regular. Thus 
the completion of the 
local ring of $\C$ at $R$ is (isomorphic to) 
the quotient of the formal power series ring 
$K[[t,x_0,x_1]]$ by the ideal $(t-x_0x_1)$. That of the 
local ring of $\C^2$ at $(R,R)$ is the quotient of the formal power series ring 
$K[[t,x_0,x_1,y_0,y_1]]$ by the ideal $(t-x_0x_1,t-y_0y_1)$, with $\Delta$ 
corresponding to the ideal $(x_0-y_0,x_1-y_1)$. We need to compare the 
symmetric and Rees algebras of this ideal. We 
are reduced to considering the cone $V$ in $\IA^4_K$, with 
coordinates $x_0,x_1,y_0,y_1$, given by $x_0x_1=y_0y_1$, and showing that 
its blowup along the subcone given by the ideal 
$I:=(x_0-y_0,x_1-y_1)\subset K[V]$ 
is equal to $\IP_V(I)$. Both are indeed equal to the 
subscheme $Y\subset\IA^4_K\times\IP^1_K$ given by 
$$
\left\{\begin{array}{c}
\alpha_1x_0=\alpha_0 y_1\\
\alpha_1y_0=\alpha_0 x_1\\
\end{array}\right.
$$
where $\alpha_0,\alpha_1$ are homogeneous coordinates of $\IP^1_K$. Statement 1 
is proved.

The remaining statements follow from the above description. Indeed, 
Statement 3 follows from the fact that $\phi$ fails to be an isomorphism 
only above pairs $(R,R)$ where $R$ is a node 
of $C$. Also, to prove Statement 2 we need only observe 
that the map $Y\to\IA^2_K$ sending 
$(x_0,x_1,y_0,y_1;\alpha_0\:\alpha_1)$ to $(x_0,x_1)$ if $i=1$ and $(y_0,y_1)$ if 
$i=2$ is flat.

(Alternatively, in more generality, flatness follows 
from \cite{EGK}, Lemma~3.4, p.~600, as the proof given there 
applies not only to families of integral curves but more generally 
to families of reduced curves with surficial singularities.)

Finally, setting $y_0=y_1=0$ in $Y$, we 
obtain the union of three lines in $\IA^4_K\times\IP^1_K$,
$$
\begin{array}{cc}
X_0:&y_0=y_1=x_0=\alpha_0=0\\
E:&y_0=y_1=x_0=x_1=0\\
X_1:&y_0=y_1=x_1=\alpha_1=0,
\end{array}
$$
forming a chain: $X_0$ and $X_1$ intersect $E$ transversally at a single 
point, distinct for each, but do not intersect. Setting $x_0=x_1=0$ in $Y$, 
an analogous statement holds. Statement 4 follows from this 
description and the smoothness of $Y$.
\end{proof}

There is a natural subscheme $F_2\subset\wt\C^2\times_B\C$ which is flat of 
relative length 2 over $\wt\C^2$ and contains 
both the inverse images of $\Delta_1$ and $\Delta_2$ under the natural 
map $\wt\C^2\times_B\C\to\C^3$. 
More precisely, let $\rho_1$ and $\rho_2$ be the compositions of the structure map 
$\phi\:\wt\C^2\to\C^2$ with the projections onto the first and second factors 
of $\C^2$. The structure map itself can be written in different ways:
$$
\phi=(\rho_1,\rho_2)=(\rho_1\times 1)\circ(1, \rho_2)=\sigma\circ(\rho_2,\rho_1)=
\sigma\circ(\rho_2\times 1)\circ(1, \rho_1),
$$
where $\sigma$ is the switch involution of $\C^2$. Notice that
$$
\I_{(\phi\times 1)^{-1}\Delta_1|\wt\C^2\times_B\C}=(\rho_1\times 1)^*\I_{\Delta}
\quad\text{and}\quad
\I_{(\phi\times 1)^{-1}\Delta_2|\wt\C^2\times_B\C}=(\rho_2\times 1)^*\I_{\Delta}.
$$
Pulling back the first (resp. second) sheaf above to $\wt\C^2$ under 
$(1,\rho_2)$ (resp.~under $(1,\rho_1)$), 
and pushing forward under the same map, 
we get natural surjections:
$$
(\rho_1\times 1)^*\I_{\Delta}\lra(1,\rho_2)_*\phi^*\I_{\Delta|\C^2}\quad
\text{and}\quad
(\rho_2\times 1)^*\I_{\Delta}\lra(1,\rho_1)_*(\sigma\phi)^*\I_{\Delta|\C^2}.
$$
If $\epsilon\:\phi^*\I_{\Delta|\C^2}\to\O_{\wt\C^2}(1)$ is the 
tautological surjection, 
we obtain by composition two surjections,
$$
\begin{array}{c}
\I_{(\phi\times 1)^{-1}\Delta_1|\wt\C^2\times_B\C}\lra(1,\rho_2)_*\O_{\wt\C^2}(1),\\
\I_{(\phi\times 1)^{-1}\Delta_2|\wt\C^2\times_B\C}\lra(1,\rho_1)_*\O_{\wt\C^2}(1).
\end{array}
$$
(Here we use that $\sigma^*\I_{\Delta|\C^2}=\I_{\Delta|\C^2}$.) Their kernels are 
equal to the sheaf of ideals of $F_2$; see \cite{Kmult}, Section 2, p.~109, 
for proof.

\section{Blowups: local analysis}\label{3}

Recall the notation: $C$ is a curve with 
components $C_1,\ldots,C_p$ defined over $K$, 
an algebraically closed field, $\dot C$ is the smooth locus of $C$, 
and $\C/B$ is a smoothing 
of $C$ with regular total space, the smooth locus of which is $\dot\C$. 

Recall that 
$\Delta\subset\C^2$ is the diagonal subscheme, and that 
$\Delta_1:=\rho_1^{-1}(\Delta)$ and $\Delta_2:=\rho_2^{-1}(\Delta)$, 
where $\rho_i\:\C^3\to\C^2$ is the projection onto the product 
over $B$ of the $i$-th and last factors of $\C^3$ for $i=1,2$.

\subsection{The double product}\label{blbase}

Let $R$ and $S$ be reducible nodes of $C$. 
Let $C_i$ and $C_j$ be the distinct components containing $R$, and $C_k$ and 
$C_l$ those containing $S$.
Let $x_0=0$ and $x_1=0$ be local equations for the Cartier 
divisors $C_i$ and $C_j$ of $\C$ at $R$, respectively, and $y_0=0$ and $y_1=0$ 
local equations for $C_k$ and $C_l$ at $S$, respectively. 
If $R=S$, we assume $i=k$ and $j=l$, and $x_u=y_u$ for $u=0,1$.

Recall that  $B=\text{Spec}(K[[t]])$. We 
may assume that the completion of the local ring of $\C$ at $R$ is the 
quotient of $K[[t,x_0,x_1]]$ by $(x_0x_1-t)$, and that 
of $\C$ at $S$ is the quotient of $K[[t,y_0,y_1]]$ by $(y_0y_1-t)$. 
Thus the threefold $\C^2$ is described locally analytically at $(R,S)$ as 
the spectrum of the 
quotient of $K[[t,x_0,x_1,y_0,y_1]]$ by the ideal of the subscheme 
given by the equations
$$
\left\{\begin{array}{c}
x_0x_1=t,\\y_0y_1=t,
\end{array}\right.
$$
and so it can be seen as the cone in $\IA^4_K$, with coordinates 
$x_0,x_1,y_0,y_1$, given by $x_0x_1=y_0y_1$.

Thus $\C^2$ has a quadratic isolated singularity at $(R,S)$. This singularity 
can be resolved by blowing up $\C^2$ at $(R,S)$, at the cost of 
replacing the point by a quadric surface. However, the next proposition 
describes resolutions of this singularity that are more efficient 
and, more important, better for our purposes.

If $R=S$, the threefold $\C^2$ is given analytically at $(R,S)$ in the same 
way as above, and 
the diagonal is given by the equations $x_0-y_0=x_1-y_1=0$. 

\begin{Prop}\label{lemabasis} Let $R$ and $S$ be reducible nodes of $C$. Assume 
$R\in C_i\cap C_j$ and $S\in C_k\cap C_l$, for integers $i,j,k,l$ with 
$i\neq j$ and $k\neq l$. 
If $R=S$, assume $i=k$ and $j=l$. 
Let $\phi\col\wt\C^2\ra\C^2$ denote the 
blowup of $\C^2$ along $C_i\times C_l$, or along the diagonal if $R=S$. 
Put $E:=\phi^{-1}(R,S)$. Then the following statements hold:
\begin{enumerate}
\item $E$ is a smooth rational curve and $\wt\C^2$ is regular in a 
neighborhood of $E$. 
\item The strict transforms of $C_i\times C_l$ 
and $C_j\times C_k$ contain $E$, while those of 
$C_i\times C_k$ and $C_j\times C_l$ intersect $E$ transversally at 
a unique point, distinct for each transform.
\item If $R=S$, the strict transform of the diagonal contains $E$.
\item The composition 
$\wt\C^2\to \C$ of $\phi$ with the projection of $\C^2$ onto any of its 
factors is flat.
\end{enumerate}
\end{Prop} 

\begin{proof} Keep the notation prior to the statement of the 
proposition. Assume first that $\phi$ is the blowup along 
$C_i\times C_l$, whether $R=S$ or not. 
Here are the equations at $(R,S)$ of the products 
listed:
$$
\begin{array}{cc}
C_i\times C_k\colon&x_0=y_0=0\\
C_i\times C_l\colon&x_0=y_1=0\\
C_j\times C_k\colon&x_1=y_0=0\\
C_j\times C_l\colon&x_1=y_1=0
\end{array}
$$
To prove the statements of the lemma, we may pass to the completion of the 
local ring of $\C^2$ at $(R,S)$. So we may change $\C^2$ for 
the cone $Z$ in $\IA^4_K$, with coordinates $x_0,x_1,y_0,y_1$, given by 
$x_0x_1=y_0y_1$. 

The blowup is the subscheme $Y\subset\IA^4_K\times\IP^1_K$ given 
by 
$$
\left\{\begin{array}{c}
\alpha'x_0=\alpha y_1\\
\alpha'y_0=\alpha x_1\\
\end{array}\right.
$$
where $\alpha,\alpha'$ are homogeneous coordinates of $\IP^1_K$. The 
inverse image $E$ of the vertex of the cone in the blowup $Y\to Z$ 
is given 
by the equations $x_0=x_1=y_0=y_1=0$; thus $E$ is a smooth rational curve.

The strict transform of ${C_i\times C_l}$ is given by $x_0=0$ where 
$\alpha\neq0$ and by $y_1=0$ where $\alpha'\neq0$; the strict transform of 
${C_j\times C_k}$ is given by $y_0=0$ where $\alpha\neq0$ and by $x_1=0$ 
where $\alpha'\neq0$. Thus both transforms contain $E$.

On the other hand, the transform of ${C_i\times C_k}$ is given by $\alpha=0$, 
whence intersects $E$ transversally at a certain point. Analogously, 
the transform of ${C_j\times C_l}$ is given by $\alpha'=0$, whence intersects 
$E$ transversally at another point.

Where $\alpha'\neq0$, setting $u=\alpha/\alpha'$, $v=x_1$ and 
$w=y_1$, the blowup is given by the equations $x_0=uw$ and $y_0=uv$. 
On 
the other hand, where $\alpha\neq0$, setting $u=\alpha'/\alpha$, 
$v=y_0$ and $w=x_0$, the blowup is given by $x_1=uv$ and $y_1=uw$.
In any case, we have that the blowup is isomorphic to $\IA^3_K$, with 
coordinates $u$, $v$ and $w$, thus nonsingular. Furthermore, $\wt\C^2\to\C$ 
corresponds to the map $\IA^3_K\to\IA^2_K$ taking $(u,v,w)$ to $(uw,v)$ or 
$(uv,w)$ in case $\alpha'\neq 0$ and to $(w,uv)$ or $(v,uw)$ 
in case $\alpha\neq 0$. So, $\wt\C^2\to\C$ is flat on a neighborhood of $E$, 
and thus everywhere. 

As for the second case, the blowup along the diagonal, the equations of the 
blowup are the same as above, and 
thus the same results hold. Also, since $E$ is given by 
$x_0=x_1=y_0=y_1=0$, it follows that $E$ is contained in the 
strict transform of the diagonal.
\end{proof}

The same proposition 
can be used, {\it mutatis mutandis}, to describe the blowup 
along $C_j\times C_k$ (same results) and those along $C_i\times C_k$ and 
$C_j\times C_l$. Also, the results of the lemma are the same if we blow up 
along a product of 
two subcurves $X\times Y$, as long as 
$C_i\subseteq X$ but $C_j\not\subseteq X$, and $C_l\subseteq Y$ but 
$C_k\not\subseteq Y$.

\subsection{The triple product}\label{blowup3}

Keep the notation of Subsection \ref{blbase}. Let $T$ be another reducible 
node of $C$, not necessarily distinct from $R$ and $S$. Let $C_m$ and 
$C_n$ be the distinct components of $C$ containing $T$. Let $z_0=0$ and 
$z_1=0$ be local equations for the Cartier divisors $C_m$ and $C_n$ of 
$\C$ at $T$, respectively. If $T=R$ (resp. $T=S$) we set $z_u:=x_u$ 
(resp. $z_u:=y_u$) for $u=0,1$.

Let $\wt\C^2$ be the blowup of $\C^2$ along $C_i\times C_l$. Let $E$ be the 
``exceptional'' rational smooth curve on $\wt\C^2$ 
over $(R,S)$. Let $A$ (resp. $A'$) 
be the intersection of the strict transform 
of $C_i\times C_k$ (resp. $C_j\times C_l$) with $E$. Recall the 
proof of Proposition~\ref{lemabasis}. The strict transform of 
$C_i\times C_l$ is given at $A$ by the equation $y_1=0$ and at 
$A'$ by the equation $x_0=0$. That of ${C_j\times C_k}$ is 
given at $A$ by $x_1=0$ and at $A'$ by $y_0=0$. Both contain $E$. 
The divisor given by $t=0$ contains these strict transforms. Its 
residue is the strict transform of ${C_i\times C_k}$ at $A$, and the 
transform of ${C_j\times C_l}$ at $A'$. Thus $t=x_1y_1\alpha$ at $A$, with 
$\alpha=0$ being an equation for ${C_i\times C_k}$, and $t=x_0y_0\alpha'$ at 
$A'$, with $\alpha'=0$ being an equation for ${C_j\times C_l}$. Schematically,
\begin{equation}\label{typeI}
\begin{array}{c}
\begin{array}{c} u=\alpha\\ v=x_1\\w=y_1\end{array}
\quad\text{and}\quad
\begin{array}{l}
 x_0=uw\\ y_0=uv\end{array}\\
\mbox{(at $A$)}
\end{array}
\quad
\left|
\quad
\begin{array}{c}
\begin{array}{c} u=\alpha'\\ v=y_0\\w=x_0\end{array}
\quad\text{and}\quad
\begin{array}{l}
 x_1=uv\\ y_1=uw\end{array}\\
\mbox{(at $A'$)}
\end{array}\right.
\end{equation}

\bigskip 

\[
\begin{xy} <20pt,0pt>:
(0,0)*{\bullet}="a"; 
(0,-2)*{}="b";
(1.3,1.3)*{}="c";
(-2,0)*{}="d";
"a"+0;"b"+0**\dir{-};
"a"+0;"c"+0**\dir{-};
"a"+0;"d"+0**\dir{-};
"a"+(-1,-1)*{u=0}; 
"a"+(1,-0.5)*{v=0};
"a"+(-0.6,1)*{w=0};
(6,-0.8)*{\bullet}="f"; 
"f"+0;"f"+(2,0)**\dir{-};
"f"+0;"f"+(0,2)**\dir{-};
"f"+(-1.3,-1.3);"f"**\dir{-};
"f"+(1,1)*{u=0}; 
"f"+(0.5,-1)*{v=0};
"f"+(-1.3,0.5)*{w=0};
"f"+(-3,-2)*{\text{\bf Figure 1. } \text{Blowup at $(R,S)$}};
"f"+(-6.1,1.1)*{A};
"f"+(0.3,0.4)*{A'};
\end{xy}
\]
 
To summarize, at $A$ (resp. $A'$), we have the equation $t=uvw$, where $u=0$ 
is a local equation of the transform of $C_i\times C_k$ (resp. $C_j\times C_l$), 
and where $v=0$ and 
$w=0$ are local equations of the transforms of $C_j\times C_k$ and 
$C_i\times C_l$ respectively. The functions $u,v,w$ form a regular system of 
parameters of $\wt\C^2$ at $A$ (resp. $A'$). 

In case $R=S$, and $\wt\C^2$ is the blowup of $\C^2$ along the diagonal, 
exactly the same summary applies. In either case, the strict transform of 
the diagonal, 
whose equations on $\C^2$ are $x_0-y_0=x_1-y_1=0$, is given 
by $v-w=0$.

Thus, the following equations hold on $\wt\C^2\times_B\C$ at $(A,T)$ or 
$(A',T)$:
$$\left\{\begin{array}{c}
uvw=t\\z_0z_1=t
\end{array}\right.$$
Passing to the completion of the local ring of 
$\widetilde\C^2\times_B\C$ at $(A,T)$ or $(A',T)$, we may view the fourfold 
as the hypersurface of $\IA^5_K$, 
with coordinates 
$u,v,w,z_0,z_1$ given by 
$z_0z_1=uvw$. In particular, we see that $\widetilde\C^2\times_B\C$ is 
singular at $(A,T)$ and $(A',T)$;   in fact, it is singular all along 
$E\times\{T\}$.

In case $R=T$, the diagonal $\Delta_1$ passes through $(R,S,T)$. It is 
given at the point by the equations $z_0-x_0=z_1-x_1=0$. Its inverse image 
in $\wt\C^2\times_B\C$ is given at $(A,T)$ by $z_0-uw=z_1-v=0$ and at 
$(A',T)$ by $z_0-w=z_1-uv=0$. Likewise, if $S=T$, the diagonal $\Delta_2$ 
passes through $(R,S,T)$, given at the point by the equations 
$z_0-y_0=z_1-y_1=0$; its inverse image 
in $\wt\C^2\times_B\C$ is given at $(A,T)$ by $z_0-uv=z_1-w=0$ and at 
$(A',T)$ by $z_0-v=z_1-uw=0$. 

To resolve the singularities of $\wt\C^2\times_B\C$ along $E\times\{T\}$, we 
will use two blowups. The next lemma describes the effect of the two 
blowups in the local 
analytic setup of the hypersurface of $\IA^5_K$ described above.

\begin{Lem}\label{propblows}
Let $Y$ be the hypersuface of $\IA^5_K$, with coordinates $a,b,c,z_.,z_{..}$, 
given by $z_.z_{..}=abc$. Let $p\:Y\to \IA^3_K$ be the restriction of 
the projection
$$
(a,b,c,z_.,z_{..})\mapsto(a,b,c).
$$
Let $X:=p^{-1}(0)$. Let  $X_.$ and 
$X_{..}$ be the components of $X$ defined respectively by $z_.=0$ and $z_{..}=0$.
Let $\psi_1\col Y_1\ra Y$ be the blowup of\,\, $Y$ along the subscheme given 
by $z_.=a=0$ and $\psi_2\col Y_2\ra Y_1$ that of $Y_1$ along the 
strict transform 
of the subscheme given by $z_{..}=b=0$. Set $\rho:=p\psi_1\psi_2$ and 
$\widetilde X:=\rho^{-1}(0)$. Then:
\begin{enumerate}
\item $Y_2$ is smooth, $\rho$ is flat, and 
$\widetilde X$ is a nodal curve with four components, two, $\wt X_.$ and 
$\wt X_{..}$, isomorphic to 
$X_.$ and $X_{..}$ under $\psi_1\psi_2$, 
but not intersecting each other, and (a chain of) 
two rational, smooth, 
projective components $E_1$ and $E_2$ intersecting each other 
at a single point, the first meeting $\wt X_{..}$ transversally at a 
single point away from the second, but not meeting 
$\wt X_.$, the second meeting $\wt X_.$ transversally at a single point away 
from the first, but not meeting $\wt X_{..}$; see Figure 2.
\item Let $\lambda\:\IA^1_K\to\IA^3_K$ be the map sending $t$ to $(t,t,t)$, 
and form the Cartesian diagram:
$$
\begin{CD}
W_2 @>\lambda_2>> Y_2\\
@V\rho_\lambda VV @V\rho VV\\
\IA_K^1 @>\lambda >> \IA^3_K.
\end{CD}
$$ 
Then $W_2$ is a smooth surface. In addition, set 
$$
\wh X_.:=\lambda_2^{-1}(\wt X_.),\quad\wh X_{..}:=\lambda_2^{-1}(\wt X_{..}),\quad
\wh E_1:=\lambda_2^{-1}(E_1),\quad\wh E_2:=\lambda_2^{-1}(E_2).
$$
Then these inverse images are prime Cartier divisors of $W_2$ summing up 
to $\rho_{\lambda}^{-1}(0)$. Furthermore, 
the strict transform to $Y_2$ of the subscheme of $Y$ given by the 
equations heading each column of the following two tables is a Cartier 
divisor, and its pullback to $W_2$ is a still a Cartier divisor, a sum of 
$\wh X_.$, $\wh X_{..}$, $\wh E_1$ and $\wh E_2$ 
with the corresponding multiplicities 
specified in the first four entries of each column, plus, in the case of the 
second table only, a $\rho_\lambda$-flat prime divisor intersecting the fiber 
$\rho_{\lambda}^{-1}(0)$ transversally on the component specified in the fifth 
row:
\begin{center}
\begin{tabular}{|c||c|c|c|c|c|c|c|}
\hline 
      &$a=0$   &$b=0$   &$c=0$   &$a=0$   &$b=0$   &$c=0$  \\
      &$z_.=0$ &$z_.=0$ &$z_.=0$ &$z_{..}=0$ &$z_{..}=0$ &$z_{..}=0$\\\hline\hline
$\wh X_.$ &$1$ &$1$ &$1$ &$0$ &$0$ &$0$ \\\hline
$\wh X_{..}$ &$0$ &$0$ &$0$ &$1$ &$1$ &$1$ \\\hline
$\wh E_1$ & $1$   &  $0$   & $0$    & $0$    &  $1$   & $1$           \\\hline
$\wh E_2$ & $1$    &  $0$   & $1$   &  $0$   &  $1$  &  $0$           \\\hline
\end{tabular}
\end{center}
\begin{center}
\begin{tabular}{|c||c|c|c|c|c|c|c|}
\hline 
           &  $z_.=a$  &  $z_.=b$  &   $z_.=c$  & $z_.=ab$ & $z_.=ac$  & $z_.=bc$
\\
           &$z_{..}=bc$ &$z_{..}=ac$ &$z_{..}=ab$ &$z_{..}=c$ &$z_{..}=b$ &$z_{..}=a$
\\\hline
\hline
$\wh X_.$  &    $0$     &    $0$     &    $0$     &   $0$     &   $0$     &  $0$
\\\hline
$\wh X_{..}$&    $0$     &    $0$     &    $0$     &   $0$     &   $0$     &  $0$
\\\hline
$\wh E_1$      &    $0$     &    $1$     &    $1$     &   $0$     &   $0$     &  $1$
\\\hline
$\wh E_2$      &    $0$     &    $1$     &    $0$     &   $1$     &   $0$     &  $1$
\\\hline
           &   $\wh E_1$   &   $\wh E_1$   &   $\wh E_1$    &  $\wh E_2$   &  $\wh E_2$   &  $\wh E_2$
\\\hline
\end{tabular}
\end{center}
\end{enumerate}
\end{Lem}

\begin{proof}
By construction of the blowup, $Y_1$ is the subscheme of $Y\times\IP^1_K$ 
given in $\IA^5_K\times\IP^1_K$ by 
\begin{equation}\label{Y1}
\left\{\begin{array}{c}
\gamma_1'z_.=\gamma_1 a\\
\gamma_1'bc=\gamma_1 z_{..}
\end{array}\right.
\end{equation}
where $\gamma_1,\gamma_1'$ are homogeneous coordinates of $\IP^1_K$. 
Set $E_1=\psi_1^{-1}(0)$, thus given by the equations $a=b=c=z_.=z_{..}=0$. 
Clearly, $E_1$ is mapped isomorphically to $\IP^1_K$ under the second 
projection. Let 
$T_1$ and $T'_1$ be the points on $E_1$ given by $\gamma_1=0$ and $\gamma'_1=0$.
Note that $Y_1$ is isomorphic to $\IA^4_K$ and $p\psi_1$ is flat 
where $\gamma_1\neq0$. 
On the other hand, 
where $\gamma'_1\ne 0$, we see that $Y_1$ is singular along the points 
given by $z_.=z_{..}=b=c=\gamma_1=0$. In particular, $Y_1$ is singular at 
$T_1$. The strict transform of $X_.$ is given by $a=b=c=z_.=\gamma_1=0$ 
and intersects $E_1$ at $T_1$, while that of $X_{..}$ is given 
by $a=b=c=z_{..}=\gamma'_1=0$ and intersects $E_1$ transversally at $T'_1$.

Now, $Y_2$ is the blowup of $Y_1$ along a subscheme of codimension 1, whence 
$Y_2$ is isomorphic to $Y_1$ where $Y_1$ is smooth, in particular, where 
$\gamma_1\neq 0$. Over $\gamma'_1\neq 0$, 
setting $\gamma'_1:=1$, the blowup is the subscheme of 
$Y_1\times\IP^1$ given as a subset of $\IA^5_K\times\IA^1_K\times\IP^1_K$ by:
\begin{equation}\label{Y2}
 \left\{\begin{array}{c}
z_.=\gamma_1a\\
\gamma_2'z_{..}=\gamma_2 b\\
\gamma_2'c=\gamma_2 \gamma_1
\end{array}\right.
\end{equation}
where $\gamma_2,\gamma_2'$ are homogeneous coordinates of $\IP^1_K$. 
Notice that 
$Y_2$ is smooth and $\rho$ is flat. Let $E_2:=\psi_2^{-1}(T_1)$, thus given by 
the equations $a=b=c=z_.=z_{..}=\gamma_1=0$. 
Clearly, $E_2$ is mapped isomorphically 
to $\IP^1_K$ under the third projection. 
Let $T_2$ and $T'_2$ be the points on $E_2$ given by $\gamma_2=0$ and 
$\gamma'_2=0$. 
Abusing notation, we still denote by $E_1$ the strict transform of $E_1$ to 
$Y_2$ via $\psi_2$: It is given where $\gamma'_1\neq 0$ 
by the equations $a=b=c=z_.=z_{..}=\gamma_2=0$. 
We thus see that the ``new'' $E_1$ intersects $E_2$ 
transversally at $T_2$, which is the point over $T_1$ on the ``old'' $E_1$. 
The strict 
transform $\wt X_.$ of $X_.$ lies where $\gamma'_1\neq 0$, and is given by 
$a=b=c=z_.=\gamma_1=\gamma'_2=0$, and thus intersects $E_2$ transversally 
at $T'_2$. 

So, the proof of Statement 1 is complete, the curve $\wt X$ being 
depicted in Figure 2 below:
\[
\begin{xy} <30pt,0pt>:
(0,0)*{\bullet}="a"; 
(1.3,1.3)*{}="b";
(1.3,-1.3)*{}="c";
(3.5,1)*{}="d";
(3.5,-1)*{}="e";
"a"+(-0.2,-0.2);"b"+0**\dir{-};
"a"+(-0.2,0.2);"c"+0**\dir{-};
"b"+(-0.8,-0.3);"d"+0**\dir{-};
"c"+(-0.8,0.3);"e"+0**\dir{-};
"a"+(0.2,0.6)*{E_1}; 
"a"+(0.2,-0.5)*{E_2}; 
"d"+(0.4,0)*{\wt X_{..}}; 
"e"+(0.4,0)*{\wt X_.}; 
"b"+(-0.3,-0.3)*{\bullet};
"b"+(-0.3,-2.3)*{\bullet};
"e"+(-2.4,0.2)*{T'_2};
"e"+(-2.8,1)*{T_1=T_2};
"e"+(-2.3,1.8)*{T'_1};
"e"+(-1.2,-0.6)*{\text{\bf Figure 2. }\text{The curve $\wt X$}}
\end{xy}
\]

\smallskip

To prove Statement 2, we consider the covering of $Y_2$ by three open 
subschemes 
isomorphic to $\IA^4_K$, that isomorphic to the open subscheme 
of $Y_1$ given by 
$\gamma_1\neq 0$, and those given by $ \gamma'_1\gamma_2\neq 0$ and 
$\gamma'_1\gamma'_2\neq 0$. The following two tables list on each entry the 
equation(s) on the open subscheme of $Y_2$ given by the inequality heading the 
row of the strict transform of the subscheme of $Y$ given by the 
equations heading the column:
\begin{center}
\begin{tabular}{|c||c|c|c|c|c|c|c|}
\hline 
&$a=0$&$b=0$         &$c=0$&$a=0$&  $b=0$  & $c=0$ 
\\
& $z_.=0$ & $z_.=0$ & $z_.=0$ & $z_{..}=0$    & $z_{..}=0$ & $z_{..}=0$
\\\hline
$\gamma_1\neq 0$&$z_.=0$&$1=0$& $1=0$  & $\gamma'_1=0$ &   $b=0$   &   $c=0$
\\\hline
$\gamma'_1\gamma_2\neq 0$ &$a=0$&$\gamma'_2=0$&$c=0$&$1=0$&$z_{..}=0$&$1=0$
\\\hline
$\gamma'_1\gamma'_2\neq 0$&$a=0$&$1=0$  &$\gamma_1=0$&$1=0$&  $b=0$ &$\gamma_2=0$
\\\hline
\end{tabular}
\end{center}
\begin{center}
\begin{tabular}{|c||c|c|c|c|c|c|c|}
\hline 
&  $z_.=a$    &  $z_.=b$    & $z_.=c$   &    $z_.=ab$    &   $z_.=ac$  & $z_.=bc$
\\
& $z_{..}=bc$  & $z_{..}=ac$ & $z_{..}=ab$&   $z_{..}=c$    &   $z_{..}=b$ &$z_{..}=a$
\\\hline
$\gamma_1\neq 0$&$\gamma'_1=\gamma_1$&$z_.=b$&$z_.=c$&$\gamma'_1b=\gamma_1$&
$\gamma'_1c=\gamma_1$&$z_.=bc$
\\\hline
$\gamma'_1\gamma_2\neq 0$&$\gamma'_1=\gamma_1$& $z_{..}=ac$& $\gamma'_2a=\gamma_2$ & $z_{..}=c$ & $\gamma'_2=\gamma_2$ & $z_{..}=a$\\\hline
$\gamma'_1\gamma'_2\neq 0$ &$\gamma'_1=\gamma_1$& $z_.=b$& $\gamma'_2a=\gamma_2$ 
& $\gamma'_1b=\gamma_1$ & $\gamma'_2=\gamma_2$ & $z_{..}=a$\\\hline
\end{tabular}
\end{center}
In particular, the strict transforms to $Y_2$ 
of the subschemes of $Y$ given by the 
equations heading all the columns above are Cartier divisors. 

Recall the equations 
defining $E_1$, $E_2$, $\wt X_.$ and $\wt X_{..}$:
\begin{equation}\label{EEXX}
\left\{\begin{array}{c}
E_1\:    a=b=c=z_.=z_{..}=\gamma_2=0\\
E_2\:     a=b=c=z_.=z_{..}=\gamma_1=0\\
\wt X_.\: a=b=c=z_.=\gamma_1=\gamma'_2=0\\
\wt X_{..}\: a=b=c=z_{..}=\gamma'_1=0 
\end{array}\right.
\end{equation}
It must be noted that the equations for $\wt X_{..}$ are taken in 
$\IA^5_K\times\IP^1_K$, whence consist of one less equation 
than for the other three, which are 
taken in $\IA^5_K\times\IA^1_K\times\IP^1_K$, assuming $\gamma'_1\neq 0$. 
Of these other three, $E_1$ is the only curve that has a point with 
$\gamma'_1=0$: the image of $E_1$ in $Y_1$ is given by $a=b=c=z_.=z_{..}=0$.

Now, where $\gamma_1\neq 0$, thus setting $\gamma_1:=1$, viewing 
$Y_1$, and thus $Y_2$, as the subscheme of $\IA^6_K$ given by 
Equations \eqref{Y1}, it follows that $W_2$ is the subscheme of $\IA^4_K$ 
given by the equations
\begin{equation}\label{W1}
\left\{\begin{array}{c}
\gamma_1'z_.=t\\
\gamma_1't^2=z_{..}.
\end{array}\right.
\end{equation}
It follows that $W_2$ is smooth, isomorphic to $\IA^2_K$ with coordinates 
$\gamma'_1,z_.$. On the other hand, 
if $\gamma'_1\neq 0$, 
we may view $Y_2$ as the subscheme of $\IA^7_K$ given by Equations \eqref{Y2}, 
setting $\gamma_2:=1$ or $\gamma'_2:=1$. It follows that $W_2$ is the 
subscheme of $\IA^5_K$ given by
\begin{equation}\label{W2}
\left\{\begin{array}{c}
z_.=\gamma_1t\\
t=\gamma'_2z_{..}\\
\gamma_1=\gamma'_2t
\end{array}\right.\quad\text{and}\quad
\left\{\begin{array}{c}
z_.=\gamma_1t\\
z_{..}=\gamma_2t\\
t=\gamma_2\gamma_1,
\end{array}\right.
\end{equation}
the first set of equations where $\gamma_2\neq 0$ and the second where 
$\gamma'_2\neq 0$. So $W_2$ is smooth, isomorphic to $\IA^2_K$ with 
coordinates $\gamma'_2,z_{..}$ in the first case and $\gamma_1,\gamma_2$ in 
the second.

The following first 
three tables list the restrictions to $W_2$ of the equations 
listed in the last two tables, whereas the fourth lists on each entry 
the equation(s) on the open subscheme of $W_2$ given by the inequality 
heading the row of the intersection with $W_2$ of the subscheme of $Y_2$ 
heading the column:
\begin{center}
\begin{tabular}{|c||c|c|c|c|c|c|c|}
\hline 
&$a=0$&$b=0$         &$c=0$&$a=0$&  $b=0$  & $c=0$ 
\\
& $z_.=0$ & $z_.=0$ & $z_.=0$ & $z_{..}=0$    & $z_{..}=0$ & $z_{..}=0$
\\\hline
$\gamma_1\neq 0$&$z_.=0$&$1=0$&$1=0$&$\gamma'_1=0$&$\gamma'_1z_.=0$&
$\gamma'_1z_.=0$
\\\hline
$\gamma'_1\gamma_2\neq 0$ &$\gamma'_2z_{..}=0$&$\gamma'_2=0$&$\gamma'_2z_{..}=0$&
$1=0$&$z_{..}=0$&$1=0$
\\\hline
$\gamma'_1\gamma'_2\neq 0$&$\gamma_1\gamma_2=0$&$1=0$&$\gamma_1=0$&$1=0$&
$\gamma_1\gamma_2=0$ &$\gamma_2=0$
\\\hline
\end{tabular}
\end{center}
\begin{center}
\begin{tabular}{|c||c|c|c|}
\hline 
&  $z_.=a$    &  $z_.=b$    & $z_.=c$  
\\
& $z_{..}=bc$  & $z_{..}=ac$ & $z_{..}=ab$
\\\hline
$\gamma_1\neq 0$&$\gamma'_1=1$&$z_.=\gamma'_1z_.$&$z_.=\gamma'_1z_.$
\\\hline
$\gamma'_1\gamma_2\neq 0$&$1=(\gamma'_2)^2z_{..}$& $z_{..}=(\gamma'_2)^2z_{..}^2$&
$(\gamma'_2)^2z_{..}=1$
\\\hline
$\gamma'_1\gamma'_2\neq 0$ &$1=\gamma_1$& $\gamma_1^2\gamma_2=\gamma_1\gamma_2$&
$\gamma_1\gamma_2=\gamma_2$
\\\hline
\end{tabular}
\end{center}
\begin{center}
\begin{tabular}{|c||c|c|c|c|c|c|c|}
\hline 
&  $z_.=ab$    &   $z_.=ac$  & $z_.=bc$
\\
&  $z_{..}=c$    &   $z_{..}=b$ &$z_{..}=a$
\\\hline
$\gamma_1\neq 0$&
$(\gamma'_1)^2z_.=1$&$(\gamma'_1)^2z_.=1$&$z_.=(\gamma'_1)^2z_.^2$
\\\hline
$\gamma'_1\gamma_2\neq 0$&$z_{..}=\gamma'_2z_{..}$&$\gamma'_2=1$&
$z_{..}=\gamma'_2z_{..}$
\\\hline
$\gamma'_1\gamma'_2\neq 0$ &$\gamma_1\gamma_2=\gamma_1$ & $1=\gamma_2$ & 
$\gamma_1\gamma_2^2=\gamma_1\gamma_2$
\\\hline
\end{tabular}
\end{center}
\begin{center}
\begin{tabular}{|c||c|c|c|c|}
\hline 
&  $\wt X_.$  &  $\wt X_{..}$  &  $E_1$  &  $E_2$
\\\hline
$\gamma_1\neq 0$&$1=0$&$\gamma'_1=0$&$z_.=0$&$1=0$
\\\hline
$\gamma'_1\gamma_2\neq 0$ &$\gamma'_2=0$&$1=0$&$1=0$&$z_{..}=0$
\\\hline
$\gamma'_1\gamma'_2\neq 0$&$1=0$&$1=0$&$\gamma_2=0$&$\gamma_1=0$
\\\hline
\end{tabular}
\end{center}
So $\wh X_.$, $\wh X_{..}$, $\wh E_1$ and $\wh E_2$ are indeed prime Cartier 
divisors of $W_2$, and the intersections with $W_2$ of the strict transforms to 
$Y_2$ of the subschemes of $Y$ given by the equations heading the columns of 
the above first three tables are Cartier. The two tables in Statement 2 follow 
from the above four tables by comparing the equations in the entries.
\end{proof}

Observe that the strict transform to $Y_1$ of the subscheme of $Y$ defined by 
$z_{..}=a=0$ is a Cartier divisor, given by $\gamma'_1=0$. On the other hand, 
the strict transform to $Y_1$ of the subscheme defined by 
$z_.=b=0$ is given by $\gamma_1=b=0$ and that of the subscheme defined 
by $z_.=c=0$ is given by $\gamma_1=c=0$. Thus the blowups of $Y_1$ along the 
strict transform of the subscheme given by $z_.=b=0$ and along that of the 
subscheme given by $z_{..}=c=0$ are equal. And so are the blowups along 
the transform of the subscheme given by $z_.=c=0$ and along that of the 
subscheme given by $z_{..}=b=0$. Thus the same proposition can be used, 
\emph{mutatis mutandis}, to describe all ``possible'' blowups.

\section{Blowups: global analysis}\label{4}

Recall that $C$ is a curve with components $C_1,\ldots,C_p$ defined over $K$, 
an algebraically closed field, $\dot C$ is the smooth locus of $C$, 
and $\C/B$ is a regular smoothing of $C$ with regular total space, the 
smooth locus of which is $\dot\C$. 

Recall that 
$\Delta\subset\C^2$ is the diagonal subscheme, and that 
$\Delta_1:=\rho_1^{-1}(\Delta)$ and $\Delta_2:=\rho_2^{-1}(\Delta)$, 
where $\rho_i\:\C^3\to\C^2$ is the projection onto the product 
over $B$ of the $i$-th and last factors of $\C^3$ for $i=1,2$.

\subsection{The double product} 

We will use the symbol $[X,Y]$ to describe the blowup of 
(a modification by blowups of) $\C^2$ along (the strict transform of) 
$X\times Y$, for proper subcurves $X,Y\subseteq C$. If $X$ and $Y$ are 
components, $X=C_i$ and $Y=C_l$, we shall use the shorter notation $[i,l]$ 
for $[C_i,C_l]$. A 
modification of $\C^2$ by such blowups can thus be described as a sequence 
\begin{equation}\label{b2}
[\Delta],[X_1,Y_1],[X_2,Y_2],\dots,[X_u,Y_u],
\end{equation}
where $[\Delta]$ stands for blowup along the diagonal. By 
Proposition \ref{fl}, the 
blowup along the diagonal adds a projective line on top of 
each pair $(R,R)$ of identical nodes, whereas by Proposition \ref{lemabasis} 
the blowup $[X_i,Y_i]$ 
adds a projective line on top of each pair 
$(R,S)\in(X_i\cap X'_i)\times(Y_i\cap Y'_i)$, as long as a projective line 
has not been created before, in a preceding blowup, over the same pair. We 
say the projective lines created are \emph{exceptional}. On the 
exceptional line over each pair $(R,S)$ of reducible nodes there are two 
\emph{distinguished points}, each being on the intersection of the 
strict transforms of three distinct 
products of components of $C$ containing $R$ and $S$.

By considering a sequence of blowups long and varied enough, always starting 
with $[\Delta]$, we obtain a 
scheme $\wt\C^2$ which is nonsingular away from the points over pairs 
$(R,S)$ of distinct nodes $R$ and $S$ of $C$ where either $R$ or $S$ is 
irreducible. At any rate, if this is the case, since the product 
$X\times Y$ of subcurves 
$X$ and $Y$ of $C$ is Cartier away from $(X\cap X')\times(Y\cap Y')$, its 
strict transform to $\wt\C^2$ is Cartier. 
If this is achieved by the sequence 
of blowups, that is, if the strict transforms of all the products   $X\times Y$ of subcurves $X$ and $Y$ of $C$ are Cartier, we call $\wt\C^2$ a \emph{good partial desingularization} of 
$\C^2$.  

\begin{Exa}\label{exac2} Assume $C$ consists of four components 
$C_1,C_2,C_3,C_4$ such that 
$$
\begin{array}{cc}
C_1\cap C_2=&\{N_1\},\\
C_1\cap C_3=&\{N_2\},\\
C_1\cap C_4=&\emptyset,\\
C_2\cap C_3=&\{N_3,N_4,N_5\},\\
C_2\cap C_4=&\{N_6\},\\
C_3\cap C_4=&\{N_7,N_8\}.
\end{array}
$$ 
In addition, suppose $C_2,C_3,C_4$ are smooth but $C_1$ has a node $N_0$. 
The first blowup in the sequence
$$
[\Delta],[4,4],[2,2],[1,2],[2,1],[2,4],[4,2],[1,4],[4,1]
$$
adds a $\IP^1_K$ over each pair $(N_i,N_i)$ for $i=0,\dots,8$, the second 
blowup adds a $\IP^1_K$ over each pair $(N_i,N_j)$ for distinct $i,j$ such that 
$6\leq i,j\leq 8$, the third adds a $\IP^1_K$ over each pair 
$(N_i,N_j)$ for distinct $i,j$ such that $i,j=1,3,4,5,6$, the fourth 
adds a $\IP^1_K$ over each pair $(N_2,N_j)$ for $j=1,3,4,5,6$, the fifth 
adds a $\IP^1_K$ over each pair $(N_i,N_2)$ for $i=1,3,4,5,6$, the sixth 
adds a $\IP^1_K$ over each pair $(N_i,N_j)$ with $i=1,3,4,5$ and $j=7,8$, the 
seventh adds a $\IP^1_K$ over each pair $(N_i,N_j)$ with $i=7,8$ and 
$j=1,3,4,5$, the eigth adds a $\IP^1_K$ over $(N_2,N_7)$ and $(N_2,N_8)$, while 
the last adds a $\IP^1_K$ over $(N_7,N_2)$ and $(N_8,N_2)$. 

The above sequence desingularizes $\C^2$ at all points not of the 
form $(N_0,N_i)$ or $(N_i,N_0)$ for $i=1,\dots,8$. The strict transforms of 
the diagonal and all the products $C_i\times C_j$ are Cartier divisors in the 
resulting scheme. We have obtained a good partial desingularization 
of $\C^2$. Notice that the resulting scheme would be different if 
we had exchanged the position of $[1,4]$ with that of $[1,2]$ in the 
sequence of blowups. As it is, 
the resulting $\IP^1_K$ over $(N_2,N_6)$ is contained in the strict 
transform of $C_1\times C_2$, whereas with the exchange the $\IP^1_K$ would 
intersect the transform transversally. The new resulting scheme would 
nonetheless still be a good partial desingularization 
of $\C^2$.
\end{Exa}

Recall that $\N(C)$ is the collection of reducible nodes of $C$. 

\begin{Def} For each $i,k\in\{1,\dots,p\}$, let $\N_{i,k}(C)$ 
denote the subset of $\N(C)^2$ containing pairs of nodes $(R,S)$ such that 
$R\in C_i$ and $S\in C_k$. Let $\phi\:\wt\C^2\to\C^2$ be a  good 
partial desingularization. Denote by $\N_{i,k}(\phi)$ the subset of 
$\N_{i,k}(C)$ containing the pairs $(R,S)$ such that $\phi^{-1}(R,S)$ lies 
on the strict transform of $C_i\times C_k$. 
\end{Def}

\begin{Prop} Let $\phi\:\wt\C^2\to\C^2$ be a good  
partial desingularization given by 
$$
[\Delta],[X_1,Y_1],[X_2,Y_2],\dots,[X_u,Y_u].
$$
Let $i,k\in\{1,\dots,p\}$ and $R,S\in C$. 
Then $(R,S)\in\N_{i,k}(\phi)$ if and only if $R$ and $S$ are reducible nodes 
of $C$ with $R\in C_i$ and $S\in C_k$ and either of the following two 
statements hold:
\begin{enumerate}
\item $R=S$ and $i\neq k$.
\item $R\neq S$ and the first integer $v$ such that $R\in X_v\cap X'_v$ and 
$S\in Y_v\cap Y'_v$ satisfies either $C_i\subseteq X_v$ and $C_k\subseteq Y_v$ 
or $C_i\not\subseteq X_v$ and $C_k\not\subseteq Y_v$.
\end{enumerate}
\end{Prop}

\begin{proof} Follows easily from Proposition \ref{lemabasis}.
\end{proof}

Recall that $C_R$ is the curve obtained from $C$ 
by replacing the node $R$ by a 
smooth rational curve, and that $C(1)$ is the one obtained by 
replacing each reducible node of $C$ by a smooth rational curve; see 
Subsection \ref{preli}.

\begin{Prop}\label{lemabasis2}
Let $\phi\:\wt\C^2\to\C^2$ be a good partial 
desingularization. Let 
$\rho\:\wt\C^2\to\C$ denote its composition with 
the first projection $p_1\:\C^2\to\C$. 
Let $R\in C$ and $X:=\rho^{-1}(R)$. Let 
$\mu\:X\to C$ be the 
restriction to $X$ of $\phi$ composed with the second projection 
$p_2\:\C^2\to\C$. Then the following 
statements hold:
\begin{enumerate}
\item $\rho$ is flat.
\item $\wt\C^2$ is regular along each smooth rational curve of $X$ 
contracted by $\mu$.
\item If $R$ is not a node of $C$ then $\mu$ is an isomorphism.
\item If $R$ is an irreducible node of $C$ then $X$ is 
$C$-isomorphic to $C_R$. 
\item If $R$ is a reducible node of $C$ then $X$ is $C$-isomorphic to $C(1)$.
\end{enumerate}
Furthermore, for each $i,k\in\{1,\dots,p\}$, let 
$D_{i,k}$ denote the strict transform to $\wt\C^2$ 
of $C_i\times C_k$. Let
$$
D=\sum_{i,k}w_{i,k}D_{i,k}
$$
for given integers $w_{i,k}$. Then, if $R$ is a reducible node of $X$, 
the restriction $\O_{\wt\C^2}(D)|_X$ is a twister of $X$. More specifically, 
for each $i=1,\dots,p$, let $\wh C_i$ be the 
strict transform to $X$ of $C_i$ under $\mu$, and for each reducible node 
$S$ of $C$, let $E_S:=\mu^{-1}(S)$. Then 
$$
\O_{\wt\C^2}(D)|_X\cong\O_X(\sum_{k=1}^p a_k\wh C_k+\sum_{S\in\N(C)}b_SE_S),
$$
where $a_k:=\sum_iw_{i,k}$, the sum over the two
$i$ such that $R\in C_i$, and $b_S:=\sum_{i,k}w_{i,k}$, the sum over
the two pairs $(i,k)$ such that $(R,S)\in\N_{i,k}(\phi)$.
\end{Prop}

\begin{proof} Since flatness is a local property, Statement~1 follows 
directly from Propositions~\ref{fl} and \ref{lemabasis}. 

Now, $\C^2$ is regular except at the pairs of nodes of $C$. Thus 
Statement 3 holds. Furthermore, suppose $R$ is a node of $C$. Then, for 
each node $S\in C$, if $S\neq R$ then $\Delta$ is Cartier at $(R,S)$, and if 
$R$ or $S$ is irreducible then every product $Y\times Z$ of subcurves 
of $C$ is Cartier at $(R,S)$. So $\phi$ is an isomorphism over a 
neighborhood of $(R,S)$ except if $S=R$ or both $R$ and $S$ are reducible 
nodes of $C$.

Suppose first that $R$ is irreducible. Then $\mu$ is an isomorphism 
except over $R$. Now, over a neighborhood of $(R,R)$, 
we have that $\wt\C^2$ is isomorphic 
to $\IP_{\C^2}(\I_\Delta)$. Thus $\mu^{-1}(R)$ is a smooth rational curve, 
$X$ is $C$-isomorphic to $C_R$ and 
$\wt\C^2$ is regular along $\mu^{-1}(R)$ by 
Proposition \ref{fl}. Statement 4 and part of Statement 2 are 
proved.     

Suppose now that $R$ is reducible. Then $\mu$ is an isomorphism 
except over reducible nodes of $C$. So, let $S\in\N(C)$. 
Recall the notation introduced before Proposition~\ref{lemabasis}. 
We may assume without loss of generality that 
$(R,S)\in\N_{i,l}(\phi)$. So, to describe $\mu$ on a neighborhood of $E_S$, 
we may assume $\phi$ is simply the blowup $[i,l]$. 

Then $E_S$ is a smooth rational curve and $\wt\C^2$ is regular along it 
by Proposition~\ref{lemabasis}. The proof of Statement 2 is complete. 
Moreover, recall the computation done in the proof of 
Proposition~\ref{lemabasis}. The fiber $X$ corresponds to the subscheme 
of $Y$ given by $x_0=x_1=0$. Its equations in $\IA^4_K\times\IP^1_K$ are thus
$$
\alpha'y_0=0,\quad \alpha y_1=0,\quad x_0=0,\quad x_1=0.
$$
This subscheme is the union of a projective line, given by $x_0=x_1=y_0=y_1=0$, 
and two disjoint affine lines, given by 
$x_0=x_1=y_0=\alpha=0$ and $x_0=x_1=y_1=\alpha'=0$, intersecting 
transversally the projective line at distinct points. The first affine line 
corresponds to the strict transform of $C_k$ and the second to that of $C_l$. 
This proves Statement 5. 

As for the last statements, 
consider the map $\lambda\:B\to\C$ sending the special 
point of $B$ to $R$ and induced by the ring homomorphism
$$
\frac{K[[x_0,x_1,t]]}{(x_0x_1-t)}\lra K[[t]]
$$
which sends $x_i$ to $t$ for $i=1,2$, and thus $t$ to $t^2$. Form the 
Cartesian diagram:
$$
\begin{CD}
\W @>\xi >> \wt\C^2\\
@V\rho_\lambda VV @V\rho VV\\
B @>\lambda >> \C.
\end{CD}
$$ 
Then $\rho_\lambda$ is a regular smoothing of $X$.

Indeed, the special fiber of $\rho_\lambda$ is isomorphic to $X$ under 
$\xi$, hereby identified with $X$. Let $\wt C$ denote the 
strict transform of $C$ under $\mu$. 
We need to show that $\W$ is regular along 
$E_S\cap\wt C$ for each $S\in\N(C)$. Again, 
recall the notation introduced before Proposition~\ref{lemabasis} and 
that introduced in its proof. We may assume that 
$(R,S)\in\N_{i,l}(\phi)$ and that $\phi$ is simply the blowup $[i,l]$. Thus 
$\W$ corresponds to the base change of $Y$ under the diagram
$$
\begin{CD}
@. \IA^4_K\times\IP^1_K\\
@. @VV((x_0,x_1,y_0,y_1),(\alpha:\alpha'))\mapsto(x_0,x_1)V\\
\IA^1_K @>t\mapsto(t,t)>> \IA^2_K,
\end{CD}
$$
that is, to the subscheme $W$ of $\IA^3_K\times\IP^1_K$, with coordinates 
$((t,y_0,y_1),(\alpha:\alpha'))$, given by
$$
\left\{\begin{array}{c}
\alpha't=\alpha y_1\\
\alpha' y_0=\alpha t\\
\end{array}\right.
$$
Since $W$ is smooth, so is $\W$ along $E_S$.

It follows that $\xi^{-1}(D)$ is a Cartier divisor supported on 
the special fiber of $\rho_\lambda$, and thus $\O_{\wt C^2}(D)|_X$ is a 
twister of $X$. 

More specifically, in order to describe $\O_{\wt C^2}(D)|_X$, 
we need only describe the pullbacks to $\W$ of 
$D_{i,k}$, $D_{i,l}$, $D_{j,k}$ and $D_{j,l}$. We need only do so in a neighborhood 
of $E_S$ for each $S\in\N(C)$. Thus we need 
only describe the pullbacks of the corresponding divisors to $W$ for 
$\alpha\neq 0$ and $\alpha'\neq 0$. 
If $\alpha'\neq 0$, then 
$\W$ corresponds to the subscheme of $\IA^4_K$ given 
by $y_0=\alpha t$ and $t=\alpha y_1$, and the pullbacks 
to the divisors given by $\alpha=0$, $y_1=0$, $t=0$ and $\alpha'=0$, 
respectively. On the other hand, if $\alpha\neq 0$, then 
$\W$ corresponds to the subscheme of $\IA^4_K$ given by 
$y_1=\alpha' t$ and $t=\alpha' y_0$, and the pullbacks 
to the divisors given by $\alpha=0$, $t=0$, $y_0=0$ and 
$\alpha'=0$, respectively. Thus
\begin{align*}
D_{i,k}|_\W=&\cdots+\wh C_k+0E_S+0\wh C_l+\cdots,\\
D_{i,l}|_\W=&\cdots+0\wh C_k+E_S+\wh C_l+\cdots,\\
D_{j,k}|_\W=&\cdots+\wh C_k+E_S+0\wh C_l+\cdots,\\
D_{j,l}|_\W=&\cdots+0\wh C_k+0E_S+\wh C_l+\cdots.
\end{align*}
Use now that $(R,S)\in\N_{i,l}(\phi)\cap\N_{j,k}(\phi)$ but 
$(R,S)\not\in\N_{i,k}(\phi)\cup\N_{j,l}(\phi)$ to conclude.
\end{proof}

\subsection{The triple product}\label{subsintersect}

The blowup of (any modification of) $\C^2$ along 
(the strict transform of) the product $X\times Y$ of two subcurves 
$X$ and $Y$ of $C$ is denoted by $[X,Y]$. 
We will also denote the blowup of (any 
modification of) $\C^3$ along 
(the strict transform of) the product $X\times Y\times Z$ of three 
subcurves $X$, $Y$ and $Z$ of $C$ by $[X,Y,Z]$. If $X$, $Y$ and $Z$ are 
components, say $X=C_i$, $Y=C_l$ and $Z=C_m$, we shall use the shorter 
notation $[i,l,m]$ for $[C_i,C_l,C_m]$. We will consider a partial 
desingularization of $\C^3$ consisting of the base change of the sequence 
of blowups in \eqref{b2}, resulting in a map 
$\wt\C^2\to\C^2$, followed by a sequence of blowups of $\wt\C^2\times_B\C$ 
along strict transforms of products of three subcurves of $C$. Symbolically, 
the partial desingularization of $\C^3$ 
can be described by a sequence of the form:
\begin{equation}\label{b3}
[\Delta], [X_1,Y_1],\dots,[X_u,Y_u],[X_{u+1},Y_{u+1},Z_{u+1}],
\dots,[X_{u+v},Y_{u+v},Z_{u+v}].
\end{equation}

If the sequence of blowups in \eqref{b3} is long and varied enough, we 
obtain a good partial desingularization $\wt\C^2$ of $\C^2$, and the 
resulting modification $\wt\C^3$ of $\wt\C^2\times_B\C$ 
will be such that no other blowup along the (strict transform of a) product of three 
subcurves of $C$ affects it, that is, such that 
the strict transforms to $\wt\C^3$ of 
all the products $X\times Y\times Z$ of 
subcurves $X$, $Y$ and $Z$ of $C$ are Cartier. In this case, we call 
$\wt\C^3$ a \emph{good partial desingularization} of 
$\wt\C^2\times_B\C$. 

Recall the natural subscheme $F_2\subset\mathbb P_{\C^2}(\I_{\Delta|\C^2})\times_B\C$ 
defined in Subsection~\ref{flag}. The map $\phi$ factors through the
structure map $\mathbb P_{\C^2}(\I_{\Delta|\C^2})\to\C^2$, whence we
may consider the strict transform of $F_2$ to $\wt\C^3$.

\begin{Prop}\label{propblows2} 
Let $\phi\:\wt\C^2\to\C^2$ and $\psi\:\wt\C^3\to\wt\C^2\times_B\C$ 
be good partial desingularizations. Let $\rho=p_1\circ\psi$, where 
$p_1\:\wt\C^2\times_B\C\to\wt\C^2$ is the projection. Then $\rho$ is
flat. In addition, let $A$ be a closed point of 
$\wt\C^2$ and set $(R_1,R_2):=\phi(A)$. Put 
$X:=\rho^{-1}(A)$. Let $\mu\:X\to C$ be the restriction to $X$ of 
$\psi$ followed by the second projection $p_2\:\wt\C^2\times_B\C\to\C$. 
Then the following statements hold:
\begin{itemize}
\item[1.] If neither $R_1$ nor $R_2$ is a reducible node 
of $C$, then $\mu$ is an isomorphism.
\item[2.] If only one between $R_1$ and $R_2$ is a reducible node of $C$, 
or if both are but $A$ is not one of the two distinguished points of 
$\phi^{-1}(R_1,R_2)$, then $X$ is $C$-isomorphic to $C(1)$.
\item[3.] If $R_1$ and $R_2$ are reducible nodes of $C$ and $A$ 
is one of the two distinguished points of 
$\phi^{-1}(R_1,R_2)$, then $X$ is $C$-isomorphic to $C(2)$. 
\end{itemize}
Finally, let
$\wt\Delta_1$, $\wt\Delta_2$ and $\wt F_2$ denote the strict
transforms of $\Delta_1$, $\Delta_2$ and $F_2$ to $\wt\C^3$. Let
$T\in C$. Then
the following statements hold:
\begin{itemize}
\item[4.] For $i=1,2$, the strict transform $\wt\Delta_i$ is $\rho$-flat along 
$\mu^{-1}(T)$ unless, possibly, $R_i$ is a reducible node of $C$
and $T=R_i$.
\item[5.] For $i=1,2$, the strict transform $\wt\Delta_i$ is Cartier along 
$\mu^{-1}(T)$ unless $R_i$ is an irreducible node of $C$ and $T=R_i$.
\item[6.] $\wt F_2$ is $\rho$-flat along $\mu^{-1}(T)$ 
unless, possibly, $T$ is a reducible node of $C$ and 
either $R_1=T$ or $R_2=T$.
\item[7.]
  $\I_{\wt{F_2}|\wt\C^3}=\I_{\wt\Delta_1|\wt\C^3}\I_{\wt\Delta_2|\wt\C^3}=
\I_{\wt\Delta_1|\wt\C^3}\otimes\I_{\wt\Delta_2|\wt\C^3}$ in a
neighborhood of $\mu^{-1}(T)$, unless $T$ is an irreducible node of
$C$ and $R_1=R_2=T$.
\end{itemize}
\end{Prop}

\begin{proof} Let $R_1,R_2,T\in C$ and $A\in\phi^{-1}(R_1,R_2)$. 
Let $X$ and $\mu$ be as in the statement of the proposition. 
First of all, notice that Statement~7 follows from Statement~5. Indeed,
$$
\I_{\wt{F_2}|\wt\C^3}=\I_{\wt\Delta_1|\wt\C^3}\I_{\wt\Delta_2|\wt\C^3}=
\I_{\wt\Delta_1|\wt\C^3}\otimes\I_{\wt\Delta_2|\wt\C^3}
$$
in a neighborhood of $\mu^{-1}(T)$ if either $\wt\Delta_1$ or
$\wt\Delta_2$ is Cartier along $\mu^{-1}(T)$. 

To prove the remaining statements, notice first that all 
the products $X\times Y\times Z$ 
of proper subcurves 
$X$, $Y$ and $Z$ of $C$ are Cartier at $(R_1,R_2,T)$ 
if at most one among $R_1$, $R_2$ and $T$ is a reducible node of $C$. 
Statement 1 follows. Furthermore, the strict transform of each 
product $X\times Y\times Z$ to 
$\wt\C^2\times_B\C$ is the strict transform of 
$D\times Z$, where $D$ is the strict transform to $\wt\C^2$ of 
$X\times Y$. Since $\phi$ is good, any such $D$ is Cartier, and hence
the strict transforms of all the products $X\times Y\times Z$ to 
$\wt\C^2\times_B\C$ are Cartier at $(A,T)$ if $T$ 
is not a reducible node of $C$. Thus $\psi$ fails to be an isomorphism
over $(A,T)$ and $\rho$ fails to be flat in a neighborhhod of
$\mu^{-1}(T)$ only if $T$ and $R_i$ are reducible nodes of $C$ for $i=1$ or
$i=2$.

Now, by base change, the strict transforms of $\Delta_1$, $\Delta_2$ and
$F_2$ to $\wt\C^2\times_B\C$ are $p_1$-flat. Thus, $\wt\Delta_1$ or $\wt\Delta_2$ or
$\wt F_2$ is $\rho$-flat along $\mu^{-1}(T)$ unless, possibly, $T$ is a reducible
node of $C$. However, $\wt\Delta_i$ does not intersect $\mu^{-1}(T)$
unless $R_i=T$, for $i=1,2$.  Statements 4 and 6 follow. 

It follows as
well that $\wt\Delta_i$ is trivially Cartier along $\mu^{-1}(T)$ if
$T\neq R_i$, for $i=1,2$. If $T=R_i$ and $R_i$ is a nonsingular point
of $C$, then $\Delta_i$ is Cartier at $(R_1,R_2,T)$ and both $\psi$ is
an isomorphism over $(A,T)$ and $\phi$ is an isomorphism
over $(R_1,R_2)$. So $\wt\Delta_i$ is Cartier along $\mu^{-1}(T)$ if $R_i$ is
a nonsingular point of $C$.

If $T=R_i$ and $R_i$ is an irreducible node of $C$, then $\psi$ is
an isomorphism over $(A,T)$. Furthermore, $\phi$ is an isomorphism
over $(R_1,R_2)$ unless $R_1=R_2$. Since $\Delta_i$ fails to be
Cartier at $(R_1,R_2,T)$, it follows that $\wt\Delta_i$ fails to be
Cartier along $\mu^{-1}(T)$ unless $R_1=R_2$. However, even in this case,
$\wt\Delta_i$ fails to be Cartier along $\mu^{-1}(T)$ as well, as a
local reasoning, similar to that done before Lemma~\ref{propblows},
shows. 

To finish the proof of the proposition we
may now assume, without loss of generality, that $R_2$ is a reducible
node of $C$. We need only 
describe the structure of $\wt\C^3$ locally around $\mu^{-1}(T)$ when $T$ is a 
reducible node of $C$, and show 
that $\rho$ is flat and $\wt\Delta_2$ is Cartier 
along $\mu^{-1}(T)$.

Assume first that $R_1$ is not a reducible node of $C$. Then $\phi$ is 
an isomorphism over $(R_1,R_2)$ and 
$\psi$ is locally around $X$ of the form 
$1_\C\times_B\varphi\:\C\times_B\wh\C^2\to\C^3$, 
where $\varphi\:\wh\C^2\to\C^2$ is a good partial desingularization. 
Thus $X$ is $C$-isomorphic to $C(1)$ and $\rho$ is flat along $X$ by 
Proposition \ref{lemabasis2}. 

If $R_1$ is not an irreducible node of 
$C$ either, it follows as well from the same proposition that 
$\wt\C^3$ is regular along each smooth rational curve of $X$ 
contracted by $\mu$. Now, $\wt\Delta_2$ intersects $X$ 
along $\mu^{-1}(R_2)$, which is a smooth rational curve of $X$ because 
$R_2$ is reducible. Since $\wt\C^3$ is regular along $\mu^{-1}(R_2)$, it 
follows that $\wt\Delta_2$ is Cartier along $X$. 

We may now assume that $R_1$ is a reducible node of $C$. Assume as well
that $T$ is a 
reducible node of $C$. Suppose first 
that $A$ is one of the distinguished points of $\phi^{-1}(R_1,R_2)$. Then 
$p_1$ looks locally analytically over $(A,T)$ like the map $p$ in 
Lemma~\ref{propblows} over the origin, and $\psi$ like $\psi_1\psi_2$. 
In this case, the flatness of $\rho$ along $\mu^{-1}(T)$ and
Statements 3 and 5 follow from that lemma.

Suppose now that $A$ is not any of the distinguished points of
$\phi^{-1}(R_1,R_2)$. 
Recall the observation after the proof of Lemma~\ref{propblows}: In order 
that $\psi$ be a good partial desingularization, one of the blowups leading to 
it must be along (the strict transform of) $D\times Z$, where $D$ is 
the strict transform of a product $X\times Y$ 
of two subcurves $X$ and $Y$ of $C$ such that $R_1\in X\cap X'$, $R_2\in Y\cap Y'$ 
and $\phi^{-1}(R_1,R_2)\subseteq D$, and $Z$ is a subcurve of $C$ such that 
$T\in Z\cap Z'$. Then
$p_1$ looks locally analytically over $(A,T)$ like the map $p$ in 
Lemma~\ref{propblows} over a point $(a,b,c)$ with $b\neq 0$, and $\psi$ like 
$\psi_1$. In this situation, the local computations 
are similar to those done in the proofs of Propositions \ref{lemabasis} and 
\ref{lemabasis2}. Thus we get Statement~2, as well as the flatness of 
$\rho$ at any point of $\wt\C^3$ over $(A,T)$ and the regularity of 
$\wt\C^3$ along $\mu^{-1}(T)$ for every $T\in\N(C)$. Finally, since 
$\wt\Delta_2$ is trivial along $X$ away from 
$\mu^{-1}(R_2)$, the regularity of $\wt\C^3$ there yields that 
$\wt\Delta_2$ is Cartier along all of $X$, finishing the proof of
Statement 5. 
\end{proof}

Let $R$, $S$ and $T$ be reducible nodes of $C$. Let $A\in\wt\C^2$ be a 
distinguished point of $\phi^{-1}(R,S)$. We say that $[X,Y,Z]$ 
\emph{affects} $(A,T)$ if 
$$
R\in X\cap X',\quad S\in Y\cap Y',\quad T\in Z\cap Z',
$$
and $A$ lies on the strict transform to $\wt\C^2$ of $X\times Y$. 
Let $C_i$ and $C_j$ be the components containing $R$, 
let $C_k$ and $C_l$ be those containing $S$, and $C_m$ and $C_n$ those 
containing $T$. We say that $[X,Y,Z]$ is \emph{of type $[i,k,m]$} at 
$(R,S,T)$ if 
$$
C_i\times C_k\times C_m\subseteq X\times Y\times Z\quad\text{and}\quad
C_j\times C_l\times C_n\subseteq X'\times Y'\times Z'.
$$

\begin{Prop}\label{propblows2a} 
Let $\phi\:\wt\C^2\to\C^2$ and $\psi\:\wt\C^3\to\wt\C^2\times_B\C$ 
be good partial desingularizations, given by the sequence of blowups
\begin{equation}\label{tagbp}
[\Delta], [X_1,Y_1],\dots,[X_u,Y_u],[X_{u+1},Y_{u+1},Z_{u+1}],
\dots,[X_{u+v},Y_{u+v},Z_{u+v}].
\end{equation}
Fix reducible nodes $R$ and $S$ of $C$, 
and a distinguished point $A$ of $\phi^{-1}(R,S)$. Then there is a unique 
choice of integers $i,j,k,l\in\{1,\dots,p\}$ such that
$A$ is the intersection of the strict transforms to $\wt\C^2$ of 
$$
C_i\times C_k,\quad C_i\times C_l,\quad C_j\times C_k.
$$ 
Put $X:=\psi^{-1}p_1^{-1}(A)$, 
where $p_1\:\wt\C^2\times_B\C\to\wt\C^2$ is the projection. 
For each $m=1,\dots,p$, let $\wh C_m$ denote the strict transform to $X$ of 
$C_m$ under $p_2\psi$, 
where $p_2\:\wt\C^2\times_B\C\to\C$ is the projection.
For each reducible node $T$ of $C$, 
let $m_T,n_T\in\{1,\dots,p\}$ be the distinct integers such that 
$T\in C_{m_T}\cap C_{n_T}$, and let $E_{T,m_T}$ (resp. $E_{T,n_T}$) 
be the rational curve in 
$\psi^{-1}(A,T)$ intersecting $\wh C_{m_T}$ 
(resp. $\wh C_{n_T}$). Choose $m_T$ such that the first blowup in \eqref{tagbp} 
to affect $(A,T)$ is of type $[r_{T,1},s_{T,1},m_T]$ for 
$(r_{T,1},s_{T,1})\in\{i,j\}\times\{k,l\}$. Also, let 
$(r_{T,2},s_{T,2}),(r_{T,3},s_{T,3})\in\{i,j\}\times\{k,l\}$ distinct from 
each other and from $(r_{T,1},s_{T,1})$ such that
\begin{enumerate}
\item $\{(r_{T,1},s_{T,1}),(r_{T,2},s_{T,2}),(r_{T,3},s_{T,3})\}=
\{(i,k),(i,l),(j,k)\}$,
\item either $[r_{T,2},s_{T,2},n_T]$ or $[r_{T,3},s_{T,3},m_T]$ appears 
first in the sequence of types of blowups affecting $(A,T)$.
\end{enumerate}
For each $r,s,m\in\{1,\dots,p\}$, let $D_{r,s,m}$ denote the strict transform to 
$\wt\C^3$ of $C_r\times C_s\times C_m$. Also, let $\wt\Delta_1$ and 
$\wt\Delta_2$ denote the 
strict transforms to $\wt\C^3$ of $\Delta_1$ and $\Delta_2$. Let
$$
D:=\sum_{r,s,m}w_{r,s,m}D_{r,s,m},
$$
where the $w_{r,s,m}$ are given integers. For each $m=1,\dots,p$, set
\begin{equation}\label{wm}
w(m):=w_{i,k,m}+w_{i,l,m}+w_{j,k,m}.
\end{equation}
For each reducible node $T$ of $C$, 
set 
\begin{equation}\label{wmT}
\begin{array}{c}
w(T,m_T):=w_{r_{T,1},s_{T,1},m_T}+w_{r_{T,2},s_{T,2},n_T}+w_{r_{T,3},s_{T,3},m_T}\\
w(T,n_T):=w_{r_{T,1},s_{T,1},m_T}+w_{r_{T,2},s_{T,2},n_T}+w_{r_{T,3},s_{T,3},n_T}.
\end{array}
\end{equation}
Finally, set 
\begin{align}
\big(h(R,i),h(R,j)\big):=&
\begin{cases}\label{wmR}
(1,0)&\text{if $r_{R,1}=r_{R,2}=i$}\\
(1,1)&\text{if $(r_{R,1},m_R)=(j,i)$ or $(r_{R,2},n_R)=(j,i)$}\\
(0,0)&\text{if $(r_{R,1},m_R)=(j,j)$ or $(r_{R,2},n_R)=(j,j)$,} 
\end{cases}\\
\big(h(S,k),h(S,l)\big):=&
\begin{cases}\label{wmS}
(1,0)&\text{if $s_{S,1}=s_{S,2}=k$}\\
(1,1)&\text{if $(s_{S,1},m_S)=(l,k)$ or $(s_{S,2},n_S)=(l,k)$}\\
(0,0)&\text{if $(s_{S,1},m_S)=(l,l)$ or $(s_{S,2},n_S)=(l,l)$.} 
\end{cases}
\end{align}
Let $\lambda\:B\to\wt\C^2$ be any section of $\wt\C^2/B$ sending the 
special point of $B$ to $A$ and such that the pullbacks of the strict 
transforms of $C_i\times C_k$, $C_i\times C_l$, and $C_j\times C_k$ are all 
prime. Form the Cartesian diagram:
$$
\begin{CD}
\W @>\xi >> \wt\C^3\\
@V\rho VV @Vp_1\psi VV\\
B @>\lambda >> \wt\C^2.
\end{CD}
$$
Then $\rho$ is a regular smoothing of $C$, and the pullbacks of 
$\wt\Delta_1$, $\wt\Delta_2$ and $D$ to $\W$ under $\xi$ are 
Cartier divisors. More precisely,
\begin{align*}
\xi^*D=&\sum_{m=1}^pw(m)\wh C_m+\sum_{T\in\N(C)}\big(w(T,m_T)E_{T,m_T}+
w(T,n_T)E_{T,n_T}\big)\\
\xi^*\wt\Delta_1=&h(R,i)E_{R,i}+h(R,j)E_{R,j}+\Gamma_1\\
\xi^*\wt\Delta_2=&h(S,k)E_{S,k}+h(S,l)E_{S,l}+\Gamma_2,
\end{align*}
where $\Gamma_1$ and $\Gamma_2$ are relative effective Cartier divisors of 
$\W/B$ intersecting $X$ transversally on $E_{R,i}$ and $E_{S,k}$, 
respectively.
\end{Prop}

\begin{proof} Locally analytically, the map $\lambda$ is defined by 
the homomorphism of $K$-algebras
$$
K[[a,b,c]]\lra K[[t]],
$$
sending $a$, $b$ and $c$ to $t$, where 
$a=0$, $b=0$ and $c=0$ are local equations at $A$ 
of the strict transforms to $\wt\C^2$ of 
$C_i\times C_k$, $C_i\times C_l$ and $C_j\times C_k$. To show 
that $\rho$ is a smoothing of $C$, we need only show that 
$\W$ is regular along $\xi^{-1}\psi^{-1}(A,T)$ for 
every reducible node $T$ of $X$. This follows from the regularity of $W_2$ 
in Lemma~\ref{propblows}.

It follows that $\xi^*D$, $\xi^*\wt\Delta_1$ and $\xi^*\wt\Delta_2$ 
are Cartier divisors, 
the first supported on the special 
fiber of $\rho$. Also, the supports of 
$\xi^*\wt\Delta_1$ and $\xi^*\wt\Delta_2$ intersect $X$ only along 
$\xi^{-1}\psi^{-1}(A,R)$ and $\xi^{-1}\psi^{-1}(A,S)$, respectively. So we need 
only describe $\xi^*D$, $\xi^*\wt\Delta_1$ and $\xi^*\wt\Delta_2$ 
in a neighborhood 
of $\xi^{-1}\psi^{-1}(A,T)$ for each reducible node $T$ of $C$. Hence, we 
may use Lemma~\ref{propblows}, 
and need only describe the restrictions to $\wt X$ of the Cartier divisors on 
$Y_2$ corresponding to $D$, $\wt\Delta_1$ and $\wt\Delta_2$.

So, fix a reducible node $T$. Let $z_.=0$ and $z_{..}=0$ be local 
equations for the Cartier divisors $C_{m_T}$ and $C_{n_T}$ of $\C$ at $T$, 
respectively. We may assume 
$a=0$, $b=0$ and $c=0$ are local equations at $A$ 
of the strict transforms to $\wt\C^2$ of 
$C_{r_{T,1}}\times C_{s_{T,1}}$, $C_{r_{T,2}}\times C_{s_{T,2}}$ and 
$C_{r_{T,3}}\times C_{s_{T,3}}$, respectively.

The divisor $D$ is a sum of the $D_{r,s,m}$, and 
$\xi^*D_{r,s,m}$ is nonzero in a neighborhood of $\xi^{-1}\psi^{-1}(A,T)$ if 
and only if $m$ is either $m_T$ or $n_T$ and 
$(r,s)$ is equal to $(i,k)$, $(i,l)$ or $(j,k)$, or equivalently, to 
$(r_{T,1},s_{T,1})$, $(r_{T,2},s_{T,2})$ or $(r_{T,3},s_{T,3})$. The six cases, 
\begin{align*}
&(r_{T,1},s_{T,1},m_T),(r_{T,2},s_{T,2},m_T),(r_{T,3},s_{T,3},m_T),\\
&(r_{T,1},s_{T,1},n_T),(r_{T,2},s_{T,2},n_T),(r_{T,3},s_{T,3},n_T),
\end{align*}
correspond exactly to the six cases in the heading of the first table of 
Lemma~\ref{propblows}, in the same ordering, from which follows that 
the coefficients of $\wh C_{m_T}$, $\wh C_{n_T}$, $E_{T,m_T}$ and 
$E_{T,n_T}$ in the description of $\xi^*D$ 
are exactly those prescribed by \eqref{wm} and \eqref{wmT}.  

Furthermore, as observed above, $\xi^*\wt\Delta_1$ intersects $X$ only along 
$\xi^{-1}\psi^{-1}(A,R)$.
So, assume $T=R$. 
The equations at $(A,T)$ 
of the strict transform of $\Delta_1$ to $\wt\C^2\times_B\C$ depend 
on the choice of the sequence $(r_{T,1},s_{T,1}),(r_{T,2},s_{T,2}),(r_{T,3},s_{T,3})$. 
There are six possible sequences:
\begin{align*}
&(i,k),(i,l),(j,k)\\
&(i,l),(i,k),(j,k)\\
&(i,k),(j,k),(i,l)\\
&(i,l),(j,k),(i,k)\\
&(j,k),(i,k),(i,l)\\
&(j,k),(i,l),(i,k).
\end{align*}
There are also two cases: $m_R=i$ and $n_R=j$ or $m_R=j$ and $n_R=i$. If 
$m_R=i$ and $n_R=j$, the equations at $(A,T)$ 
of the strict transform of $\Delta_1$ to 
$\wt\C^2\times_B\C$ corresponding to each of the six sequences above are those 
on the headings of the last three columns of the second table of 
Lemma \ref{propblows}, each set of equations being repeated twice, more 
precisely,
\begin{align*}
&z_.=ab,\, z_{..}=c\\
&z_.=ab,\, z_{..}=c\\
&z_.=ac,\, z_{..}=b\\
&z_.=ac,\, z_{..}=b\\
&z_.=bc,\, z_{..}=a\\
&z_.=bc,\, z_{..}=a.
\end{align*}
On the other hand, if $m_R=j$ and $n_R=i$ we obtain the headings of 
the first three columns, just exchanging $z_.$ and $z_{..}$ in the above 
sequence of equations. The reasoning behind these conclusions is laid out 
right before the statement of Lemma \ref{propblows}. Now, it is just a 
matter of using the entries of the second table of 
Lemma \ref{propblows} to conclude that $\xi^*\wt\Delta_1$ is of the stated 
form, with the coefficients of $E_{R,i}$ and $E_{R,j}$ as 
prescribed by \eqref{wmR}.

The same proof works, \emph{mutatis mutandis}, to describe $\xi^*\wt\Delta_2$.
\end{proof}

\section{Admissible invertible sheaves}\label{5}

If $E$ is a chain of rational curves and $\L$ is an invertible sheaf on 
$E$, then $\L$ is determined by its restrictions to the irreducible 
components of $E$, and thus by its multidegree. In particular, 
$\L\cong\O_E$ if and only if $\deg(\L|_F)=0$ for each component $F\subseteq E$.

\begin{Lem}\label{chainh1h0} Let $E$ be a chain of rational curves of length 
$n$. Let $E_1$ and $E_n$ denote the extreme curves. 
Let $\L$ be an invertible sheaf on $E$.
Then the following statements hold:
\begin{enumerate}
\item $\deg(\L|_F)\geq -1$ for every subchain $F\subseteq E$ if and
  only if $h^1(E,\L)=0$.
\item $\deg(\L|_F)\leq 1$ for every subchain $F\subseteq E$ if and
  only if
$$
h^0(E,\L(-P-Q))=0
$$
for any two points $P\in E_1$ and $Q\in E_n$ on the nonsingular locus of $E$.
\end{enumerate}
\end{Lem}

\begin{proof}  Let $E_1,\dots,E_n$ be the irreducible components of $E$, 
ordered in such a way that $\# E_i\cap E_{i+1}=1$ for $i=1,\dots,n-1$. 
We prove the statements by induction on $n$. If $n=1$ all the 
statements follow from the knowledge of the cohomology of 
the sheaves $\O_{\IP^1_K}(j)$. 

Suppose $n>1$. We show Statement 1. Assume that $\deg(\L|_F)\geq -1$ for every 
subchain $F\subseteq E$. Consider the natural exact sequence
$$
0\to\L|_{E_1}(-N)\to\L\to\L|_{E'_1}\to0,
$$
where $E'_1:=\ol{E-E_1}$ and $N$ is the unique point of $E_1\cap E'_1$. 
By induction, $h^1(E'_1,\L|_{E'_1})=0$. If $\deg\L|_{E_1}\geq 0$ then 
$h^1(E_1,\L|_{E_1}(-N))=0$ as well, and hence
$h^1(E,\L)=0$ from the long exact sequence in cohomology.

Suppose now that $\deg\L|_{E_1}<0$. If $\deg\L|_{E_n}\geq 0$, we 
invert the ordering of the 
chain, and proceed as above. Thus we may suppose $\deg\L|_{E_n}<0$ as well. 
Since $\deg\L|_E\geq -1$, there is $i\in\{2,\dots,n-1\}$ such that
$\deg\L|_{E_i}\geq 1$. Let $F_1:=E_1\cup\cdots\cup E_{i-1}$ and 
$F_2:=E_{i+1}\cup\cdots\cup E_n$.
Consider the natural exact sequence
$$
0\to\L|_{E_i}(-N_1-N_2)\to\L\to\L|_{F_1}\oplus\L|_{F_2}\to0,
$$
where $N_1$ and $N_2$ are the two points of intersection of $E_i$ with 
$E'_i:=\ol{E-E_i}$. By induction,
$h^1(F_1,\L|_{F_1})=h^1(F_2,\L|_{F_2})=0$. Also, since 
$\deg \L|_{E_i}\geq 1$, we have $h^1(E_i,\L|_{E_i}(-N_1-N_2))=0$, and thus 
it follows from the long
exact sequence in cohomology that $h^1(E,\L)=0$ as well.

Assume now that $h^1(E,\L)=0$. Then $h^1(F,\L|_F)=0$ for every subchain
$F\subseteq E$. By induction, $\deg(\L|_F)\geq -1$ for every proper
subchain $F\subsetneq E$. Since $E$ is the union of two proper
subchains, it follows that $\deg(\L)\geq -2$. Assume by contradiction
that $\deg(\L)=-2$. Then $\deg(\L|_F)=-1$ for every proper subchain 
$F\subsetneq E$ containing $E_1$ or $E_n$. It follows that 
$$
\deg(\L|_{E_i})=\begin{cases}0&\text{if $1<i<n$},\\
-1&\text{otherwise.}\end{cases}
$$
But then $\L$ is the dualizing sheaf of $E$, and thus $h^1(E,\L)=1$,
reaching a contradiction. The proof of Statement 1 is complete.

Statement 2 is proved in a similar way. Alternatively, it is enough 
to observe that $\O_E(-P-Q)$ is the dualizing sheaf of $E$, 
and thus, by Serre Duality,
$$
h^0(E,\L(-P-Q))=h^1(E,\L^{-1}).
$$
So Statement 2 follows from 1.
\end{proof}

Let $\phi\:\X\to S$ be a family of connected curves. 
An $S$-flat coherent sheaf  
$\I$ on $\X$ is said to be a \emph{relatively torsion-free, rank-1 sheaf} 
(of \emph{relative degree} $d$) on $\X/S$ if the restriction of $\I$ to 
each geometric fiber of $\phi$ is torsion-free, rank-1 (of degree $d$).  
Let $\E$ be a locally free sheaf on $\X$ of constant rank and 
$\I$ a relatively torsion-free, rank-1 sheaf on $\X/S$. Let $\sigma\:S\to\X$ 
be a section of $\phi$ through its smooth locus. 
We say that $\I$ 
is \emph{semistable} (resp.~\emph{stable}, resp.~\emph{$\sigma$-quasistable}) 
\emph{with respect to $\E$} if, for every 
geometric point $s$ of $S$, 
\begin{enumerate}
\item $\chi(I_s\ox E_s)=0$,
\item $\chi((I_s)_Y\ox E_s|_Y)\geq 0$ for every proper subcurve 
$Y\subset X_s$ (resp.~with equality never, resp. 
with equality only if $\sigma(s)\not\in Y$),
\end{enumerate}
where $I_s$ and $E_s$ denote the restrictions of 
$\I$ and $\E$ to the fiber $X_s$ of $\phi$ over $s$. 
Notice that it is enough to 
check Property 2 above for connected subcurves $Y$.

The above notions of stability coincide with those in Section 
\ref{sect2.2} for an appropriate choice of component and polarization. In 
fact, let $\omega$ be the dualizing sheaf of $X_s$ and $Z$ the component 
of $X_s$ containing $\sigma(s)$. If $F$ is a 
vector bundle on $X_s$ with $\chi(I_s\ox F)=0$, there is a a polarization 
$\ul e$ for $X_s$ satisfying
\begin{equation}\label{eE}
e_Y=-\frac{\deg(F|_Y)}{\text{rk}(F)}+\frac{\deg(\omega|_Y)}{2}
\end{equation}
for every subcurve $Y$ of $X_s$. Then 
$\ul e$ has degree $\deg(I_s)$ and the sheaf 
$I_s$ is semistable (resp.~stable, resp.~$\sigma|_s$-quasistable) 
with respect to $F$ if and only if $I_s$ is semistable (resp.~stable, 
resp.~$Z$-quasistable) with respect to $\ul e$. 
Conversely, given a polarization $\ul e$ of $X_s$ of degree $\deg(I_s)$, 
there is a vector bundle $F$ on $X_s$ with $\chi(I_s\ox F)=0$ 
such that \eqref{eE} holds for every subcurve $Y$ of $X_s$; 
see \cite{MeVi}, Rmk.~1.16. 

Let $\psi\:\Y\to \X$ be a proper 
morphism such that the composition $\rho:=\phi\circ\psi$ is another family 
of curves. We say that $\psi$ is a \emph{semistable modification} of $\phi$ 
if for each geometric point $s$ of $S$ there are 
a collection of nodes $\N$ 
of the fiber $X_s$ of $\phi$ over $s$ and a map
$\eta\:\N\to\mathbb N$ such that the induced map $Y_s\to X_s$, where 
$Y_s$ is the fiber of $\rho$ over $s$, is 
$X_s$-isomorphic to $\mu_{\eta}\:(X_s)_{\eta}\to X_s$.

Assume $\psi$ is a semistable modification of $\phi$. Let $\L$ be an 
invertible sheaf on $\Y$.  We say that $\L$ is 
$\psi$-\emph{admissible} (resp.~\emph{strongly $\psi$-admissible}, 
resp.~$\psi$-\emph{invertible}) \emph{at} a given $s\in S$ 
if the restriction of $\L$ to every chain of rational curves on $\Y$ 
over a node of a geometric fiber of $\phi$ over $s$ has degree $-1$, $0$ or $1$ 
(resp.~$-1$ or $0$, resp.~$0$). We say that $\L$ is 
$\psi$-\emph{admissible} (resp.~\emph{strongly $\psi$-admissible}, 
resp.~$\psi$-\emph{invertible}) if $\L$ is so at every $s\in S$. 
Notice that, if $\L$ is strongly $\psi$-admissible, for every 
chain of rational curves on $\Y$ over a node of a geometric fiber of 
$\phi$, the degree of $\L$ on each of the components of the chain is 0 but 
for at most one component where the degree is $-1$.

\begin{Prop}\label{famchain}
Let $\phi\:\X\to S$ be a family of connected curves, $\psi\:\Y\to\X$ a 
semistable modification of $\phi$ and $\rho:=\phi\psi$. 
Let $\L$ be an invertible sheaf on $\Y$ of relative degree $d$ over $S$. 
Then the following statements hold:
\begin{enumerate}
\item The points $s$ of $S$ at which $\L$ is
  $\psi$-admissible (resp.~strongly $\psi$-admissible,
  resp.~$\psi$-invertible) form an open subset of $S$.
\item
$\L$ is $\psi$-admissible if and only if 
$\psi_*\L$ is a relatively torsion-free, rank-$1$ sheaf on $\X/S$ of 
relative degree $d$, whose formation commutes with base change. In
this case, $R^1\psi_*\L=0$.
\item
Let $\sigma\:S\to\Y$ be a section through the smooth locus of $\rho$ 
such that $\sigma(s)$ is not on any component of $Y_s$ contracted by
$\psi_s$ for any geometric point $s$ of $S$.
Let $\E$ be a vector bundle on $\X$. Then $\L$ is 
semistable (resp.~$\sigma$-quasistable, resp.~stable) 
with respect to $\psi^*\E$ if and only if 
$\L$ is $\psi$-admissible (resp.~strongly $\psi$-admissible,
resp.~$\psi$-invertible) 
and $\psi_*\L$ is 
semistable (resp.~$(\psi\sigma)$-quasistable, resp.~stable) 
with respect to $\E$.
\end{enumerate}
\end{Prop}

\begin{proof} All of the statements and hypotheses 
are local with respect to the \'etale topology of $S$. So we may
assume $S$ is Noetherian and that there is an 
invertible sheaf $\A$ on $\X$ that is relatively ample over $S$.  Let 
$\widehat{\A}:=\psi^*\A$. 

For each geometric point $s$ of $S$, let $Y_s:=\rho^{-1}(s)$ and 
$\psi_s:=\psi|_{Y_s}\:Y_s\to X_s$, with $X_s:=\phi^{-1}(s)$. 

We prove  Statement 1 first. 
For each geometric point $s$ of $S$, let $E_s$ be the subcurve of
$Y_s$ which is the union of all the components contracted by $\psi_s$,
and let $\wt X_s$ be the partial
normalization of $X_s$ obtained as the union of the remaining
components.
Since $\psi|_{\widetilde X_s}\:\widetilde X_s\to X_s$ is a 
finite map, it follows that $\widehat\A|_{\wt X_s}$ is ample, and thus 
$h^1(\wt X_s,(\L\ox\widehat{\A}^{\ox m_s})|_{\wt X_s}(-\sum P_i))=0$ 
for every large enough integer $m_s$, where the sum runs over all the branch points of 
$\wt  X_s$ above $X_s$.  Since $S$ is Noetherian, a large enough
integer works for all $s$, that is, for every $m>>0$,
\begin{equation}\label{hm>0}
h^1(\wt X_s,(\L\ox\widehat{\A}^{\ox m})|_{\wt X_s}(-\sum
P_i))=0\quad\text{for every geometric point $s$ of $S$}.
\end{equation}

So, for each large enough integer $m$ such that \eqref{hm>0}
holds, it follows from the long exact sequence in 
cohomology associated to the natural exact sequence
\begin{equation}\label{exseq}
0\to(\L\ox\widehat{\A}^{\ox m})|_{\wt X_s}(-\textstyle\sum P_i)\to
\L_s\ox\widehat{\A}_s^{\ox m}\to(\L\ox\widehat{\A}^{\ox m})|_{E_s}\to 0
\end{equation}
that 
\begin{equation}\label{h1E00}
h^1(Y_s,\L_s\ox\widehat{\A}_s^{\ox m})=
h^1(E_s,\L\ox\widehat{\A}^{\ox m}|_{E_s}),
\end{equation}
where $\L_s$ and $\widehat{\A}_s$ are the 
restrictions of $\L$ and $\widehat{\A}$ to $Y_s$. 
On the other hand, since $\widehat{\A}$ is a pullback from $\X$, it follows
that 
\begin{equation}\label{h1E0}
h^1(E_s,\L\ox\widehat{\A}^{\ox m}|_{E_s}) =
\sum h^1(F,\L|_F)\quad
\text{for every integer $m$},
\end{equation}
where the sum is over all the maximal chains $F$ of rational curves on
$Y_s$ contracted by $\psi_s$. Putting together \eqref{h1E00} and
\eqref{h1E0}, it follows now from Lemma~\ref{chainh1h0}
that 
\begin{equation}\label{adm=0}
h^1(Y_s,\L_s\ox\widehat{\A}_s^{\ox m})=0
\end{equation}
if and only if $\deg(\L|_F)\geq -1$ for every chain $F$ of
rational curves on $Y_s$ contracted by $\psi_s$. This is the case if
$\L_s$ is $\psi_s$-admissible.

It follows from semicontinuity of cohomology that the geometric points
$s$ of $S$ such that $\L|_{Y_s}$ has degree at least $-1$ on every chain
of rational curves of $Y_s$ contracted by $\psi_s$ 
form an open subset $S_1$ of $S$. Likewise, for
each integer $n$, the geometric points $s$ of $S$ such that
$\L^{\otimes n}|_{Y_s}$ has degree at least $-1$ on every chain of
rational curves of $Y_s$ contracted by $\psi_s$ form an open subset
$S_n$ of $S$. Then $S_1\cap S_{-1}$ parameterizes those $s$ for which
$\L|_{Y_s}$ is $\psi_s$-semistable, $S_1\cap S_{-2}$ parameterizes
those $s$ for which $\L|_{Y_s}$ is strongly $\psi_s$-semistable, and
$S_2\cap S_{-2}$ parameterizes those $s$ for which $\L|_{Y_s}$ is
$\psi_s$-invertible.

We prove Statement 2 now. 
Assume for the moment that $\L$ is $\psi$-admissible. 
To show that $\psi_*\L$ is flat over $S$, 
we need only show that $\phi_*(\psi_*\L\ox\A^{\ox m})$ is locally free for
each $m>>0$. By the projection formula, we need only show that
$\rho_*(\L\ox\widehat{\A}^{\ox m})$ is locally free for each $m>>0$. 
This follows from what we have already proved: For each large enough
integer $m$ such that \eqref{hm>0} holds, also 
\eqref{adm=0} holds for each geometric 
point $s$ of $S$, because $\L$ is $\psi$-admissible.

Furthermore, taking the long exact sequence in higher direct images of $\psi_s$ for 
the exact sequence \eqref{exseq} with $m=0$, using \eqref{h1E0} and 
that $\psi_s|_{\wt X_s}\:\wt X_s\to X_s$ is a finite map, it follows 
that $R^1\psi_{s*}(\L_s)=0$ for every geometric point $s$ of $S$. 
Since the fibers of $\psi$ have at most 
dimension 1, the formation of $R^1\psi_*(\L)$ commutes with base change, and 
thus $R^1\psi_*(\L)=0$.

Another consequence of \eqref{adm=0} holding for each geometric point
$s$ of $S$ is that the formation of 
$\rho_*(\L\ox\widehat{\A}^{\ox m})$ commutes with base
change for $m>>0$. 
We claim now that the base change map 
$\lambda_\X^*\psi_*\L\to\psi_{T*}\lambda_\Y^*\L$ is an
isomorphism for each Cartesian diagram of maps
$$
\begin{CD}
\Y_T @>\lambda_\Y >> \Y\\
@V\psi_TVV @V\psi VV\\
\X_T @>\lambda_\X >> \X\\
@V\phi_TVV @V\phi VV\\
T @>\lambda >> S.
\end{CD}
$$
Indeed, since $\A$ is relatively ample over $S$, 
it is enough to check that the induced map
\begin{equation}\label{bch}
\phi_{T*}(\lambda_\X^*\psi_*\L\ox\lambda_\X^*\A^{\ox m})\lra
\phi_{T*}(\psi_{T*}\lambda_\Y^*\L\ox\lambda_\X^*\A^{\ox m})
\end{equation}
is an isomorphism for $m>>0$. But, by the projection formula,
the right-hand side is simply 
$\phi_{T*}\psi_{T*}\lambda_\Y^*(\L\ox\wh\A^{\ox m})$.
Also, since $\psi_*\L$ is $S$-flat, 
the left-hand side is $\lambda^*\phi_*(\psi_*(\L)\ox\A^{\ox m})$ for $m>>0$, 
whence equal to 
$\lambda^*\phi_*\psi_*(\L\ox\widehat{\A}^{\ox m})$ by the projection formula. 
So, since the formation of $\rho_*(\L\ox\widehat{\A}^{\ox m})$ 
commutes with base
change for $m>>0$, 
it follows that \eqref{bch} is an isomorphism for $m>>0$, as asserted.

To prove the remainder of Statement 2 we 
may now assume that $S$ is a geometric point. Let $X:=\X$ and 
$Y:=\Y$.  We need only show that $\psi_*\L$ is a torsion-free, rank-1
sheaf of degree $d$ on $X$ if and only if $\L$ is $\psi$-admissible. 
Again, let $E$ be the union of the components of $Y$ contracted by $\psi_s$,
and let $\wt X$ be the union of the remaining
components. Taking higher direct images under $\psi$ in the 
natural exact sequences
\begin{align*}
&0\to\L|_{\wt X}(-\textstyle\sum P_i)\to\L\to\L|_E\to0,\\
&0\to\L|_E(-\textstyle\sum P_i)\to\L\to\L|_{\wt X}\to0,
\end{align*}
where the sums run over the intersection points 
$P_i$ of $\widetilde X$ and $E$, and using that $\psi|_{\wt X}$ is a
finite map, we get 
\begin{equation}\label{RRR}
R^1\psi_*\L=R^1\psi_*\L|_E
\end{equation}
and the exact sequence
$$
0\to\psi_*\L|_E(-\textstyle\sum P_i)\to\psi_*\L\to\psi_*\L|_{\wt X}\to
R^1\psi_*\L|_E(-\textstyle\sum P_i)\to R^1\psi_*\L\to 0 .
$$
Since $\psi|_{\widetilde X}$ is also birational, $\psi_*\L|_{\widetilde  X}$ is a torsion-free, 
rank-1 sheaf of degree $\deg\L|_{\widetilde  X}+e$, where $e$ 
is the number of points of $X$ 
desingularized in $\wt X$, thus equal to the number of maximal chains of 
rational curves on $Y$ contracted by $\psi$. Since 
$\psi_*\L|_E(-\textstyle\sum P_i)$ is supported at finitely many
points, it follows that $\psi_*\L$ is
torsion-free if and only if $h^0(E,\L|_E(-\sum P_i)=0$. The latter holds if
and only if the degree of $\L$ on each chain of rational curves in $E$
is at most $1$, by Lemma \ref{chainh1h0}. Furthermore, if the latter
holds, then $R^1\psi_*\L|_E(-\textstyle\sum P_i)$ has length
$e-\deg\L|_E$ by the Riemann--Roch Theorem. Since 
$\deg\L|_{\wt  X}+\deg\L|_E=d$, it follows that $\deg\psi_*\L=d$ if
and only if $R^1\psi_*\L=0$. By \eqref{RRR}, the latter holds if and
only if $h^1(E,\L|_E)=0$, thus if and only if the degree 
of $\L$ on each chain of rational curves in $E$
is at least $-1$, by Lemma \ref{chainh1h0}. The proof of Statement 2
is complete.

We prove Statement 3 now. 
We may assume that $S$ is a geometric point. Let $X:=\X$ and 
$Y:=\Y$.  
Let $P\in Y$ be the image of the section $\sigma$. Since $\psi^*\E$ has 
degree 0 on every component of $Y$ contracted by $\psi$, and $P$ does not lie on any of these components, 
it follows from the definitions that a semistable (resp.~$P$-quasistable, 
resp.~stable) sheaf has degree $-1$, $0$ or $1$ 
(resp.~$-1$ or $0$, resp. $0$) on every chain of rational curves of
$Y$ contracted by $\psi$.

We may now assume that $\L$ is $\psi$-admissible. 
Let $W$ be any connected subcurve of $X$. 
Set $W':=\ol{X-W}$ and $\Delta_W:=W\cap W'$. 
Set $\delta:=\#\Delta_W$. Let $V_1:=\ol{Y-\psi^{-1}(W')}$ and 
$V_2:=\ol{Y-\psi^{-1}(W)}$. Let $F_1,\dots,F_r$ be the maximal chains of 
rational curves contained in $\psi^{-1}(\Delta_W)$. Then $0\leq r\leq\delta$. 

{\it Claim:} $(\psi_*\L)_W\cong\psi_*(\L|_Z)$ 
for a certain connected subcurve $Z\subseteq Y$ such that:
\begin{enumerate}
\item $V_1\subseteq Z\subseteq\psi^{-1}(W)$.
\item For each connected subcurve $U\subseteq Y$ 
such that $V_1\subseteq U\subseteq\psi^{-1}(W)$,
$$
\deg(\L|_U)\geq\deg(\L|_Z).
$$
\end{enumerate}
(Notice that Property 1 implies that $P\in Z$ if and only if $\psi(P)\in W$.) 

Indeed, if $W=X$, let $Z:=\psi^{-1}(W)$. Suppose $W\neq X$. Then $\delta>0$.  
Let $M_1,\dots,M_\delta$ be the points of intersection of $V_1$ 
with $V'_1:=\ol{Y-V_1}$ and $N_1,\dots,N_\delta$ those of $V_2$ with 
$V'_2:=\ol{Y-V_2}$. 

Write $F_i=F_{i,1}\cup\dots\cup F_{i,e_i}$, 
where $F_{i,j}\cap F_{i,j+1}\neq\emptyset$ for $j=1,\dots,e_i-1$ and 
$F_{i,1}$ intersects $V_1$. Up ro reordering the $M_i$ and $N_i$, 
we may assume that $F_{i,1}$ intersects $V_1$ 
at $M_i$ and $F_{i,e_i}$ intersects $V_2$ at $N_i$ for $i=1,\dots,r$. 
(Thus $M_i=N_i$ for $i=r+1,\dots,\delta$.) Up to 
reordering the $F_i$, we may also assume that there 
are nonnegative integers $u$ and $t$ with $u\leq t$ 
such that 
$$
\deg(\L|_{F_i})=\begin{cases}
1&\text{for $i=1,\dots,u$}\\
0&\text{for $i=u+1,\dots,t$}\\
-1&\text{for $i=t+1,\dots,r$.}
\end{cases}
$$

Up to reordering the $F_i$, we 
may assume there is an integer $b$ with $u\leq b\leq t$ such that, 
for each $i=u+1,\dots,t$, we have that $i>b$ if and only if 
$\deg(\L|_{F_{i,j}})=0$ for every $j$ or the largest integer $j$ such that 
$\deg(\L|_{F_{i,j}})\neq 0$ is such that $\deg(\L|_{F_{i,j}})=-1$. Set 
$G_i:=F_i$ for $i=b+1,\dots,r$. For each $i=u+1,\dots,b$, let 
$G_i:=F_{i,1}\cup\dots\cup F_{i,j-1}$, where $j$ is the largest integer such 
that $\deg(\L|_{F_{i,j}})=1$, let $\wh G_i:=\ol{F_i-G_i}$ and denote by 
$B_i$ the point of intersection of $G_i$ and $\wh G_i$. (Notice 
that $1< j\leq e_i$.) Let $B_i:=M_i$ and $\wh G_i:=F_i$ for $i=1,\dots,u$, and 
$B_i:=N_i$ for $i=b+1,\dots,\delta$.  

For $i=u+1,\dots,r$, since the degree of $\L|_{G_i}(B_i)$ on each subchain of 
$G_i$ is at most 1, it follows from Lemma~\ref{chainh1h0} that
\begin{equation}\label{MN3}
h^0(G_i,\L|_{G_i}(-M_i))=0\quad\text{for $i=u+1,\dots,r$}.
\end{equation}
Furthermore, for $i=1,\dots,b$, the total degree of $\L|_{\wh G_i}$ is 1; thus, 
by Lemma~\ref{chainh1h0} and the Riemann--Roch Theorem,
\begin{equation}\label{MN4}
h^1(\wh G_i,\L|_{\wh G_i}(-B_i-N_i))=0\quad\text{for $i=1,\dots,b$}.
\end{equation}

Set
$$
Z:=V_1\cup G_{u+1}\cup\dots\cup G_r
$$
and $Z':=\ol{Y-Z}$. Put $\Delta_Z:=Z\cap Z'$. Notice that 
$\Delta_Z=\{B_1,\dots,B_\delta\}$. Also, notice that $Z$ is connected, and 
$$
\deg(\L|_U)\geq\deg(\L|_Z)=\deg(\L|_{V_1})-(b-u)-(r-t)
$$
for each connected subcurve $U\subseteq Y$ such that 
$V_1\subseteq U\subseteq\psi^{-1}(W)$.

We have three natural exact sequences:
\begin{equation}\label{ex1}
0\to\L|_{Z'}(-\sum_{i=1}^\delta B_i)\to\L\to\L|_Z\to 0,
\end{equation}
\begin{equation}\label{ex2}
0\to\bigoplus_{i=1}^b\L|_{\wh G_i}(-B_i-N_i)\to\L|_{Z'}(-\sum_{i=1}^\delta B_i)
\to\L|_{V_2}(-\sum_{i=b+1}^\delta B_i)\to 0,
\end{equation}
\begin{equation}\label{ex3}
0\to\bigoplus_{i=u+1}^r\L|_{G_i}(-M_i)\to\L|_Z\to\L|_{V_1}\to0.
\end{equation}
Since $\L$ is $\psi$-admissible, so are 
$\L|_{V_1}$ with respect to $\psi|_{V_1}\:V_1\to W$ and $\L|_{V_2}$ with 
respect to $\psi|_{V_2}\:V_2\to W'$. Then $\psi_*(\L|_{V_1})$ is a 
torsion-free, rank-1 sheaf on $W$ and 
$R^1\psi_*(\L|_{V_2}(-\sum B_i))=0$ by Statement 1. 

Since $R^1\psi_*(\L|_{V_2}(-\sum B_i))=0$, from 
\eqref{MN4} and the long exact sequence of higher direct images under 
$\psi$ of \eqref{ex2} and \eqref{ex1} we get that 
$R^1\psi_*(\L|_{Z'}(-\sum B_i))=0$ 
and the natural map $\psi_*\L\to\psi_*(\L|_Z)$ is surjective. Also, it 
follows from \eqref{MN3} and the long exact sequence of 
higher direct images under $\psi$ of \eqref{ex3} that the natural map 
$\psi_*(\L|_Z)\to\psi_*(\L|_{V_1})$ is injective. Thus, since 
$\psi_*(\L|_{V_1})$ is a torsion-free, rank-1 sheaf on $W$, so is 
$\psi_*(\L|_Z)$. And, since $\psi_*\L\to\psi_*(\L|_Z)$ is surjective, we 
get an isomorphism $(\psi_*\L)_W\cong\psi_*(\L|_Z)$, 
finishing the proof of the claim.

To prove the ``only if'' part of Statement 2, let $W$ be any connected 
subcurve of $X$. Let $Z$ be as in the claim. Since $\L$ is admissible 
with respect to $\psi$, Statement 1 yields $R^1\psi_*\L=0$, and hence 
$R^1\psi_*(\L|_Z)=0$ from the long exact sequence of higher direct images 
under $\psi$ of \eqref{ex1}. Thus, by the claim and the 
projection formula,  
\begin{equation}\label{ZW}
\chi((\psi_*\L)_W\ox\E|_W)=\chi(\psi_*(\L|_Z)\ox\E|_W)=
\chi(\L|_Z\ox(\psi^*\E)|_Z).
\end{equation}
If $\L$ is semistable (resp.~$P$-quasistable, resp.~stable) then 
$\chi(\L|_Z\ox(\psi^*\E)|_Z)\geq 0$ (resp.~with equality only if $Z=Y$ or 
$Z\not\ni P$, resp.~with equality only if $Z=Y$). Now, if 
$Z=Y$ then $W=X$. Also, $P\in Z$ if and only if $\psi(P)\in W$. So 
\eqref{ZW} yields $\chi((\psi_*\L)_W\ox\E|_W)\geq 0$ 
(resp.~with equality only if $W=X$ or 
$W\not\ni\psi(P)$, resp.~with equality only if $W=X$).

As for the ``if'' part, let $U$ be a connected subcurve of $Y$. 
If $U$ is a union of components of $Y$ contracted by $\psi$, then
$U$ is a chain of rational curves of $Y$ collapsing to a node of $X$, 
and hence $\L|_U$ has degree at least $-1$ 
(exactly $0$ if $\L$ is $\psi$-invertible). Thus 
$$
\chi(\L|_U\ox\psi^*\E|_U)=\text{rk}(\E)\chi(\L|_U)\geq 0,
$$
with equality only if $\L$ is not $\psi$-invertible.

Suppose now that $U$ contains a component of $Y$ not contracted by
$\psi$. Then $W:=\psi(U)$ is a connected subcurve of $X$. 
Let $\wh U$ be the smallest subcurve of $Y$ containing $U$ and 
$\ol{Y-\psi^{-1}(W')}$, where $W':=\ol{X-W}$. Then $\wh U$ is connected 
and contained in $\psi^{-1}(W)$. Furthermore, $\chi(\O_U)-\chi(\O_{\wh U})$ 
is the number of connected components of $\ol{\wh U-U}$. Thus 
\begin{equation}\label{UU}
\deg(\L|_U)+\chi(\O_U)\geq\deg(\L|_{\wh U})+\chi(\O_{\wh U}),
\end{equation}
with equality only if $\L$ has degree 1 on every connected component of 
$\ol{\wh U-U}$. Let $Z$ be as in the claim. Notice that 
$\chi(\O_{\wh U})=\chi(\O_Z)$. Since 
$\deg(\L|_{\wh U})\geq\deg(\L|_Z)$ by the claim,
using \eqref{ZW} and \eqref{UU} we get
\begin{align*}
\chi(\L|_U\ox\psi^*\E|_U)&=\text{rk}(\E)(\deg(\L|_U)+\chi(\O_U))
+\deg(\psi^*\E|_U)\\
&\geq\text{rk}(\E)(\deg(\L|_{\wh U})+\chi(\O_{\wh U}))
+\deg(\psi^*\E|_{\wh U})\\
&=\text{rk}(\E)(\deg(\L|_{\wh U})+\chi(\O_Z))
+\deg(\psi^*\E|_Z)\\
&\geq\text{rk}(\E)(\deg(\L|_Z)+\chi(\O_Z))
+\deg(\psi^*\E|_Z)\\
&=\chi(\L|_Z\ox(\psi^*\E)|_Z)\\
&=\chi((\psi_*\L)_W\ox\E|_W).
\end{align*}
Assume that $\psi_*\L$ is 
semistable (resp.~$(\psi\sigma)$-quasistable, resp.~stable) 
with respect to $\E$. Then $\chi((\psi_*\L)_W\ox\E|_W)\geq 0$ 
(resp.~with equality only if $W=X$ or 
$W\not\ni\psi(P)$, resp.~with equality only if $W=X$). So 
$\chi(\L|_U\ox\psi^*\E|_U)\geq 0$. Suppose $\chi(\L|_U\ox\psi^*\E|_U)=0$. 
Then $\chi((\psi_*\L)_W\ox\E|_W)=0$ and equality holds in \eqref{UU}. 
If $W\not\ni\psi(P)$ then $U\not\ni P$. Suppose $W=X$. Then 
$\wh U=Y$. If $U\neq Y$ then $\L$ has degree 1 on each connected component of 
$\ol{Y-U}$, and thus $\L$ is not strongly admissible. 
\end{proof}

\begin{Prop}\label{compadm} Let $X$ be a connected curve. Let 
$\psi\:Y\to X$ be a semistable modification of $X$. 
Let $\L$ and $\M$ be invertible sheaves on $Y$ which are 
$\psi$-admissible. Assume that $\M\otimes\L^{-1}$ is a twister of the form
$$
\O_Y\left(\sum c_E E \right), \,\,\,  c_E\in\mathbb{Z},
$$
where the sum runs over the set of components $E$ of $Y$ contracted by
$\psi$. Then $\psi_*\L\simeq\psi_*\M$.
\end{Prop}

\begin{proof} Set $\T:=\M\otimes\L^{-1}$. Let $\R$ be the 
set of smooth, rational curves contained in $Y$ and contracted by
$\psi$. If $\R=\emptyset$, then $\T=\O_Y$ and thus $\L\cong\M$. 
Suppose $\R\ne\emptyset$. Let $\K$ be the set of maximal chains of rational 
curves contained in $\R$. 

{\it Claim:} For every $F\in \K$ and every two components $E_1,E_2\subseteq F$ 
such that $E_1\cap E_2\ne\emptyset$, we have $|c_{E_1}-c_{E_2}|\le 1$. In 
addition, if $E$ is an extreme component of $F$, then $|c_E|\le 1$.

Indeed, let $E_1,\dots E_n$ be the components of $F$, 
ordered in such a way that $\# E_i\cap E_{i+1}=1$ for $i=1,\dots,n-1$. Since  
 $\L$ and $\M$ are admissible, $|\deg_G\T|\le 2$ for every subchain 
$G$ of $F$. Set $c_{E_0}:=c_{E_{n+1}}:=0$. We will reason by contradiction. 
Thus, up to reversing the order of the $E_i$, we may assume 
that $c_{E_i}-c_{E_{i+1}}\ge 2$ 
for some $i\in\{0,\dots,n\}$. Then
$$
c_{E_i}\le c_{E_{i-1}}\le \cdots \le c_{E_1}\le c_{E_0}=0,
$$
because, if  
$c_{E_j}>c_{E_{j-1}}$ for some $j\in\{1,\dots,i\}$, then 
$$
\deg_{E_j\cup\cdots\cup E_i}\T=c_{E_{j-1}}-c_{E_j}+c_{E_{i+1}}-c_{E_i}<-2.
$$
Similarly, $c_{E_{i+1}}\ge c_{E_{i+2}}\ge \cdots \ge c_{E_n}\ge c_{E_{n+1}}=0$. 
But then
$$
0\le c_{E_{i+1}}<c_{E_i}\le 0,
$$
a contradiction that proves the claim.
  
Now, for each $F\in\K$, let $F^\dagger$ be the (possibly empty) union of 
components $E\subseteq F$ such that $c_E=0$. For each connected 
component $G$ of $\overline{F-F^\dagger}$ and irreducible components   
$E_1,E_2\subseteq G$, it follows from the claim that $c_{E_1}\cdot c_{E_2}>0$. 
Let $\K^+$ (resp. $\K^-$) be the collection of connected components $G$ of  
$\overline{F-F^\dagger}$ for $F\in\K$ such that $c_E>0$ (resp. $c_E<0$) 
for every irreducible component $E\subseteq G$.

Notice that, again by the claim, 
  \begin{equation}\label{extrcoef}
 c_E=
  \begin{cases}
  \begin{array}{ll}
  \,\,\,\,1 & \text{if $E$ is an extreme component of some $G\in \K^+_F$}\\
  -1 & \text{if $E$ is an extreme component of some $G\in \K^-_F$.}
    \end{array}
  \end{cases}
   \end{equation}
So, being $\L$ and $\M$ admissible, 
 \begin{equation}\label{degF}
 \deg_G \L=-\deg_G\M=
  \begin{cases}
  \begin{array}{ll}
  \,\,\,\,1 & \text{if } G\in \K^+ \\
  -1 & \text{if } G\in \K^-.  
  \end{array}
  \end{cases}
 \end{equation}

Define 
 $$
W^+:=\overline{Y- \cup_{G\in \K^+} G}, 
  \,\,\,\,\,\,\,\,\,\,
  W^-:=\overline{Y-\cup_{G\in \K^-} G},
  \,\,\,\,\,\,\,\,\,\,
  W:=\overline{Y-\cup_{G\in \K^-\cup\K^+} G}.
$$
For each $G\in\K^+\cup \K^-$, let $N_G$ and $N'_G$ denote the points of 
$G\cap \overline{Y-G}$, and put
$$
D^+:=\sum_{G\in\K^+}(N_G+N'_G) \,\,\text{ and }\,\, 
D^-:=\sum_{G\in\K^-}(N_G+N'_G).
$$
We may view $D^+$ and $D^-$ as divisors of $W$. 
Thus, by \eqref{extrcoef}, 
 \begin{equation}\label{rel1}
  \M|_W\simeq \L|_W(D^+-D^-).
 \end{equation}
 
Consider the natural diagram
\begin{equation}\label{diagram}
\begin{CD}
@. @.   @. 0 @. \\
 @. @.  @.   @VVV \\
@. @.   @. \L|_W(-D^-)@. \\
 @. @.  @.   @VVV \\
0@>>>  \underset{G\in \K^+}{\oplus} \L|_G(-N_G-N'_G)  
@>>> \L @>>>\L|_{W^+} @>>> 0\\
 @. @.  @.   @VVV \\
@.  @.  @. \underset{G\in \K^-}{\oplus} \L|_G   \\
 @. @.  @.   @VVV \\
@.  @.  @. 0  \\
\end{CD}
\end{equation}
where the horizontal and vertical sequences are exact. 
By (\ref{degF}) and Lemma \ref{chainh1h0}, and using the 
Riemann--Roch Theorem, 
$$
R^i\psi_* \L|_G(-N_G-N'_G)=H^i(G,\L|_G(-N_G-N'_G))\ox\O_{\psi(G)}=0
$$
for $G\in \K^+$ and $i=0,1$, whereas
$$
\psi_*\L|_G=H^0(G,\L|_G)\ox\O_{\psi(G)}=0 \; \text{ for } G\in \K^-.
$$
Hence, it follows from Diagram (\ref{diagram}), by considering the 
associated long exact sequences in higher direct images of $\psi$, that
\begin{equation}\label{rel2}
\psi_*\L\simeq (\psi|_W)_*\L|_W(-D^-).
\end{equation}

Consider a second  diagram, similar to Diagram \eqref{diagram}, but with the 
roles of $\K^+$ and $\K^-$, and thus of $D^+$ and $D^-$, reversed, and 
$\M$ substituted for $\L$. As before, 
$$
R^i\psi_*\M|_G(-N_G-N'_G))=
H^i(G,\M|_G(-N_G-N'_G))\ox\O_{\psi(G)}=0
$$
for $G\in \K^-$ and $i=0,1$, whereas
$$
\psi_*\M|_G\simeq H^0(G,\M|_G)\ox\O_{\psi(G)}=0 \text{ for } G\in \K^+.
$$
Hence, taking the associated long exact sequences, 
\begin{equation}\label{rel3}
\psi_*\M\simeq (\psi|_W)_*\M|_W(-D^+).
 \end{equation} 
 Combining (\ref{rel1}), (\ref{rel2}) and (\ref{rel3}), 
we get $\psi_*\L\simeq \psi_*\M$.
 \end{proof}

\section{The degree-$2$ Abel map}\label{6}

\subsection{Corrections}

Recall the notation: $C$ is a curve with irreducible 
components $C_1,\ldots,C_p$ defined over $K$, 
an algebraically closed field, $\dot C$ is the smooth locus of $C$, 
and $\C/B$ is a smoothing of $C$ with regular total space $\C$, 
the smooth locus of which is $\dot\C$. Also, $\P$ is an invertible sheaf on 
$\C$ and $\N(C)$ is the set of reducible nodes of $C$.

We would like to resolve the rational map 
\begin{equation}\label{alpha2dot}
\alpha^2_{\C/B}\col\C^2\dashrightarrow\ol{\mathcal J}
\end{equation}
given by \eqref{dotalpha}. We will reduce the resolution of the above 
map to a combinatorial question.

As seen in Subsection \ref{2.3}, for each pair of 
integers $(i,k)$ such that $1\leq i,k\leq p$, there is a formal 
sum 
$$
Z_{(i,k)}=\sum_{m=1}^p w_{(i,k)}(m)C_m
$$ 
of components of $C$ such that 
$$
\P|_C\ox\O_C(-Q_1-Q_2)\ox\O_C(-Z_{(i,k)})
$$
is $C_1$-quasistable for every $Q_1\in\dot C_i$ and $Q_2\in\dot C_k$, where 
$\dot C_j:=C_j\cap\dot C$ for every $j$. The 
formal sum is not unique, but differs from another by a multiple of 
the sum of all the components of $C$. Given two components $C_m$ and $C_n$ of 
$C$, set $\delta_{(i,k)}(m,n):=w_{(i,k)}(m)-w_{(i,k)}(n)$. 
Note that $\delta_{(i,k)}(m,n)$ 
remains the same if $Z_{(i,k)}$ is replaced by another 
formal sum having the same property. Also, 
$$
\delta_{(i,k)}(m,n)=\delta_{(k,i)}(m,n)=-\delta_{(k,i)}(n,m)=-\delta_{(i,k)}(n,m).
$$

We emphasize the numerical nature of the $\delta_{(i,k)}(m,n)$ by the 
following Definition and Proposition:

\begin{Def}\label{Abdata} 
Let $\Gamma$ be a connected graph without loops. Let $V$ be its 
set of vertices and $E$ its set of edges. 
For each $i\in V$, let 
$$
\underline c_i\:V\to\mathbb Z
$$
be defined by letting $\ul c_i(j)$ be the number of edges of 
$\Gamma$ with ends $i$ and $j$, for $j\neq i$, and 
$$
\ul c_i(i):=-\sum_{j\neq i}\ul c_i(j).
$$
Let $\ul q\:V\to\mathbb Z$ and $\ul e\:V\to\mathbb Q$ be such that
$$
f:=\sum_{i\in V}\ul e(i)=-2+\sum_{i\in V}\ul q(i).
$$

For each proper nonempty subset $I\subset V$, let 
$k_I$ be the number of edges of $\Gamma$ with one end inside $I$ and the 
other end outside. Let $v\in V$. A function 
$\ul d\:V\to\mathbb Z$ satisfying $\sum_i\ul d(i)=f$ is called 
\emph{$v$-quasistable with respect to $\ul e$} if 
$$
\sum_{i\in I}\big(\ul d(i)-\ul e(i)\big)\geq-\frac{k_I}{2}
$$
for every proper nonempty subset $I\subset V$, 
with equality only if $v\not\in I$.

For each $u\in V$, let $\ul\delta_u\:V\to\mathbb Z$ be defined 
by letting $\ul\delta_u(w):=0$ if $w\neq u$ and $\ul\delta_u(u):=1$. 
As pointed out in Subsection \ref{2.3}, 
for each pair $(i,k)\in V^2$, there exists 
$\ul w_{(i,k)}\:V\to\mathbb Z$ such that 
$$
\ul q-\ul \delta_i-\ul\delta_k-\sum_{m\in V}\ul w_{(i,k)}(m)\ul c_m
$$
is $v$-quasistable with respect to $\ul e$. The functions $\ul w_{(i,k)}$ are 
not uniquely defined, but their differences, 
$\ul w_{(i,k)}(m)-\ul w_{(i,k)}(n)$, are.

Let
$$
\ul\delta\:V^4\to\mathbb Z
$$
be defined by 
$$
\ul\delta(i,k,m,n):=\ul w_{(i,k)}(m)-\ul w_{(i,k)}(n)
$$
for each $(i,k,m,n)\in V^4$. The following conditions hold:
$$
\ul\delta(i,k,m,n)=\ul\delta(k,i,m,n)=-\ul\delta(k,i,n,m)=-\ul\delta(i,k,n,m)
$$
for all $(i,k,m,n)\in V^4$.

The data $(\Gamma,\ul e,\ul q,v)$ are called \emph{degree-2 Abel data}. 
And $\ul\delta$ is called their \emph{correction function}.
\end{Def}

\begin{Prop} Let $\Gamma$ be the essential dual graph of $C$ and $V$ its 
set of vertices. Let $\ul e\:V\to\mathbb Q$ and 
$\ul q\:V\to\mathbb Z$ be defined by setting $\ul e(C_i)=e_i$ and 
$\ul q(C_i):=\deg(\P|_{C_i})$ for each $i=1,\dots,p$. Set $v:=C_1$. 
Then 
$(\Gamma,\ul e,\ul q,v)$ are degree-2 Abel data and their correction 
function $\ul\delta$ satisfies
$$
\ul\delta(C_i,C_k,C_m,C_n)=\delta_{(i,k)}(m,n)
$$
for all $i,k,m,n$ between $1$ and $p$.
\end{Prop}

\begin{proof} Immediate.
\end{proof}

\subsection{Admissibility}

Let $\phi\:\wt\C^2\to\C^2$ and $\psi\:\wt\C^3\to\wt\C^2\times_B\C$ be 
good partial desingularizations. Let $\rho\:\wt\C^3\to\wt\C^2$ be the 
composition of $\psi$ with the 
projection $p_1\:\wt\C^2\times_B\C\to\wt\C^2$. It follows from 
Proposition \ref{propblows2} that the fibers of $\rho$ over 
closed points of $\wt\C^2$ are isomorphic to $C$, to $C(1)$ or to $C(2)$. 

For each pair $(i,k)$ such that $1\leq i,k\leq p$, denote by 
$\Z_{(i,k)}$ the weighted sum over $m=1,\dots,p$ of the strict transforms 
to $\wt\C^3$ of the triple products $C_i\times C_k\times C_m$, each with 
weight $w_{(i,k)}(m)$. Also, let $\Delta_l$ be the inverse image of the diagonal 
$\Delta\subset\C^2$ 
under the projection $\C^3\to\C^2$ onto the product over $B$ of the 
$l$-th and last factors of $\C^3$, and $\wt\Delta_l$ be the strict 
transform of $\Delta_l$ to $\wt\C^3$ 
for $l=1,2$. The $\Z_{(i,k)}$ are Cartier divisors of $\wt\C^3$, 
whereas the $\wt\Delta_l$ are so at least along 
the fibers of $\rho$ isomorphic to $C(2)$, again by 
Proposition \ref{propblows2}.

Recall the natural subscheme 
$F_2\subset\IP_{\C^2}(\I_{\Delta|\C^2})\times_B\C$ defined in 
Subsection~\ref{flag}. Let $\wt F_2$ be the strict transform 
of $F_2$ to $\wt\C^3$. It follows from Proposition~\ref{propblows2}
that wherever $\wt F_2$ is not $\rho$-flat, 
$$
\I_{\wt{F_2}|\wt\C^3}=\I_{\wt\Delta_1|\wt\C^3}\I_{\wt\Delta_2|\wt\C^3}=
\I_{\wt\Delta_1|\wt\C^3}\otimes\I_{\wt\Delta_2|\wt\C^3}
$$
and both $\wt\Delta_1$ and $\wt\Delta_2$ are Cartier. In any case, it
follows that the sheaf of ideals $\I_{\wt F_2|\wt\C^3}$ is a relatively torsion-free, 
rank-1 sheaf of relative degree $-2$ on $\wt\C^3/\wt\C^2$ agreeing 
with $\O_{\wt\C^3}(-\wt\Delta_1-\wt\Delta_2)$ around the fibers of 
$\rho$ isomorphic to $C(2)$.

Set 
$$
\L_\psi:=\psi^*p_2^*\P\ox\I_{\wt F_2|\wt\C^3}\ox\O_{\wt\C^3}(-\sum_{i,k=1}^p\Z_{(i,k)}),
$$
where $p_2\:\wt\C^2\times_B\C\to\C$ is the projection.  

\begin{Lem}\label{lemblowindep} Let $\phi\:\wt\C^2\to\C^2$ and 
$\psi\:\wt\C^3\to\wt\C^2\times_B\C$ be good partial
desingularizations. Then $\psi$ is a semistable modification of the 
projection $p_1\:\wt\C^2\times_B\C\to\wt\C^2$. In addition,
\begin{itemize}
\item[1.] $\L_\psi$ is $\psi$-admissible if and only if 
$\L_\psi$ is $\psi$-admissible at each distinguished point of
$\phi^{-1}(R,S)$ for every $R,S\in\N(C)$.
\end{itemize}
Furthermore, let $\rho\:\wt\C^3\to\wt\C^2$ be the composition of 
$\psi$ with $p_1$. For each fiber $X$ of 
$\rho$ over a closed point of $\wt\C^2$, let $\mu_X\:X\to C$ be the 
restriction of $\psi$ to $X$ composed with the projection 
$p_2\:\wt\C^2\times_B\C\to\C$. If $\L_\psi$ is $\psi$-admissible then
\begin{itemize}
\item[2.] $\psi_*\L_\psi$ defines a map to 
$\ol{\mathcal J}$ if and only if $\mu_{X*}(\L_\psi|_X)$ is 
$C_1$-quasistable with respect to $\ul e$ for every fiber 
$X$ over a distinguished point of $\phi^{-1}(R,S)$ for $R,S\in\N(C)$.
\end{itemize}
\end{Lem}

\begin{proof} The first statement follows from Lemma~\ref{propblows2},
  which says that the fibers of $\rho$ over closed points of $\wt\C^2$ 
are isomorphic to $C$, $C(1)$ or $C(2)$. 
There is a stratification $\wt\C^2=U_{-1}\cup U_0\cup U_1\cup U_2$ 
by locally closed subschemes such that the fibers of 
$\rho$ over $U_i$ are isomorphic to 
$C(i)$ for $i=0,1,2$, and nonsingular for $i=-1$. 
(Of course, $U_{-1}$ is the fiber of $\wt\C^2$ over the generic point of $B$.) 
By Proposition \ref{propblows2}, 
the subscheme $U_2$ is simply the collection of distinguished points on 
the fibers of $\phi$ over pairs of reducible nodes of $C$. 

By 
preservation of degree, the degrees of $\L_\psi|_X$ on the exceptional curves 
of $\mu_X\:X\to C$ are the same for every fiber $X$ of $\rho$ over a point on 
a connected component of $U_1$. Now, every connected component of $U_1$ has a 
point of $U_2$ in the boundary. And the exceptional curves of the fibers 
of $\rho$ over points on this connected component degenerate to either an 
exceptional curve or an exceptional chain in the corresponding fiber over 
the boundary point on $U_2$. Thus, again by preservation of degree, if 
$\L_\psi|_X$ is $\mu_X$-admissible 
for the fibers $X$ over points on $U_2$, then so is 
$\L_\psi|_X$ for the fibers $X$ over points on $U_1$, and thus so 
is $\L_\psi$ with respect to $\psi$. This proves the ``if'' part of 
Statement 1. The ``only if'' part is trivial.

Suppose now that $\L_\psi$ is $\psi$-admissible. By 
Proposition \ref{famchain}, $\psi_*\L_{\psi}$ is a relatively 
torsion-free, rank-1 sheaf on $\wt\C^2\times_B\C/\wt\C^2$ of relative 
degree $f$, with formation commuting with base change. 
We prove Statement 2. Since 
$$
(\psi_*\L_{\psi})|_{p_1^{-1}(A)}\cong\mu_{\rho^{-1}(A)*}
(\L_{\psi}|_{\rho^{-1}(A)})
$$
for every $A\in\wt\C^2$, only its ``if'' part is nontrivial. So assume 
$\mu_{X*}(\L_\psi|_X)$ is $C_1$-quasistable with respect to $\ul e$ for every 
fiber $X$ over a point on $U_2$. Since quasistability is an open 
property by \cite{E01}, Prop.~34, p.~3071,
$\mu_{X*}(\L_\psi|_X)$ is $C_1$-quasistable with respect to $\ul e$ 
for every fiber $X$ over a closed point of $\wt\C^2$ on a neighborhood of 
$U_2$. Any such neighborhood intersects all connected 
components of $U_0$ and of $U_1$. But quasistability is a numerical condition, 
and since, by preservation of degree, the multidegree of 
$\L_\psi|_X$ is constant as $X$ varies as a fiber on each connected component 
of $U_0$ or $U_1$, if $\mu_{X*}(\L_\psi|_X)$ is 
$C_1$-quasistable with respect to $\ul e$ for some such fiber $X$, then 
so is $\mu_{X*}(\L_\psi|_X)$ 
for any such fiber $X$. So $\psi_*\L$ defines 
a map to $\ol{\mathcal J}$.
\end{proof}

\begin{Def}\label{admfun} 
Let $(\Gamma,\ul e,\ul q,v)$ be degree-2 Abel data. Let $V$ 
be the set of vertices and $E$ the set of edges of $\Gamma$. A 
\emph{resolution} of the degree-2 Abel data is a map
$$
\ul r\: E^2\to V^2
$$
that assigns to each pair of edges $(e_1,e_2)$ of $\Gamma$ a pair of vertices 
$(v_1,v_2)$ where $v_1$ is an end of $e_1$ and $v_2$ is an end of $e_2$, and 
takes the diagonal of $E^2$ to the diagonal of $V^2$.
 
Two resolutions $\ul r_1$ and $\ul r_2$ are said to be
\emph{equivalent on} $F\subseteq E^2$ if for each pair of edges
$(e_1,e_2)\in F$, either $v_{1,1}=v_{2,1}$ and 
$v_{1,2}=v_{2,2}$ or $v_{1,1}\neq v_{2,1}$ and $v_{1,2}\neq v_{2,2}$, 
where $\ul r_i(e_1,e_2)=(v_{i,1},v_{i,2})$ for $i=1,2$. They are
simply said to be \emph{equivalent} if they are equivalent on $E^2$. Given a resolution 
$\ul r$, an equivalent resolution $\check{\ul r}$ is obtained by the
condition that $\ul r(e_1,e_2)_i$ and $\check{\ul r}(e_1,e_2)_i$ 
be the distinct vertices of $e_i$ for each $(e_1,e_2)\in E^2$ 
and $i=1,2$. We call $\check{\ul r}$ the 
\emph{mirror resolution} of $\ul r$. 

Let $\ul\delta$ be the correction function of $(\Gamma,\ul e,\ul q,v)$. 
Let $\ul r$ be a resolution of the Abel data. We call $\ul r$ 
\emph{admissible at} $(e_1,e_2)\in E^2$ if, 
letting $(v_1,v_2)=\ul r(e_1,e_2)$ and $(w_1,w_2)=\check{\ul r}(e_1,e_2)$, 
the following three conditions hold:
\begin{enumerate}
\item For each edge $e$ distinct from $e_1$ and $e_2$, letting $m$ and $n$ be 
the ends of $e$, the following inequalities hold:
\begin{align*}
&\big|\ul\delta(v_1,v_2,m,n)-\ul\delta(w_1,v_2,m,n)\big|\leq 1,\\
&\big|\ul\delta(v_1,v_2,m,n)-\ul\delta(v_1,w_2,m,n)\big|\leq 1,\\
&\big|\ul\delta(w_1,v_2,m,n)-\ul\delta(v_1,w_2,m,n)\big|\leq 1,\\
&\big|\ul\delta(w_1,w_2,m,n)-\ul\delta(w_1,v_2,m,n)\big|\leq 1,\\
&\big|\ul\delta(w_1,w_2,m,n)-\ul\delta(v_1,w_2,m,n)\big|\leq 1.
\end{align*}
\item If $e_1\neq e_2$ then 
\begin{align*}
&\big|\ul\delta(v_1,v_2,v_1,w_1)-\ul\delta(w_1,v_2,v_1,w_1)-1\big|\leq 1,\\
&\big|\ul\delta(v_1,v_2,v_1,w_1)-\ul\delta(v_1,w_2,v_1,w_1)\big|\leq 1,\\
&\big|\ul\delta(w_1,v_2,v_1,w_1)-\ul\delta(v_1,w_2,v_1,w_1)+1\big|\leq 1,\\
&\big|\ul\delta(w_1,w_2,v_1,w_1)-\ul\delta(w_1,v_2,v_1,w_1)\big|\leq 1,\\
&\big|\ul\delta(w_1,w_2,v_1,w_1)-\ul\delta(v_1,w_2,v_1,w_1)+1\big|\leq 1,\\
&\big|\ul\delta(v_1,w_2,v_2,w_2)-\ul\delta(v_1,v_2,v_2,w_2)+1\big|\leq 1,\\
&\big|\ul\delta(v_1,v_2,v_2,w_2)-\ul\delta(w_1,v_2,v_2,w_2)\big|\leq 1,\\
&\big|\ul\delta(w_1,v_2,v_2,w_2)-\ul\delta(v_1,w_2,v_2,w_2)-1\big|\leq 1,\\
&\big|\ul\delta(w_1,w_2,v_2,w_2)-\ul\delta(w_1,v_2,v_2,w_2)+1\big|\leq 1,\\
&\big|\ul\delta(v_1,w_2,v_2,w_2)-\ul\delta(w_1,w_2,v_2,w_2)\big|\leq 1.
\end{align*}
\item If $e_1=e_2$ then 
\begin{align*}
&\big|\ul\delta(v_1,w_1,v_1,w_1)-\ul\delta(v_1,v_1,v_1,w_1)+1\big|\leq 1,\\
&\big|\ul\delta(v_1,w_1,v_1,w_1)-\ul\delta(w_1,w_1,v_1,w_1)-1\big|\leq 1,
\end{align*}
(Notice that in this case $(v_1,w_1)=(v_2,w_2)$.)
\end{enumerate}
We say that $\ul r$ is
\emph{admissible} if $\ul r$ is admissible at every $(e_1,e_2)\in
E^2$. Notice that $\ul r$ is admissible if and only if any other resolution
equivalent to $\ul r$ is admissible. 
\end{Def}

\begin{Thm}\label{isadmissible} 
Let $\Gamma$ be the essential dual graph of $C$. Let $V$ be its 
set of vertices and $E$ its set of edges. Let $\ul e\:V\to\mathbb Q$ and 
$\ul q\:V\to\mathbb Z$ be defined by setting $\ul e(C_i):=e_i$ and 
$\ul q(C_i):=\deg(\P|_{C_i})$ for each $i=1,\dots,p$. Set $v:=C_1$. 
Let $\phi\:\wt\C^2\to\C^2$ be a good partial desingularization. Let 
$\ul r\:E^2\to V^2$ be a function ``defined'' by sending each pair of 
reducible nodes $(R,S)$ to $(C_i,C_k)$, where 
$(R,S)\in C_i\times C_k$ and 
the strict transform to $\wt\C^2$ of $C_i\times C_k$ does not 
contain $\phi^{-1}(R,S)$. Then:
\begin{itemize}
\item[1.] $\ul r$ is a resolution of $(\Gamma,\ul e,\ul q,v)$. Any other 
function satisfying the same condition ``defining'' $\ul r$ is 
equivalent to it.
\end{itemize}
Let $\psi\:\wt\C^3\to\wt\C^2\times_B\C$ be a good  
partial desingularization and $\rho\:\wt\C^3\to\wt\C^2$ the
composition of $\psi$ with the first projection
$p_1\:\wt\C^2\times_B\C\to\wt\C^2$.  
\begin{itemize}
\item[2.] For each $R$, $S$ reducible nodes of $C$, the sheaf 
$\L_\psi$ is $\psi$-admissible at the two distinguished points of 
$\phi^{-1}(R,S)$ if and only if 
$\ul r$ is admissible at $(R,S)$.
\item[3.] $\L_\psi$ is $\psi$-admissible if and only if $\ul r$ is admissible.
\end{itemize}
\end{Thm}

\begin{proof} 
Since $\phi$ is a composition of blowups, the first along the 
diagonal, the strict transform to $\wt\C^2$ of $C_j\times C_j$ does not 
contain $\phi^{-1}(R,R)$ for any $R\in\N(C)$. So $\ul r$ is a resolution. 
Furthermore, let $(R,S)\in\N(C)^2$ and $G:=\phi^{-1}(R,S)$. 
Let $C_i$ and $C_j$ be the two 
components containing $R$ and $C_k$ and $C_l$ those containing $S$. Then the 
strict transform to $\wt\C^2$ of $C_i\times C_k$ does not contain 
$G$ if and only if that of $C_j\times C_l$ does not contain 
$G$ if and only if that of $C_j\times C_k$ contains $G$ if and only if 
that of $C_i\times C_l$ contains $G$. So, any other function 
satisfying 
the same condition ``defining'' $\ul r$ is equivalent to it.

We prove Statement 2 now. Let $\ul\delta$ be the 
correction function of $(\Gamma,\ul e,\ul q,v)$. 
Let $(R,S)$ be a pair of reducible nodes 
of $C$. Let $C_i$ and $C_j$ be the two 
components containing $R$, and $C_k$ and $C_l$ those containing $S$. 
Assume that $\ul r(R,S)=(C_i,C_k)$. Let $A_1$ and $A_2$ be the 
distinguished points of $\phi^{-1}(R,S)$. Assume that 
$A_1$ is the point of intersection of the strict transforms to $\wt\C^2$ of 
$C_i\times C_k$, $C_i\times C_l$ and $C_j\times C_k$, whereas 
$A_2$ is the point of intersection of the strict transforms to $\wt\C^2$ of
$C_j\times C_l$, $C_i\times C_l$ and $C_j\times C_k$. 

The components of $C$ will be viewed as components of $X_1$ and $X_2$. 
As for the remaining components, let $T$ be a reducible node of $C$. 
Let $C_m$ and $C_n$ be the components 
of $C$ containing $T$. For each $a=1,2$, let $E_{a,T,m}$ (resp. $E_{a,T,n}$) 
be the irreducible component of $\psi^{-1}(A_a,T)$ intersecting $C_m$ (resp. 
$C_n$). We need to understand when
\begin{equation}\label{TnRnS}
\begin{aligned}
&|\deg(\L_\psi|_{E_{a,T,m}})|\leq 1,\\
&|\deg(\L_\psi|_{E_{a,T,n}})|\leq 1,\\
&|\deg(\L_\psi|_{E_{a,T,m}})+\deg(\L_\psi|_{E_{a,T,n}})|\leq 1
\end{aligned}
\end{equation}
for each $T\in\N(C)$ and $a=1,2$.

Assume first that $T\neq R$ and $T\neq S$. Then the degree of $\L_\psi$ 
on $E_{a,T,m}$ or $E_{a,T,n}$ is the same as the degree of 
$$
\M:=\O_{\wt\C^3}(-\Z_{(i,k)}-\Z_{(i,l)}-\Z_{(j,k)}-\Z_{(j,l)}).
$$
By Proposition \ref{propblows2a}, independently of $\psi$, Inequalities 
\eqref{TnRnS} 
are satisfied for $a=1$ if and only if
\begin{align*}
&\big|\delta_{(i,k)}(m,n)-\delta_{(j,k)}(m,n)\big|\leq 1,\\
&\big|\delta_{(i,k)}(m,n)-\delta_{(i,l)}(m,n)\big|\leq 1,\\
&\big|\delta_{(j,k)}(m,n)-\delta_{(i,l)}(m,n)\big|\leq 1,
\end{align*}
and are satisfied for $a=2$ if and only if
\begin{align*}
&\big|\delta_{(j,l)}(m,n)-\delta_{(j,k)}(m,n)\big|\leq 1,\\
&\big|\delta_{(j,l)}(m,n)-\delta_{(i,l)}(m,n)\big|\leq 1,\\
&\big|\delta_{(j,k)}(m,n)-\delta_{(i,l)}(m,n)\big|\leq 1. 
\end{align*}

Assume now that $T=R$ but $T\neq S$. Use again Proposition \ref{propblows2a}. 
There are several cases to be checked but all of them yield that, 
independently of $\psi$, Inequalities \eqref{TnRnS} 
are satisfied for $a=1$ if and only if
\begin{align*}
&\big|\delta_{(i,k)}(i,j)-\delta_{(j,k)}(i,j)-1\big|\leq 1,\\
&\big|\delta_{(i,k)}(i,j)-\delta_{(i,l)}(i,j)\big|\leq 1,\\
&\big|\delta_{(j,k)}(i,j)-\delta_{(i,l)}(i,j)+1\big|\leq 1,
\end{align*}
and are satisfied for $a=2$ if and only if
\begin{align*}
&\big|\delta_{(j,l)}(i,j)-\delta_{(j,k)}(i,j)\big|\leq 1,\\
&\big|\delta_{(j,l)}(i,j)-\delta_{(i,l)}(i,j)+1\big|\leq 1,\\
&\big|\delta_{(j,k)}(i,j)-\delta_{(i,l)}(i,j)+1\big|\leq 1. 
\end{align*}

Assume $T=S$ but $T\neq R$. Use Proposition \ref{propblows2a}. 
Again, several cases need to be checked but all of them yield that, 
independently of $\psi$, Inequalities \eqref{TnRnS} 
are satisfied for $a=1$ if and only if
\begin{align*}
&\big|\delta_{(i,l)}(k,l)-\delta_{(i,k)}(k,l)+1\big|\leq 1,\\
&\big|\delta_{(j,k)}(k,l)-\delta_{(i,l)}(k,l)-1\big|\leq 1,\\
&\big|\delta_{(i,k)}(k,l)-\delta_{(j,k)}(k,l)\big|\leq 1,
\end{align*}
and are satisfied for $a=2$ if and only if
\begin{align*}
&\big|\delta_{(j,k)}(k,l)-\delta_{(i,l)}(k,l)-1\big|\leq 1,\\
&\big|\delta_{(j,l)}(k,l)-\delta_{(j,k)}(k,l)+1\big|\leq 1,\\
&\big|\delta_{(i,l)}(k,l)-\delta_{(j,l)}(k,l)\big|\leq 1. 
\end{align*}

Finally, assume that $T=S=R$. Using Proposition \ref{propblows2a}, 
independently of $\psi$, Inequalities \eqref{TnRnS} 
are satisfied for $a=1$ if and only if
$$
\big|\delta_{(i,j)}(i,j)-\delta_{(i,i)}(i,j)+1\big|\leq 1,
$$
and are satisfied for $a=2$ if and only if
$$
\big|\delta_{(i,j)}(i,j)-\delta_{(j,j)}(i,j)-1\big|\leq 1.
$$

Comparing the above inequalities with those in Definition \ref{admfun}, we see 
that $\L_\psi$ is $\psi$-admissible at $A_1$ and $A_2$ if and only if 
$\ul r$ is admissible at $(R,S)$.

Statement 3 follows from Statement 2 and Lemma \ref{lemblowindep}.
\end{proof}

\begin{Prop}\label{blowindep}
Let $\phi\:\wt\C^2\to\C^2$ 
be a good partial desingularization. Let 
$\psi_1\:\wt\C^3\to\wt\C^2\times_B\C$ and 
$\psi_2\:\wh\C^3\to\wt\C^2\times_B\C$ be good partial 
desingularizations. Then, if $\L_{\psi_1}$ is admissible, so is 
$\L_{\psi_2}$. If this is the case, and $\psi_{1*}\L_{\psi_1}$ defines a map to 
$\ol{\mathcal J}$, 
then so does $\psi_{2*}\L_{\psi_2}$, and the maps are equal.
\end{Prop}

\begin{proof} The first statement follows from Theorem \ref{isadmissible}, 
as $\ul r$ depends only on $\phi$. Suppose $\L_{\psi_1}$ and thus 
$\L_{\psi_2}$ are admissible. By Proposition \ref{famchain}, both 
$\psi_{1*}\L_{\psi_1}$ and $\psi_{2*}\L_{\psi_2}$ are relatively 
torsion-free, rank-1 sheaves on $\wt\C^2\times_B\C/\wt\C^2$ of relative 
degree $f$, with formation commuting with base change.

Let $\rho_i\:\wt\C^3\to\wt\C^2$ 
be the composition of $\psi_i$ with the 
projection $p_1\:\wt\C^2\times_B\C\to\wt\C^2$, for $i=1,2$. 
Suppose $\psi_{1*}\L_{\psi_1}$ defines a map to 
$\ol{\mathcal J}$. For each $(R,S)\in\N(C)^2$ and each distinguished point 
$A\in\phi^{-1}(R,S)$, let $X_i(A)$ be the fiber of $\rho_i$ over $A$ and 
$\mu_i(A)\:X_i(A)\to C$ the restriction to $X_i(A)$ 
of $\psi_i$ composed with the projection $p_2\:\wt\C^2\times_B\C\to\C$, 
for $i=1,2$. By Lemma \ref{lemblowindep}, the sheaf 
$\psi_{1*}\L_{\psi_1}$ defines a map to 
$\ol{\mathcal J}$ if and only if $\mu_1(A)_*(\L_{\psi_1}|_{X_1(A)})$ is 
$C_1$-quasistable with respect to $\ul e$ for every $A$ as above.
Now, $X_1(A)$ and $X_2(A)$ are $C$-isomorphic. 
Identify $X_1(A)$ with $X_2(A)$, calling both $X$. 
Then, it follows from Proposition \ref{propblows2a} 
that $\L_{\psi_1}|_X$ and $\L_{\psi_2}|_X$ 
differ by a twister ``supported'' on the exceptional components of $X$. 
Thus, by Proposition \ref{compadm}, we have that
$\mu_1(A)_*(\L_{\psi_1}|_X)\cong\mu_2(A)_*(\L_{\psi_2}|_X)$, and hence 
$\mu_2(A)_*(\L_{\psi_2}|_{X_2(A)})$ is 
$C_1$-quasistable with respect to $\ul e$ for each $(R,S)\in\N(C)^2$ and 
each distinguished point 
$A\in\phi^{-1}(R,S)$. Using Lemma \ref{lemblowindep} again, it follows 
that $\psi_{2*}\L_{\psi_2}$ defines a map to $\ol{\mathcal J}$. That the maps 
defined by $\psi_{1*}\L_{\psi_1}$ and $\psi_{2*}\L_{\psi_2}$ are equal is a 
consequence of the fact that they are equal on the fiber of 
$\wt\C^2$ over the generic point of $B$, which is dense 
in $\wt\C^2$.
\end{proof}

\subsection{Quasistability}

\begin{Def}\label{admfun2} 
Let $\Gamma$ be a connected graph without loops, 
with set of vertices $V$ 
and set of edges $E$. Let $i$ be a nonnegative integer. Let 
$\Gamma(i)$ be the graph obtained from $\Gamma$ by 
replacing each edge with a directed graph of $i+1$ edges. Then also 
$\Gamma(i)$ is connected without loops. 
Let $V(i)$ and $E(i)$ denote 
the sets of vertices and edges of $\Gamma(i)$, respectively.
Notice that $V$ may be viewed as a subset of $V(i)$; call 
the vertices of $V(i)-V$ \emph{exceptional}.

Recall that, for each $v\in V(i)$, 
$$
\underline c_v\:V(i)\to\mathbb Z
$$
is defined by letting $\ul c_v(w)$ be the number of edges of 
$\Gamma(i)$ with ends $v$ and $w$, for $w\neq v$, and 
$$
\ul c_v(v):=-\sum_{w\neq v}\ul c_v(w).
$$
Two functions $\ul d_1,\ul d_2\:V(i)\to\mathbb Z$ are said to be 
\emph{$\Gamma$-equivalent} if
$$
\ul d_1-\ul d_2=\sum_{v\in V(i)-V}a_v\ul c_v
$$
for some integers $a_v$. A function $\ul d\:V(i)\to\mathbb Z$ is
always 
$\Gamma$-equivalent to 
a unique function $\check{\ul d}\:V(i)\to\mathbb Z$ such that, for 
each edge $e\in E$, if $v_1,\dots,v_i$ are the exceptional vertices 
of $\Gamma(i)$ created 
on $e$, then $\check{\ul d}(v_j)=0$ for all $j=1,\dots,i$ with the 
possible exception of a single $j$ for which $\check{\ul d}(v_j)=-1$. Call 
$\check{\ul d}$ the \emph{reduction} of $\ul d$.

Let $(\Gamma,\ul e,\ul q,v)$ be 
degree-2 Abel data. Then there are natural degree-2 Abel data 
$(\Gamma(i),\ul e^i,\ul q^i,v)$, where $\ul e^i$ (resp. $\ul q^i$) 
restricts to $\ul e$ (resp. $\ul q$) on the vertices of $\Gamma$ and assigns 
0 to each exceptional vertex of $\Gamma(i)$. 

Assume $i=2$. Let $\ul\delta\:V^4\to\mathbb Z$ be the correction function of 
$(\Gamma,\ul e,\ul q,v)$. Let $\ul r\:E^2\to V^2$ be a resolution 
of the Abel data. Let $\check{\ul r}$ denote its mirror resolution. 
We associate to $\ul r$ two functions 
$\ul s_1,\ul s_2\:E^2\times V(2)\to\mathbb Z$, defined on a triple 
$(e_1,e_2,w)$ as follows:
\begin{enumerate}
\item If $w\in V$, let $\ul s_1(e_1,e_2,w):=\ul s_2(e_1,e_2,w):=\ul q(w)$. 
\item If $w\in V(2)-V$, let $e\in E$ be the edge of $\Gamma$ 
on which $w$ was created. 
Let $m$ be the end of $e$ to which $w$ is connected by an edge of $\Gamma(2)$, 
and let $n$ denote the other end. Set $(v_1,v_2):=\ul r(e_1,e_2)$ 
and $(w_1,w_2):=\check{\ul r}(e_1,e_2)$. Then put
\begin{align*}
\ul s_1(e_1,e_2,w):=&\ul\delta(v_1,v_2,m,n)+\ul\delta(v_1,w_2,m,n)+
\ul\delta(w_1,v_2,m,n)-\epsilon_1,\\
\ul s_2(e_1,e_2,w):=&\ul\delta(w_1,w_2,m,n)+\ul\delta(w_1,v_2,m,n)+
\ul\delta(v_1,w_2,m,n)-\epsilon_2,
\end{align*}
where
\begin{align*}
\epsilon_1:=&\#\{i\in\{1,2\}\,|\,v_i=m\text{ and }e_i=e\},\\
\epsilon_2:=&\#\{i\in\{1,2\}\,|\,w_i=m\text{ and }e_i=e\}.
\end{align*}
\end{enumerate}

We say that $\ul r$ is \emph{quasistable at} $(e_1,e_2)\in E^2$ if $\ul r$ is
admissible at $(e_1,e_2)$ and the functions 
$\ul s_1(e_1,e_2,-)$ and $\ul s_2(e_1,e_2,-)$ have $v$-quasistable reductions 
with respect to $\ul e^2$.  And we say that $\ul r$ is
\emph{quasistable} if $\ul r$ is quasistable at every $(e_1,e_2)\in E^2$.
\end{Def}

\begin{Thm}\label{isquasistable} 
Let $\Gamma$ be the essential dual graph of $C$. Let $V$ be its 
set of vertices and $E$ its set of edges. Let $\ul e\:V\to\mathbb Q$ and 
$\ul q\:V\to\mathbb Z$ be defined by setting $\ul e(C_i)=e_i$ and 
$\ul q(C_i):=\deg(\P|_{C_i})$ for every $i=1,\dots,p$. Set $v:=C_1$. 
Let $\phi\:\wt\C^2\to\C^2$ be a good partial desingularization. Let 
$\ul r\:E^2\to V^2$ be a function ``defined'' by sending each pair of 
reducible nodes $(R,S)$ to $(C_i,C_k)$, where 
$(R,S)\in C_i\times C_k$ and 
the strict transform to $\wt\C^2$ of $C_i\times C_k$ does not 
contain $\phi^{-1}(R,S)$. 
Let $\psi\:\wt\C^3\to\wt\C^2\times_B\C$ be a good  
partial desingularization. Then:
\begin{enumerate}
\item For each pair $(R,S)$ of reducible nodes of $C$, the resolution $\ul r$ is 
quasistable at $(R,S)$ if and only if there is an open neighborhood in
$\wt\C^2$ of the two
distinguished points of $\phi^{-1}(R,S)$ over which 
$\psi_*\L_\psi$ is a relatively torsion-free, rank-1 sheaf on
$\wt\C^2\times_B\C/\wt\C^2$ with formation commuting with base change and
defining a map to $\ol{\mathcal J}$.
\item The resolution $\ul r$ is quasistable  if and only if 
$\psi_*\L_\psi$ is a relatively torsion-free, rank-1 sheaf on
$\wt\C^2\times_B\C/\wt\C^2$ with formation commuting with base change and
defining a map $\ol{\mathcal J}$.
\end{enumerate}
\end{Thm}

\begin{proof} Let $R$ and $S$ be reducible nodes of $C$. Let $A_1$ and
  $A_2$ be the two distinguished points of $\phi^{-1}(R,S)$. Let 
$\rho\:\wt\C^3\to\wt\C^2$ be the composition of $\psi$ with the 
projection $p_1\:\wt\C^2\times_B\C\to\wt\C^2$. For $a=1,2$, let 
$X_a:=\rho^{-1}(A_a)$ and let $\mu_a\:X_a\to C$ be the 
restriction of $\psi$ to $X_a$ composed with the second 
projection $p_2\:\wt\C^2\times_B\C\to\C$. 

We prove the first statement. By Theorem \ref{isadmissible}, the resolution $\ul r$ is 
admissible at $(R,S)$ if and only if $\L_\psi$ is $\psi$-admissible at
$A_1$ and $A_2$, thus if and only if 
$\L_\psi$ is $\psi$-admissible over a neighborhood $U$ in $\wt\C^2$ of
$A_1$ and $A_2$, by Proposition~\ref{famchain}, so if and only if 
$\psi_*\L_\psi$ is a relatively torsion-free, rank-1 sheaf on
$\wt\C^2\times_B\C/\wt\C^2$ with formation commuting with base change
over $U$. It remains to show that $\ul r$ is quasistable at $(R,S)$ if
and only if $\mu_{a*}\L_\psi|_{X_a}$ is $C_1$-quasistable with respect
to $\ul e$ for $a=1,2$.

For each $a=1,2$, the direct image 
$\mu_{a*}\L_\psi|_{X_a}$ is $C_1$-quasistable if and only if there is a 
twister $\T$ ``supported'' on the exceptional components of $X_a$ such that 
$\L_\psi|_{X_a}\ox\T$ is $C_1$-quasistable. 
Indeed, as observed in Definition \ref{admfun2}, there is a twister 
$\T$ ``supported'' on the exceptional components of $X_a$ such that 
$\L_\psi|_{X_a}\ox\T$ is strongly $\mu_a$-admissible. 
Since also $\L_\psi|_{X_a}$ is 
$\mu_a$-admissible, $\mu_{a*}(\L_\psi|_{X_a}\ox\T)\cong\mu_{a*}(\L_\psi|_{X_a})$ by 
Proposition~\ref{compadm}. So $\mu_{a*}(\L_\psi|_{X_a})$ is $C_1$-quasistable 
if and only if $\mu_{a*}(\L_\psi|_{X_a}\ox\T)$ is $C_1$-quasistable, 
if and only if $\L_\psi|_{X_a}\ox\T$ is $C_1$-quasistable, by 
Proposition \ref{famchain}.

The first statement of the theorem follows now 
from the following claim: For each $a=1,2$, 
there is a 
twister $\T$ ``supported'' on the exceptional components of $X_a$ such that 
$\L_\psi|_{X_a}\ox\T$ is $C_1$-quasistable if and only if $\ul s_a(R,S,-)$ 
has $C_1$-quasistable reduction with respect to $\ul e^2$.

Let us prove the claim for  $a=1$ only, the other case being 
similar. Set $A:=A_1$, $X:=X_1$ and $\mu:=\mu_1$. Let $C_i$ and $C_j$ be the distinct  
components of $C$ containing $R$, and $C_k$ and $C_l$ those containing $S$.  
Assume $\ul r(R,S)=(C_i,C_k)$. We may assume that $A$  
lies on the strict transforms  
of $C_i\times C_k$, $C_j\times C_k$ and $C_i\times C_l$.
We will denote by $C_m$ the strict transform under $\mu$ of the component 
$C_m$ of $C$, for each $m$. For each reducible 
node $T$ of $C$, let $m_T$ and $n_T$ be the distinct integers such that 
$T\in C_{m_T}\cap C_{n_T}$ and let $E_{T,m_T}$ and $E_{T,n_T}$ be the exceptional 
components of $X$ over $T$, with $E_{T,m_T}$ intersecting $C_{m_T}$ and 
$E_{T,n_T}$ intersecting $C_{n_T}$. 

As in the statement of Proposition \ref{propblows2a}, let 
$\lambda\:B\to\wt\C^2$ be any section of $\wt\C^2/B$ sending the special 
point $0$ of $B$ to $A$ and such that the pullbacks of the strict transforms 
of $C_i\times C_k$, $C_i\times C_l$ and $C_j\times C_k$ are all equal to 0. 
Form the Cartesian diagram
$$
\begin{CD}
\W @>\xi >> \wt\C^3\\
@V\rho_\lambda VV @V\rho VV\\
B @>\lambda >> \wt\C^2.
\end{CD}
$$
The fiber of $\W/B$ over $0$ is isomorphic to $X$ under $\xi$ and will also 
be denoted by $X$. 

Recall that $\wt\Delta_1$ and $\wt\Delta_2$ denote 
the strict transforms of $\Delta_1$ 
and $\Delta_2$ to $\wt\C^3$. By Proposition \ref{propblows2a}, 
there are relative effective Cartier 
divisors $\Gamma_1$ and $\Gamma_2$ of $\W/B$ such that 
$\xi^*\wt\Delta_1-\Gamma_1$ and $\xi^*\wt\Delta_2-\Gamma_2$ are effective and 
supported on the exceptional components of $X$. Furthermore, $\Gamma_1$ and 
$\Gamma_2$ intersect $X$ transversally, the first at $E_{R,i}$ and the 
second at $E_{S,k}$.

Let
$$
\Z:=\sum_{i,k=1}^p\Z_{(i,k)}=\sum_{i,k=1}^p\sum_{m=1}^pw_{(i,k)}(m)D_{i,k,m},
$$
where $D_{i,k,m}$ is the strict transform to $\wt\C^3$ of 
$C_i\times C_k\times C_m$ for each $i,k,m$. By Proposition 
\ref{propblows2a}, 
$$
\xi^*\Z=\sum_{m=1}^pw(m)C_m+\sum_{T\in\N(C)}
\big(w(T,m_T)E_{T,m_T}+w(T,n_T)E_{T,n_T}\big)
$$
for certain integers $w(m)$ and $w(T,m)$, for $m=1,\dots,p$ and 
$T\in\N(C)$. More precisely, 
\begin{align*}
w(m_T)-w(T,m_T)=&\delta_2(T)\\
w(T,m_T)-w(T,n_T)=&\delta_3(T)\\
w(T,n_T)-w(n_T)=&\delta_1(T),
\end{align*}
where 
$\delta_1(T),\delta_2(T)$ and $\delta_3(T)$ are 
$$
\delta_{(i,k)}(m_T,n_T),\quad \delta_{(j,k)}(m_T,n_T),\quad \delta_{(i,l)}(m_T,n_T)
$$
in a particular order. 

Set
$$
N:=\xi^*\Z+\sum_{T\in\N(C)}\big(\delta_2(T)E_{T,m_T}-\delta_1(T)E_{T,n_T}\big)
$$
and put 
$$
\L'_\psi:=\xi^*\psi^*p_2^*\P\ox\O_\W(-\Gamma_1-\Gamma_2)\ox\O_\W(-N).
$$
Then $\xi^*\L_\psi$ and $\L'_\psi$ differ by a twister ``supported'' on 
the exceptional components of $X$. Also,
$$
\deg(\L'_{\psi}|_Y)=\ul s_1(R,S,Y)
$$
for every component $Y$ of $X$. So, $\ul s_1(R,S,-)$ has $C_1$-quasistable 
reduction with respect to $\ul e^2$ if and only if there is a twister 
$\T$ ``supported'' on the exceptional components of $X$ such that 
$\L'_\psi|_X\ox\T$ is $C_1$-quasistable, thus if and only if there is a 
twister $\T$ ``supported'' on the exceptional components of $X$ such that 
$\L_\psi|_X\ox\T$ is $C_1$-quasistable.

To prove the second statement of the theorem, notice that, 
by Lemma \ref{lemblowindep}, the sheaf $\psi_*\L_\psi$ defines a map to 
$\ol{\mathcal J}$ if and only if $\mu_{X*}(\L_\psi|_X)$ is $C_1$-quasistable 
with respect to $\ul e$ for every fiber $X$ over a distinguished point 
of $\phi^{-1}(R,S)$ for $R,S\in\N(C)$. Then apply the first statement
already proved.
\end{proof}

\begin{Def}\label{admfun3} Let $(\Gamma,\ul e,\ul q,v)$ be 
degree-2 Abel data. Let $V$ be the set of vertices and $E$ the set of edges 
of $\Gamma$. The \emph{singular locus} 
$\Sigma$ of the Abel data is the maximum subset of 
$E^2$ on which any two quasistable resolutions $\ul r_1$ and $\ul r_2$ are
equivalent. 
Notice that $\Sigma$ contains the diagonal $\Delta_E$ of $E^2$. If
quasistable resolutions exist, then we say that $\Sigma$ is \emph{solvable}.

A \emph{blowup sequence} for the Abel data is a pair $(I_1,I_2)$ of sequences 
$I_1=(I_{1,1},\dots,I_{1,u})$ and $I_2=(I_{2,1},\dots,I_{2,u})$ of
equal lengths of proper nonempty subsets of $V$. It is called
\emph{symmetric} if the subsets are symmetric to each other, that is, 
whenever $I_{1,j}\neq I_{2,j}$ we have
$(I_{1,j-1},I_{2,j-1})=(I_{2,j},I_{1,j})$ or
$(I_{1,j+1},I_{2,j+1})=(I_{2,j},I_{1,j})$. We say that a pair
$(R,S)\in E^2$ is \emph{affected} by the blowup sequence if there is $j$ such that 
one and only one end of $R$ lies in $I_{1,j}$ 
and one and only one end of $S$ lies in $I_{2,j}$. We call the minimum
such $j$ the \emph{order} of $(R,S)$ in the blowup sequence.
The \emph{center} of the
blowup sequence is the subset of $E^2$ consisting of the pairs $(R,S)\in E^2$ such that
either $R=S$ or $(R,S)$ is affected.

We say that a blowup sequence $(I_1,I_2)$ \emph{resolves} the Abel
data if the following three conditions are verified:
\begin{enumerate}
\item The center of the blowup sequence contains $\Sigma$.
\item $\Sigma$ is solvable
\item For each $(R,S)\in\Sigma-\Delta_E$, 
any quasistable resolution $\ul r$ satisfies 
$\ul r(R,S)\in I_{1,j}\times (V-I_{2,j})$ or 
$\ul r(R,S)\in (V-I_{1,j})\times I_{2,j}$, where $j$ is the order of $(R,S)$.
\end{enumerate}
We say that $(I_1,I_2)$ resolves the Abel data \emph{minimally} if its
center is $\Sigma$.
\end{Def}

\begin{Thm}\label{mainthm} 
Let $\Gamma$ be the essential dual graph of $C$. Let $V$ be its 
set of vertices and $E$ its set of edges. Let $\ul e\:V\to\mathbb Q$ and 
$\ul q\:V\to\mathbb Z$ be defined by setting $\ul e(C_i):=e_i$ and 
$\ul q(C_i):=\deg(\P|_{C_i})$ for each $i=1,\dots,p$. Set
$v:=C_1$. Let $\Sigma$ be the singular locus of the Abel data
$(\Gamma,\ul e,\ul q,v)$. Assume that $\Sigma$ is solvable. 
Then the degree-2 rational map 
$\alpha^2_{\C/B}\:\C^2\dashrightarrow\ol{\mathcal J}$ is defined on
$\C^2-\Sigma$. Furthermore,
let $(I_1,I_2)$ be a blowup sequence, and let $u$ denote its length.
Let $\phi\:\wh\C^2\to\C^2$ be the sequence of blowups
$$
[\Delta],[X_1,Y_1],\dots,[X_u,Y_u],
$$
where $X_i$ is the union of those $C_j$ in $I_{1,i}$ and $Y_i$ is the union of 
those $C_j$ in $I_{2,i}$, for each 
$i=1,\dots,u$. Let 
$\wt\phi\:\wt\C^2\to\C^2$ and $\psi\:\wt\C^3\to\wt\C^2\times_B\C$ be 
good partial desingularizations. 
Let $\ul r\:E^2\to V^2$ be a function ``defined'' by sending 
each pair of reducible nodes $(R,S)$ to $(C_i,C_k)$, 
where $(R,S)\in C_i\times C_k$ and the strict transform to 
$\wt\C^2$ of $C_i\times C_k$ does not 
contain $\wt\phi^{-1}(R,S)$. Then:
\begin{enumerate}
\item If $(I_1,I_2)$ resolves the Abel data. 
then there is a map 
$$
\wh{\alpha}^2_{\C/B}\:\wh\C^2\lra\ol{\mathcal J}
$$
resolving $\alpha^2_{\C/B}$. In addition, if $\wt\phi=\phi\phi'$ for 
$\phi'\:\wt\C^2\to\wh\C^2$, then $\ul r$ is quasistable, 
in which case $\psi_*\L_\psi$
defines a map $\wt\alpha^2_{\C/B}\:\wt\C^2\to\ol{\mathcal J}$ 
agreeing with $\wh{\alpha}^2_{\C/B}\phi'$.
\item Conversely, if $(I_1,I_2)$
resolves the Abel data minimally and $\ul r$ is quasistable, then  
$\wt\phi$ factors through $\wh\C^2$.
\end{enumerate}
\end{Thm}

\begin{proof} Assume that $\Sigma$ is solvable. Then there is a quasistable
  resolution $\ul r'$ of the Abel data. For each $(R,S)\in E^2$, 
let 
$$
\phi^{(R,S)}_1:=[i,k]\:\X\to\C^2\quad\text{and}\quad
\phi^{(R,S)}_2:=[i,l]\:\Y\to\C^2
$$
be the two different blowups,
where $\ul r'(R,S)=(C_i,C_k)$ and $\check{\ul r}'=(C_j,C_l)$. By Theorem
\ref{isquasistable}, the composition $\alpha^2_{\C/B}\phi_2^{(R,S)}$ is
defined on a neighborhood of the two distinguished points of $\Y$ over
$(R,S)$. Using preservation of the degree, as in the proof of 
Lemma~\ref{lemblowindep}, it follows that $\alpha^2_{\C/B}$ is defined
on $\C^2-E^2$.

Let now $(R,S)\in E^2-\Sigma$. Set $\phi_i:=\phi^{(R,S)}_i$ for
$i=1,2$. By definition of $\Sigma$, we have that $\check{\ul r}'$ is 
quasistable at $(R,S)$. Then, again by Theorem
\ref{isquasistable}, also $\alpha^2_{\C/B}\phi_1$ is
defined on a neighborhood of the two distinguished points of $\X$ over
$(R,S)$. It follows that both maps factor through a 
set-theoretic map $U\to\ol{\mathcal J}$ defined on a Zariski
neighborhood $U\subseteq\C^2$ of $(R,S)$ and agreeing with
$\alpha^2_{\C/B}$ away from $(R,S)$. 
Since $\phi_1$ (or $\phi_2$) is proper and surjective, the map 
$U\to\ol{\mathcal J}$ is continuous. To show it is a morphism of
schemes, we need only show now that 
$\phi_{1*}\O_\X=\O_{\C^2}$. 

We need only show the above equality at $(R,S)$. So we may work locally 
analytically. In this setup, we consider the closed subscheme 
$Z\subset\IA^4_K$ given by $x_0x_1=y_0y_1$ and the blowup 
$W\subset\IA^4_K\times\IP^1_K$ given by $\alpha'x_0=\alpha y_1$ and 
$\alpha'y_0=\alpha x_1$. Consider the projection 
$p_1\:\IA^4_K\times\IP^1_K\to\IA^4_K$. We need to show that $p_{1*}\O_W=\O_Z$. 

Consider the following diagram of exact sequences
\begin{equation}\label{bldiag}
\begin{CD}
0 @>>> \I_{Z|\IA^4_K} @>>> \O_{\IA^4_K} @>>> \O_Z @>>> 0\\
@. @VVV @VVV @VVV @.\\
0 @>>> p_{1*}\I_{W|\IA^4_K\times\IP^1_K} @>>> p_{1*}\O_{\IA^4_K\times\IP^1_K} @>>> 
p_{1*}\O_W @.\\
\end{CD}
\end{equation}
The middle vertical map is an isomorphism. Now, $\I_{W|\IA^4\times\IP^1}$ 
has the following presentation:
$$
0\lra\O_{\IA^4_K\times\IP^1_K}(-2) \lra \O_{\IA^4_K\times\IP^1_K}(-1)\oplus
\O_{\IA^4_K\times\IP^1_K}(-1)\lra \I_{W|\IA^4_K\times\IP^1_K} \lra 0.
$$
From the long exact sequence of higher direct images of $p_1$, we get 
$$
R^1p_{1*}\I_{W|\IA^4_K\times\IP^1_K}=0\quad\text{and}\quad
p_{1*}\I_{W|\IA^4_K\times\IP^1_K}\cong R^1p_{1*}\O_{\IA^4_K\times\IP^1_K}(-2)
\cong\O_{\IA^4_K}.
$$
So, the rightmost map in the bottom row of Diagram \eqref{bldiag} is 
surjective. Thus, since the middle vertical map is an
isomorphism, the comorphism $\O_Z\to p_{1*}\O_W$ is surjective. On the
other hand, it is clearly injective.

Assume now that $(I_1,I_2)$ resolves the Abel data. Then, for each
$(R,S)\in\Sigma$, we have that 
$\ul r'(R,S)=(C_i,C_k)$, where $(R,S)\in C_i\times C_k$ and the strict
transform of $C_i\times C_k$ under $\phi$ does not 
contain $\phi^{-1}(R,S)$. To prove Statement (1), we may assume that 
$\wt\phi$ factors through 
$\phi$. Then $\ul r'$ and $\ul r$ are equivalent on $\Sigma$. 
Furthermore, by definition of $\Sigma$, if $(R,S)\in E^2-\Sigma$, then there
is a quasistable resolution $\ul r''$ of the Abel data that is not
equivalent to $\ul r'$ at $(R,S)$. Thus, either $\ul r$ is equivalent
to $\ul r'$ or to $\ul r''$ at $(R,S)$. Since quasistability
is checked on each pair of reducible nodes, it follows that $\ul r$
is itself quasistable.

By Theorem~\ref{isquasistable}, since $\ul r$ is quasistable, 
$\psi_*\L_\psi$ defines a map 
$\wt\alpha^2_{\C/B}\:\wt\C^2\to\ol{\mathcal J}$. We need only prove that
$\wt\alpha^2_{\C/B}$ factors through $\wh\C^2$. This is clearly true
over the points of $\wh\C^2$ where $\phi'$ is an isomorphism, thus 
over $\phi^{-1}(\Sigma)$. This is also true away from $\Sigma$, 
since we have proved that $\alpha^2_{\C/B}$ is defined on
$\C^2-\Sigma$, and since 
$\alpha^2_{\C/B}\wt\phi$ agrees with $\wt\alpha^2_{\C/B}$ over the
generic point of $B$, so wherever the former is defined.

Finally, assume that $(I_1,I_2)$
resolves the Abel data minimally and that $\ul r$ is quasistable.
Statement (2) is local, so we need only check 
that $\wt\phi$ factors through $\phi$ at pairs $(R,S)$ 
of distinct reducible nodes of $C$. So, let $(R,S)$ be
such a pair. Let $C_i$ and $C_j$ be the components of $C$ containing
$R$ and $C_k$ and $C_l$ those of $S$. In the blowup sequence defining
$\phi$, let $[W,Z]$ be the first blowup in which $R\in W\cap W'$ and
$S\in Z\cap Z'$. Without loss of generality, we may assume that
$C_i\subseteq W$ and $C_l\subseteq Z$. Then, in a neighborhood of
$(R,S)$, the map $\wt\phi$ is equal both to the blowup $[i,l]$ and the
blowup $[j,k]$. Also $\ul r(R,S)=(i,k)$ or 
$\ul r(R,S)=(j,l)$. If $(R,S)\not\in\Sigma$, then $\phi$ is an isomorphism
at $(R,S)$, and thus $\wt\phi$ factors through $\phi$ at $(R,S)$. On
the other hand, if $(R,S)\in\Sigma$, then, since $\ul r$ is quasistable,
it follows the smallest integer $m$ such that 
$R\in X_m\cap X'_m$ and $S\in Y_m\cap Y'_m$ 
is such that $\ul r(R,S)\in I_{1,m}\times (V-I_{2,m})$ or 
$\ul r(R,S)\in (V-I_{1,m})\times I_{2,m}$. In any case, it follows that either
$$ 
C_i\subseteq X_m,\quad C_j\subseteq X'_m,\quad C_l\subseteq Y_m,\quad  
C_k\subseteq Y'_m
$$ 
or
$$
C_j\subseteq X_m,\quad C_i\subseteq X'_m,\quad C_k\subseteq Y_m,\quad 
C_l\subseteq Y'_m.
$$
In any case,  in a neighborhood of
$(R,S)$, the map $\phi$ is equal both to the blowup $[i,l]$ and the
blowup $[j,k]$, and thus equal to $\wt\phi$. 
\end{proof}

\section{Examples}\label{7}

There are {\cocoa} scripts to determine whether the singular
locus of given degree-2 Abel data 
is solvable and, if so, whether a symmetric blowup
sequence resolving minimally the Abel data exists. They are available
at 
$$
\text{\rm http://w3.impa.br/$\sim$esteves/CoCoAScripts/Abelmaps}.
$$
The scripts were applied to various
Abel data, and the output has always been positive. 

Throughout this section, $(\Gamma,\ul e,\ul q,v)$ will denote degree-2 
Abel data. Denote by $V$ the set of vertices of $\Gamma$.

First of all, the various notions associated to degree-2 Abel data
$(\Gamma,\ul e,\ul q,v)$ depend on the associated correction function 
$\ul\delta$, which in turn depends only on the difference 
$\ul q-\ul e$. Thus we may assume that $\ul q$ is any fixed function
$V\to\mathbb Z$. We will assume that
\begin{equation}\label{qfix}
\ul q(w)=
\begin{cases}
2&\text{if }w=v,\\
0&\text{if }w\neq v
\end{cases}
\end{equation}

In particular, $\ul e$ must have degree 0. 
A possible choice of $\ul e$ is the zero function. Though many
other choices are possible, it follows from the definition that the
correction function $\ul\delta$ of $(\Gamma,\ul e,\ul q,v)$ is the
same as that of $(\Gamma,\ul e+\sum_ia_i\ul c_i,\ul q,v)$, for any
choice of integers $a_1,\dots,a_p$, where the $\ul c_i$ are given in 
Definition~\ref{Abdata}. We may thus assume that $-\ul e$ is
$v$-quasistable with respect to any fixed polarization, say the zero
polarization. In other words, we may assume that
\begin{equation}\label{upol}
-\frac{k_I}{2}\leq\sum_{i\in I}\ul e(i)<\frac{k_I}{2}
\end{equation}
for each proper nonempty subset $I\subset V$ containing $v$.

Now, on $\mathbb R^p$, with coordinates $x_1,\dots,x_p$, consider the
subspace $H_p\subset\mathbb R^p$ given by
$$
x_1+\cdots+x_p=0.
$$
The affine subspaces of $H_p$ given by 
\begin{equation}\label{affsp}
\sum_{i\in I}x_i+\frac{k_I}{2}=a_I,
\end{equation}
where $I$ ranges through all proper subsets of $V$ containing
$v$ and the
$a_I$ through $\mathbb Z$, induce a stratification
$\Xi=\Xi(\Gamma)$ of $H_p$ by convex strata. A stratum is a connected
component of the intersection of certain subspaces of the form \eqref{affsp}
with the complements of certain subspaces of the same form.

Let $\Xi_0=\Xi_0(\Gamma)$ be the collection of strata of $\Xi$ whose points
$(x_1,\dots,x_p)$ satisfy
$$
-\frac{k_I}{2}\leq\sum_{i\in I}x_i<\frac{k_I}{2}
$$
for each proper nonempty subset $I\subset V$ containing $v$. Then
$\Xi_0$ is a finite collection. A polarization $\ul e$ of degree 0 satisfying \eqref{upol} 
for each proper nonempty subset $I\subset V$ containing $v$ belongs to
one of the strata in $\Xi_0$. Furthermore, if $\ul e$ and $\ul e'$ are
polarizations lying on the same stratum, the corresponding Abel data 
$(\Gamma,\ul q,\ul e, v)$ and $(\Gamma,\ul q,\ul e', v)$ have the same
correction function.

In other words, for a given connected graph without loops $\Gamma$, we
need only apply our Cocoa scripts finitely many times to show that every
Abel data $(\Gamma,\ul e,\ul q,v)$ supported in $\Gamma$ has solvable
singular locus and admits a symmetric blowup sequence resolving it
minimally.

In the following examples, Abel data $(\Gamma,\ul e,\ul q,v)$ 
are presented in the following way: 
The graph $\Gamma$ is given by its intersection 
matrix, $\Xi$, which is a square matrix whose rows (and columns) are in 
bijective correspondence with the set of vertices of $\Gamma$ and whose entry 
at position $(i,j)$ is the number of edges between the $i$-th and $j$-th 
vertices if $i\neq j$, and the negative of the number of edges with end at the 
$i$-th vertex if $i=j$. We will assume $v$ is the first vertex and $\ul q$ is 
fixed as above. The polarization $\ul e$ is given as a tuple.

\begin{Exa}\label{7.1} 
Consider the graph $\Gamma$ whose intersection matrix is 
$$
\left[\begin{matrix}
-2&1&1&0\\
1&-5&3&1\\
1&3&-6&-2\\
0&1&2&-3
\end{matrix}\right]
$$
Set $v:=1$ and $\ul q:=(2,0,0,0)$. If we set $\ul e:=(0,0,0,0)$, then the 
{\cocoa} scripts tell us that the Abel data has solvable
singular locus, and the blowup sequence 
$$
(\{1\},\{1\}),\, (\{4\},\{4\})
$$ 
resolves it minimally. We get the same result if we set 
$\ul e:=(0,-1/2,0,1/2)$. On the other hand, if we set $\ul e=(0,1/2,0,-1/2)$, 
the Abel data has solvable singular locus, but the blowup sequence that 
resolves it minimally is
$$
(\{1\},\{1\}).
$$
So, as expected, the resolution depends on the polarization.
\end{Exa}

\begin{Exa} In \cite{CP}, Abel maps for curves of compact type are studied. 
Abel maps of every degree are constructed, starting from the degree-1 Abel 
maps constructed in \cite{CE} and \cite{CCE}, using the fact that the 
generalized Jacobian of a curve of compact type is projective. Arguably, 
the simplest curves not of compact type are the ``circular curves,'' 
whose dual graph has the following intersection matrix:
$$
\left[\begin{matrix}
-2&1&0&0&\dots&0&1\\
1&-2&1&0&\dots&0&0\\
0&1&-2&1&\dots&0&0\\
\vdots&\vdots&\vdots&\vdots&\ddots&\vdots&\vdots\\
0&0&0&0&\dots&1&0\\
0&0&0&0&\dots&-2&1\\
1&0&0&0&\dots&1&-2
\end{matrix}\right].
$$
As before, we set $v:=1$ and $\ul q:=(2,0,\dots,0)$. 
To simplify, set $\ul e:=(0,\dots,0)$. 

The simplest case is that of 
graphs with two vertices only. Then the Abel data is solvable and 
$$
(\{1\},\{1\}).
$$
resolves it minimally. For three vertices, again the Abel data is solvable, 
and 
$$
(\{1\},\{1\}),\,(\{2\},\{2\}),\,(\{3\},\{3\})
$$
resolves it minimally. The case of four vertices is more interesting: a 
minimal resolution is
$$
(\{1\},\{1\}),\,(\{2\},\{2\}),\,(\{1,2\},\{1,2\}),\,(\{3\},\{3\}),\,
(\{2,3\},\{2,3\}),\,(\{1,3\},\{1,3\}).
$$
Finally, the case of five vertices may perhaps indicate the pattern: a 
minimal resolution is
\begin{align*}
&(\{1\},\{1\}),\,(\{2\},\{2\}),\,(\{1,2\},\{1,2\}),\,(\{3\},\{3\}),\,
(\{2,3\},\{2,3\}),\,(\{1,3\},\{1,3\}),\\
&(\{4\},\{4\}),\,(\{3,4\},\{3,4\}),\,
(\{2,4\},\{2,4\}),\,(\{1,4\},\{1,4\}).
\end{align*}
\end{Exa}

\begin{Exa} Assuming $\ul q$ is of the form 
\eqref{qfix}, if we assume the polarization $\ul e $ satisfies
Inequalitites \eqref{upol}, 
heuristically, the further the inequalities are from being equalities, the 
smallest a minimal resolution of the Abel data is. 
And, indeed, if
$$
2-\frac{k_I}{2}\leq\sum_{i\in I}\ul e(i)<\frac{k_I}{2}
$$
for each proper nonempty subset $I\subset V$ containing $v$, the empty 
sequence resolves the Abel data minimally. So, for instance, consider 
the graph $\Gamma$ whose intersection matrix is 
$$
\left[\begin{matrix}
-4&2&2&0\\
2&-7&3&2\\
2&3&-7&2\\
0&2&2&-4
\end{matrix}\right],
$$
a small variation of that in Example \ref{7.1}
Set $v:=1$ and $\ul q:=(2,0,0,0)$. If we set $\ul e:=(0,0,0,0)$, then the 
empty sequence resolves the Abel data minimally, whereas if we set 
$\ul e:=(-1/2,0,1/2,0)$, then 
$$
(\{1\},\{1\}).
$$
resolves it minimally, and if $\ul e:=(-1,-1,1,1)$, then 
$$
(\{1\},\{1\}),\, (\{4\},\{4\})
$$ 
resolves it minimally.
\end{Exa}

\begin{Rem} The reader might have observed that in all the blowup sequences 
$(I_1,I_2)$ above we had $I_{1,j}=I_{2,j}$ for every $j$. Indeed, the
{\cocoa} scripts 
are written in such a way that those special symmetric blowup sequences are 
preferred over the others.
\end{Rem}

\vskip0.5cm 

{\smallsc Universidade Federal Fluminense, Instituto de Matem\'{a}tica,
Rua M\'{a}rio Santos Braga, s/n, Valonguinho, 24020--005 Niter\'{o}i RJ, 
Brazil}

{\smallsl E-mail address: \small\verb?coelho@impa.br?}

\vskip0.5cm

{\smallsc Instituto Nacional de Matem\'atica Pura e Aplicada, 
Estrada Dona Castorina 110, 22460--320 Rio de Janeiro RJ, Brazil}

{\smallsl E-mail address: \small\verb?esteves@impa.br?}

\vskip0.5cm

{\smallsc Universidade Federal Fluminense, Instituto de Matem\'{a}tica,
Rua M\'{a}rio Santos Braga, s/n, Valonguinho, 24020--005 Niter\'{o}i RJ, 
Brazil}

{\smallsl E-mail address: \small\verb?pacini@impa.br?}

\end{document}